\documentclass[graybox,envcountchap,sectrefs]{svmono}
\usepackage{amsmath,amssymb,amsbsy,eucal,epsfig,algorithm}
\usepackage{algpseudocode}
\usepackage[utf8x]{inputenc}

\usepackage{type1cm}         

\usepackage{graphicx}        %
\usepackage{multicol}        %
\usepackage[bottom]{footmisc}%

\usepackage{newtxtext}       %
\usepackage{newtxmath}       %

\newcommand{\mingap}{\lambda_{\min}}

\let\emptyset\varnothing
\let\epsilon\varepsilon
\let\phi\varphi

\def\B{{\mathcal B}}
\def\X{{\bf X}}

\def\y{{\bf y}}
\def\z{{\bf z}}
\def\x{{\bf x}}
\def\C{\mathcal C}
\def\S{\mathcal S}
\def\H{\mathcal H}
\def\O{{\cal O}}

\def\G{{\mathcal G}}
\def\bH{\mathbf H}

\newcommand{\argmaxdisp}[1]{\underset{#1}{\argmax}}

\def\I{{{}^{{s}}}\hskip-2ptI}

\def\argmin{\operatorname{argmin}}
\def\argmax{\operatorname{argmax}}

\newcommand{\cl}[1]{{\operatorname{\scriptstyle\texttt{closure}}}(#1)}

\def\argmin{\operatorname{argmin}}

\def\E{\mathbb E}

\def\N{\mathbb N}
\def\R{\mathbb R}
\def\as{\text{ a.s.}}

\title{Asymptotic nonparametric statistical analysis of stationary time series}
\author{Daniil Ryabko}
\date{}

\begin{document}

 \maketitle
\preface
This book is about making statistical inference from stationary discrete-time processes. The assumption of stationarity alone is often considered too weak to make any meaningful inference.  Here  this view is challenged by showing that, while some rather basic problems indeed can be proven not to admit any solution in this setting, surprisingly many are solvable  without any further assumptions. These includes such complex problems as clustering and change-point analysis. Some general results characterizing those  problems that admit a solution are also presented.

The material in this volume is presented in a way that presumes familiarity with basic concepts of probability and statistics, up to and including probability distributions over spaces of infinite sequences. All the required background material can be found in the excellent monograph~\cite{Gray:88}, which also contains a much deeper exposition of some of the key concepts used here, such as the distributional distance. Familiarity with ergodic theory is {\em  not required} for understanding the material exposed in the present volume. Indeed, with two exceptions, the proofs do not rely on any facts deeper than the convergence of frequencies. 
One exception is  Chapter~\ref{ch:ht}, which deals with hypothesis testing and provides a characterisation of hypotheses for which consistent tests exist; the required background material for this chapter can be found in Chapter~\ref{ch:pre}.
The other exception is Section~\ref{s:hom}, which establishes impossibility of discrimination between process distributions; this section is self-contained. 
 The reader who is familiar with ergodic theory and feels the exposition in this volume is somewhat unorthodox, can find all the necessary links to the more familiar  framework in \cite{Shields:96}; the latter book is also recommended to anyone seeking a deeper understanding of such results as the slow convergence of frequencies and entropy estimates,  the classic ergodic theorem and much more.

This book is organized as follows. Chapter~\ref{ch:intro} is introductory: besides providing some motivation for studying the problems addressed, it also introduces in an informal manner the main concepts used and the main results presented. Chapter~\ref{ch:pre} introduces the notation and definitions used in the subsequent chapters, as well as some necessary background material. Chapter~\ref{ch:basic} considers the most basic problems of statistical inference, on which the rest of the volume builds: estimating a distance between processes (the distributional distance) and the problem of homogeneity testing or process discrimination, which, crucially for the subsequent problems  addressed,  is shown to be impossible to solve in the general setting of this book. Chapter~\ref{ch:clchp}  is devoted to  clustering and change-point problems, which can be solved, or, in some cases, can be shown to admit no solution, based on the result of the preceding chapter. Chapter~\ref{ch:ht} addresses the problems of hypotheses testing in the general form: studying which pairs of hypotheses admit a consistent test. Finally, Chapter~\ref{ch:conc} discusses various generalizations of the presented results, as well as some directions for future research.

\subsubsection*{Acknowledgements}
\vskip-2mm
{\small Thanks to L\'eon Bottou for giving me the idea to write a book on this subject and for encouraging me to do it.  Thanks to  Boris Ryabko and Azadeh Khaleghi, in collaboration with whom some of the results presented here were obtained.
}

\vskip15mm
\noindent Santa Cruz de la Sierra \hfill Daniil Ryabko

\tableofcontents

\chapter{Introduction}\label{ch:intro}

This book is about making statistical inference from discrete-time processes under what is perhaps the weakest of statistical assumptions: stationarity.
Before embarking on this journey, it is worth asking the question of why it is interesting to study statistical problems under this   assumption    alone, or under similar related assumptions.
To answer this question, one should first consider what it means to have a  good  set of assumptions, or a good model, for a statistical problem at hand.

Choosing the right assumptions presents the following trade-off. On the one hand, making strong assumptions makes the inference task easier and allows one to obtain stronger performance guarantees for the algorithms developed. For example, by assuming that the data are independent and identically distributed (i.i.d.), one gets at one's disposal an extremely versatile  statistical toolkit that is a result of centuries of research on this model. With this, it is possible  to obtain sharp bounds on error probabilities of the resulting methods. Even stronger results can be obtained if one further makes parametric assumptions.   On the other hand, all such results are useless if the assumptions made do not hold for the data at hand.  Of course, one can try to apply a statistical test to the data in order to verify the validity of one or another model. This, however, only pushes back the problem, because to use a test one needs to make another  set of assumptions, called the alternative. Indeed, it is not possible to test, based on data, that the assumption $H_0$ holds versus it does not hold. For example, it is not possible to test that  the data are Gaussian i.i.d.\ versus the distribution of the data is anything else except Gaussian i.i.d. This is because the alternative ``anything else'' is too general and  includes, for example, such distributions as the one that is concentrated precisely on the data available. It is, however, possible to design a test for the hypothesis ``the data are Gaussian i.i.d.'' versus ``the data are i.i.d.\ but not Gaussian'' or ``the data are i.i.d.'' versus ``the distribution of the data is stationary.''  In other words, it may be possible to test a set of assumptions $H_0$ versus an alternative set of assumptions $H_1$. The latter is typically much more general; in fact, one is interested in making it as general as possible. Nonetheless, the alternative hypothesis  is still a set of assumptions.  

And so we are back to the question of how one can  select a model or a set of assumptions for the data one has. Here we need to admit that  this question brings us outside of the realm of mathematics. The answer is simply that one should make assumptions that one can reasonably expect to hold based on the specifics of the target application.  Thus, the assumptions should be qualitative, natural and simple~--- utterly unmathematical terms, but such is the problem. Otherwise, there is little hope to be able to say whether the model is adequate for any given application.  A good example are assumptions based on independence. Indeed, this must be one of  the reasons why independent and identically distributed data are so widely studied: it is often possible to tell whether the application produces data that are independent or that are not independent.  Other models that are based on independence are Markov chains and, more generally, Bayesian networks. 

 Unfortunately, there are not many alternatives to independence-based models. Thus, if the data are utterly and completely dependent, as perhaps are most of the data in the world, a statistician is a bit short of options.   A common generalisation to resort to in such cases are various mixing assumptions. These allow one to extend the tools and methods developed for i.i.d.\ data to the cases of carefully constrained dependence. However, mixing assumptions are neither verifiable against a general alternative (such as stationarity) nor, to say the least, are easy to asses   informally from the data. 

Stationarity is perhaps the only general non-parametric model that is not based on independence, and which is also qualitative, natural and simple to assess from data. Next we take a brief and informal look at stationarity and associated concepts.

\section{Stationarity, ergodicity, AMS}
Very informally, assuming that the data are {\em stationary} means assuming that the time index itself bears no information. Thus, it does not matter whether the data we see are $X_0,X_1,\dots$ or they are in fact $X_{100},X_{101},\dots$. I.i.d.\ data obviously satisfy this assumption, as do, with some minor tweaks to be discussed below, most other models in wide use, such as Markov chains. Thus, stationarity may be used as an alternative hypothesis for testing other models. It is also suited for the cases when one knows next to nothing about the data, and thus wishes to make as few assumptions as possible.  
In fact, the assumption is so general that one wonders whether any inference is possible under stationarity alone. 
Indeed, if any inference is possible at all it is due the the associated property of ergodicity. 

A process is {\em ergodic} if the frequency of every finite-time event almost surely  converges to a constant. Thus, for binary-valued processes, the frequency of any word, such as 0, 01, or 011010, converges to some constant. We cannot say anything about the speed of this convergence, but the asymptotic property is already enough to make inference.   The  {\em ergodic decomposition} theorem establishes that every stationary process is a mixture of processes that are stationary and ergodic. Thus, a stationary process can be thought of as, first, before we start observing the data, drawing a stationary ergodic process (according to some prior distribution over such processes) and then using this stationary ergodic process to generate the data.  To  put it  simpler:  whenever we observe a stationary process, we observe, in fact, a stationary ergodic process. Thus, for most practical as well as many theoretical considerations, a stationary process  {\em is} a stationary ergodic process. 

Note that an ergodic process does not have to be stationary. A good example of an ergodic non-stationary process is a finite-state connected Markov chain with an initial distribution on the states that is different from the stationary distribution. Asymptotically, this process is equivalent to the Markov chain with the initial distribution  taken to be the stationary distribution. One can take mixtures of ergodic process, obtaining  processes that are called {\em asymptotically mean stationary} or AMS. An AMS process is such that  the frequencies of all finite-time events converge almost surely (but not necessarily to a constant). Since the definition of an ergodic process only involves its asymptotic properties, all the inference one can make about such processes concern their asymptotic behaviour. In this (asymptotic) sense, similar to stationary processes, an AMS process can be thought of, very roughly, as first drawing an ergodic process (according to some prior distribution over such processes) and then using this ergodic process to generate the data.  Again, for most purposes AMS processes are ergodic processes. In turn, ergodic processes are a certain generalisation of stationary ergodic processes: as in the Markov-chain example, they are equivalent in asymptotic. Another example of  can think of is taking a realization of a stationary ergodic process and adding some arbitrary prefix to it; or doing to it anything else  that does not affect  asymptotic frequencies. 

It is worth emphasizing that, with the exception of stationarity itself,  all the definitions we are using only tell us something about asymptotic properties of a process; moreover, for the purposes of statistical inference that we shall be exploring,  stationarity can only be used in conjunction with or via ergodicity (which, fortunately, can always be presumed via the ergodic decomposition theorem mentioned). Therefore, any results we should expect shall also be  about asymptotic properties of the algorithms that we shall construct. 

Thus,  there will be little difference for us in the course of this volume between ergodic processes and stationary ergodic processes, and between stationary processes and AMS processes.  
In fact, most of the results of this book do not require any other assumption than AMS or ergodicity.  Thus, they can be thought of as answering the question: 

\begin{svgraybox}
{\bf Stationarity-based statistical inference:} What statistical inference can one make under the only assumption that frequencies converge, without any guarantees on the speed of this convergence?
\end{svgraybox}

One exception is Chapter~\ref{ch:ht}, where we do need our processes to be stationary, and use some deeper results of ergodic theory. The other exception is  the impossibility result  concerning process discrimination (along with its implications) which applies to an even smaller class of process; since it is an impossibility result, this makes it stronger.

The main difference between the problems of statistical inference addressed in this volume  and those studied in the vast majority of statistical literature is the lack of any guarantees on the speed of convergence that one can use. 

In contrast, independence-based methods rely heavily on concentration-of-measure results that are used to bound the speed of convergence and, consequently, open the possibility to obtain finite-time bounds on the error of the resulting algorithms. In fact, the (conditional) independence assumptions are typically not used directly but rather through concentration of measure results. Mixing assumptions provide a generalisation that allows one to forego independence but still use the corresponding speed of convergence guarantees. Thus, one can think of independence-based models and their generalizations as studying the following general question: 

\begin{svgraybox}{\bf Independence-based statistical inference:}
  What statistical inference can one make under the assumption that  frequencies converge and the speed of this convergence can be bounded?
\end{svgraybox}
We shall see in this book that the difference between these two general questions is smaller than one might think, but sometimes the contrast between what is possible and what is not possible to do without any speeds of convergence is rather striking and even counter-intuitive.

\section{What is possible and what is not possible to infer from stationary processes}
It appears that, with the exception of the problem of probability forecasting to be mentioned below in this section, the prevailing view in the literature is that assuming only that a process is stationary and ergodic is not enough to make statistical inference.  This view may stem in part from the rather influential 1990 paper \cite{Ornstein:90} by Ornstein and Weiss. This paper  is full of deep and insightful results about $B$-processes, which is a set of processes smaller than that of stationary ergodic processes, but is rather dismissive of the general case. In particular, it makes statements such as   ``{\em In general, one cannot hope to guess the long-term behaviour from finite information}'' (referring to the non-$B$ case); ``{\em If a totally ergodic process is not $B$, then it cannot be approximated arbitrarily well by $k$-step Markov processes}.''  The work  \cite{Ornstein:90} goes further in this direction when it considers the problem of discrimination between two processes. This problem, also know as homogeneity testing, consists in telling, given two finite samples whose length, in this setting, is allowed to grow to infinity, whether they were generated by the same or different process distributions. It is stated in  \cite{Ornstein:90}  that, outside of the class of $B$-processes, {\em even this simple ``yes-no''  question of ``same-different'' cannot be answered in an effective way}. However, the example  used to demonstrate this statement only shows that it is not possible to estimate a certain distance, called $\bar d$ distance, between stationary ergodic processes that are not $B$. This is a rather different statement, and a one made about a different problem: indeed, in order to answer the ``same-different'' question, one might try to estimate any other distance or, more generally, use any algorithm whatsoever. Thus, the statement made in \cite{Ornstein:90} about the problem of discrimination can be at most considered a conjecture. The distance is also crucial to understanding the previous statements made: it is not possible to approximate a stationary ergodic process with $k$-step Markov processes {\em in $\bar d$-distance}, or to construct any other estimate of such a process that would be asymptotically consistent in terms of this distance. 

The picture changes dramatically if we change the distance between processes that we are trying to estimate. As we are going to see in this volume, using a different distance, it is possible to construct asymptotically consistent estimates of the distribution of an arbitrary stationary ergodic process, as well as to solve a variety of other interesting statistical problems. The distance we are going to use is well known, but had somehow remained largely unused. Gray \cite{Gray:88} calls it {\em distributional distance}, and this is the name we shall use here, despite its apparent ambiguity: indeed, it may seem to refer to any distance between distributions. 
As for the problem of discrimination between process distributions, it turns out that indeed,  as conjectured by  Ornstein and Weiss \cite{Ornstein:90}, it does not admit a solution if we only assume that the distributions are stationary ergodic. Interestingly, the same impossibility result holds for the smaller class of $B$ processes as well, for which it {\em is} possible to estimate the $\bar d$ distance, as shown in the same work \cite{Ornstein:90}. Thus, no amount of data may be sufficient to answer the simple ``same-different'' question about two process distributions. This result is formally demonstrated  in Section~\ref{s:hom}. 

Since these two problems, distance estimation and discrimination between processes, are crucial for the development of the material presented here, let us look at them at some more detail. 

Recall that one distance (or a metric~--- all distances considered in this volume are metrics unless stated otherwise) is {\em weaker}, in the topological sense, than another, if  every sequence\footnote{Here we are only concerned with separable metric spaces.} that converges in the former  converges in the latter, but the opposite does not necessarily hold. Thus, it is ``easier'' for a sequence to converge in a weaker distance, which makes it easier to construct a sequence of estimates of a process that converges to this process. %
 Likewise, given two data sequences, a weaker distance between the process distributions that generates these sequences is easier to estimate.
The distributional distance is weaker, in the topological sense, than the $\bar d$-distance.\footnote{To make complete sense of this sentence, we would need to define the distances formally first, which is done in the next chapter. We shall see that the definition of the distributional distance is ambiguous: it depends on a set of parameters, changing which may change the resulting topology. However, it is possible to make this statement formally correct.}  It is thus reasonable to expect that the former can be estimated for a larger class of processes than the latter. Indeed, as is shown in this volume, the distributional distance can be estimated for stationary ergodic processes, while, as is shown in \cite{Ornstein:90},  $\bar d$-distance  can be estimated for the smaller set of $B$-processes  but not for stationary ergodic processes. The strongest possible distance is the discrete 0-1 distance, which takes the value 0 if and only if two distributions are the same and 1 otherwise.  It is this distance that we are trying to estimate when answering the ``same-different'' question of process discrimination. It thus should be of no surprise that it is not possible to estimate it even for $B$-processes, even though it is possible to estimate it for smaller classes, such as, for example, i.i.d.\ processes. For many different problems, however,  it is enough to have consistent estimates of at least {\em some} distance between process distributions, and thus it makes sense to prefer weaker distances, since this allows one to consider wider sets of processes. We shall review shortly which problems of inference can be solved using consistent estimates of distributional distance (or, indeed, of any distance between process distributions).

Taking a different look at the problem of process  discrimination, one can see that it is linked to another fundamental impossibility result~--- the impossibility to establish the speed of convergence, say, of frequencies. The way we have defined ergodic processes, as all processes for which frequencies converge a.s.\ to a constant, makes it evident that this convergence may be arbitrary slow, so there is no guarantee on the speed. It is not so evident that such a guarantee does not exist  if we consider the set of all stationary ergodic processes (that is, adding the requirement of stationarity). The proof of the fact that indeed the convergence of frequencies can be arbitrary slow for stationary ergodic processes can be found, for example, in the excellent monograph \cite{Shields:96}, which also demonstrates the equivalence of the (unorthodox) definition that we adopt here to the more common one formulated in terms of shift-invariant sets.  Imagine now an algorithm that tries to solve the discrimination problem based on (consistent) estimates of some distance. It makes these estimates based on sampels of longer and longer size $n$. Suppose that these estimates keep approaching 0, let us say, exponentially with $n$. At some point one should reasonably expect the algorithm to say that the samples were generated by the same distribution. Suppose the estimated distance at this point is $\epsilon$. From this point on,  imagine that, as the sample size $n$ continues to grow, the estimate does not decrease at all but just stays $\epsilon$. Then, at some point, we should expect the algorithm to change its mind and to say that the samples were generated by the same distribution. At which point the estimates start decreasing again.  Since there is no guarantee on the speed of convergence (of anything), there is no way to ensure that the behaviour outlined cannot happen. In fact, the proof of the impossibility result is based on constructing, for any algorithm that presumably solves the problem of discrimination (and that may or may not be based on distance estimates), a process that tricks it into changing its mind {\em ad infinitum} in this fashion.

More generally, from the discussion above on the absence of speed of convergence guarantees, it should already be clear that:

\begin{svgraybox}
  Every algorithm that we may construct shall only have asymptotic performance guarantees in the considered setting. No finite-time bounds on the probability of error are possible.
\end{svgraybox}
From the practical point of view this is not in itself a hindrance: what  the fact that a result is asymptotic means, in practice, is that it holds when the data samples are large enough.

The only exception, where we do obtain results about what happens at every time step, is hypothesis testing. Here one  may wish to invert the question, by asking for which processes distributions can we have a certain level of error at a certain finite time. These questions are considered in Chapter~\ref{ch:ht}.

Having outlined the general framework and the main impossibility results, let us now briefly review the highlights of what is possible to achieve for stationary or stationary and ergodic processes. 

Perhaps the one important problem concerning stationary processes that has  not been deemed too difficult to solve and thus gained a fair bit of attention in the literature is the problem of {\em prediction} or {\em probability forecasting}. It consists in forecasting the probability of the next outcome $X_{n+1}$ conditional on the past observations $X_1,\dots,X_n$, where the sequence $X_1,\dots,X_n,X_{n+1},\dots$ is generated by an unknown stationary (ergodic) process distribution. This problem is of great practical importance, not in the least because it is intimately connected to the problem of data compression. Ample literature on this problem and its variations exist, which is why we do not cover it in this volume. This literature goes as far back as  \cite{Ornstein:78} for the prediction with the growing past problem, and includes \cite{BRyabko:88} that solves the forward-prediction problem for finite-alphabet processes, \cite{Algoet:92} for real-valued processes, as well as \cite{BRyabko:09,Morvai:96,Morvai:97,BRyabko:16} and others.

The problems covered in this volume are outlined in the  next section.

\section{Overview of the inference problems covered}
The first group of problems considered are those that are based directly on estimating a distance between process distributions.
Since we have an asymptotically consistent estimator of the distributional distance, we can answer questions of the form: given three samples $\x=(X_1,\dots,X_n)$, $\y=(Y_1,\dots,Y_m)$, $\z=(Z_1,\dots,Z_l)$, say whether the distribution of the process that generates $\z$ is closer to the distribution of $\x$ or to the one of  $\y$. The answer will be correct as long as the samples are long enough (that is, asymptotically correct). Some forms of this problem are known as {\em process classification} or the {\em three-sample problem}, and this is an example of a problem that we can solve. It generalizes to the problem of {\em clustering}: given   $N$ samples generated by $k$ different, unknown, stationary ergodic distributions, cluster them into $k$ groups according to the distribution that generates them.  Note that this problem can only be solved if $k$ is known. Indeed, the problem of discrimination corresponds to clustering just two samples, but with $k$ unknown (either 1 or 2), and already this case, as we have seen, has no solution. 

The next  problem to consider is {\em change-point estimation}. A sample $$X_1,\dots,X_{k},X_{k+1},\dots,X_n$$ is the concatenation of two samples $X_1,\dots,X_{k}$ and $X_{k+1},\dots,X_n$  generated by different stationary ergodic distributions. It is required to find or to approximate the change point~$k$. This is possible to do with an algorithm that essentially outputs the point that maximizes the estimated distance between what is before and after it in the sample. On the other hand, the related problem of {\em change-point detection}, which consists in saying whether the sample is generated by the same distribution or there is a change of distribution somewhere, admits no solution.  A generalisation of these problems to the case of  multiple change points presents a delicate interplay between what is possible and what is not.  We only briefly review the corresponding results in this volume (Section~\ref{s:chp}), referring the interested reader to the papers that present the full proofs \cite{Khaleghi:14,Khaleghi:15chp,Khaleghi:12mchp}.

As discussed above, one of the main reasons to study such general models as stationarity is to be able to use them as an alternative hypothesis in order to verify the validity of a smaller model. Thus, one may wish to test a  hypothesis $H_0$, which is a subset of the set of all stationary ergodic process distributions,  against its complement to this set, or against its different subset. For example, testing $H_0=$``the process is i.i.d.'' versus $H_1=$``the process is stationary ergodic and not i.i.d.''  As we have seen above, some rather simple hypotheses, such as process discrimination (known in the context of hypotheses testing as the hypothesis of homogeneity:  a hypothesis about a pair of processes that states that they have the same distribution) do not admit a consistent test, even in a very week asymptotic sense. Yet, as we shall see, some other hypotheses of practical significance, such as that the process is i.i.d.\ or that it is Markov, do admit a consistent test against the complement to the set of all stationary ergodic processes.  Thus, it appears interesting to study the general question of which hypotheses do and which do not admit a consistent test.  This is what we do in Chapter~\ref{ch:ht}. The main result is a topological ``if and only if'' criterion for the existence of a consistent test of an arbitrary subset of the set of all stationary ergodic processes against its complement. At the same time, a number of important and interesting questions remain open. In particular, this is the only chapter where we restrict the consideration to finite-alphabet processes, leaving the general case open for further research. Some of the interesting open problems related to hypotheses testing are presented in the end of Chapter~\ref{ch:ht}, while some more general ones are deferred to Chapter~\ref{ch:conc}, which is devoted to generalizations.

\chapter{Preliminaries}\label{ch:pre}
To simplify the exposition, we are considering (stationary ergodic) processes with the {\em alphabet} $A=\R$ or, in some cases, a finite set $A$.
The generalization from $A=\R$  to $A=\R^d$ is straightforward; moreover, the results can be extended
to the case when $A$ is a Polish (complete separable metric) space.
 The symbol $A^*$ is used for $\cup_{i=1}^\infty A^i$.
 Elements of $A^*$ are called words or sequences.

Let $\B_n$ be the Borel sigma-algebra of $A^n$, and $\B_\infty$ the the Borel sigma-algebra of~$A^\infty$. Let also $\B=\cup_{n=1}^\infty\B_n$.

Time-series distributions, {\em processes distributions} or simply {\em processes} are probability measures on $(A^\infty,\B_\infty)$.

We will be speaking about {\em samples}, typically denoted $\x,\y$ or $\z$, taking values in $A^*$. This is a short-hand notation for expressions like $X_{1..n}=(X_1,\dots,X_n)$ 
or $Y_{1..k}=(Y_1,\dots,Y_k)$ where $n=|\x|$ and $k=|\y|$ are lengths of the samples. The samples that we shall be considering are to be generated by process distributions, usually stationary or stationary ergodic, typically denoted $\rho$ or  $\rho_\x$, $\rho_\y$ (or other Greek letters) to make clear which sample they generate. This means that, say,  $\rho_\x$ is a (stationary ergodic) probability distribution over $(A^\infty,\B_\infty)$, and thus we are speaking about an $A^\infty$-valued random variable $(X_1,\dots,X_n,\dots)$ of which $X=(X_1,\dots,X_n)$ is the initial segment of length $n$.

\begin{definition}
 For a sequence $\x=X_{1..n}$ taking values in $A^n$ and a measurable $B\subset A^k$ with $k\in\N$ denote $\nu(\x,B)$
the frequency with which the sequence $\x$ falls in the set $B$
\begin{equation}\label{eq:freq}
\nu(\x,B):=\left\{ \begin{array}{rl}  {1\over n-k+1}\sum_{i=1}^{n-k+1}
I_{\{(X_i,\dots,X_{i+k-1})\in B\}} & \text{ if }n\ge k, \\
0 & \text{ otherwise.}\end{array}\right.
\end{equation}
\end{definition}

For example, $$\nu\big((0.5, 1.5, 1.2, 1.4, 2.1),([1.0,2.0]\times[1.0,2.0])\big)=1/2.$$

\section{Stationarity, ergodicity}

A process $\rho$ is {\em stationary}
if for any $i,j\in 1..n$ and $B \in \mathcal B$, 
we have $$\rho(X_{1..j} \in B)=\rho(X_{i..i+j-1} \in B).$$
A process $\rho$ is called {\em ergodic} if for every $B\in\mathcal B$ there exists a constant $v_B$ such that
with probability~1 we have 
$$\lim_{n\rightarrow\infty}\nu(X_{1..n},B) = v_B.$$ 
A process is called {\em stationary ergodic} if it is stationary and ergodic. 
The following statement follows from the ergodic theorem.
\begin{theorem}[ergodic theorem]
 For every stationary ergodic process $\rho$, we have
$$\lim_{n\rightarrow\infty}\nu(X_{1..n},B) = \rho(B)\as$$ 
\end{theorem}
The proof of the ergodic theorem can be found, for example, in \cite{Gray:88,Shields:96}. The latter monograph also provides the connection to the more traditional way of defining ergodicity (in terms of shift-invariant sets); in particular, it demonstrates that the two approaches are equivalent.

The symbol  $\S$ is used for  the set of all stationary  processes on $A^\infty$,
 and the symbol $\mathcal E$ for the set  of all stationary ergodic processes.

The set of all process distributions $\mathcal P$ over $A^\infty$ can be endowed with the structure of probability space $(\mathcal P,\mathcal B_{\mathcal P})$ where $\mathcal B_{\mathcal P}$ can be taken to be the Borel sigma-algebra with respect to the distributional distance defined in Section~\ref{s:dd} below.

The link between stationary and stationary ergodic processes is provided by the so-called ergodic decomposition theorem, which states that every stationary process is a mixture of stationary ergodic processes. 
\begin{theorem}[Ergodic decomposition]
For any $\rho\in\S$ there is a measure $W_\rho$ on $(\mathcal P,\mathcal B_{\mathcal P})$, such 
that  %
$%
W_\rho(\mathcal E)=1 
$%
 \ \  and  $$\rho(B)=\int d W_\rho(\mu)\mu(B)$$
 for every $B\in\mathcal \B.$%
\end{theorem}

Furthermore, a process  is called {\em asymptotically mean stationary}, or {\em AMS} for short, if, for every $B\in\mathcal B$, the frequency of $B$ converges with probability~1. These limiting frequencies define the stationary measure $\bar \rho$, which, according to the preceding theorem, admits an ergodic decomposition. Asymptotically, $\rho$ and $\bar \rho$ are equivalent, and thus there will be little distinction between the two for us in this volume. For a detailed exposition of these results the reader is referred to \cite{Gray:88}, in particular to \cite[Theorem~7.4.1]{Gray:88}  that establishes ergodic decomposition for AMS processes.

\section{Distributional distance}\label{s:dd}

The general definition of the distributional distance is as follows.
\begin{definition}[distributional distance] 
Let $(B_k)_{k\in \N}$ be a set of finite-time events  each $B_k\in\mathcal B$, $k\in\N$ that generates $\mathcal B_\infty$, and let $(w_k)_{k\in\N}$ be a\ sequence of positive reals such that $\sum_{k\in\N}w_k=1$. For a pair of processes $\rho_1,\rho_2$ the   distributional distance $d(\rho_1,\rho_2)$ is defined  as
\begin{equation}\label{eq:ddis}
d(\rho_1,\rho_2):=\sum_{i=1}^\infty w_i |\rho_1(B_i)-\rho_2(B_i)|.
\end{equation}
\end{definition}

Note that there are two sets of parameters in this definition, $B_k$ and $w_k$, which we shall now make more specified.
Let us  first fix 
\begin{equation}\label{wk}
  w_k:=1/k(k+1).
\end{equation}
The choice of the sets $B_k$ is more significant. Different choices may result in  different topologies. 
In particular, some choices of $B_k$ make the set of all process distributions $\mathcal P$ compact with the topology of the distributional distance $d$. 
This is the case if the set $B_k, k\in\N$ is a {\em standard basis} of $\B_\infty$.   While there is a standard basis for $\B_\infty$ in the case of $A=\R$, unfortunately, as Gray \cite{Gray:88} notes, there is no easy construction for such a basis even for the space of reals $(\R,\B)$. In this volume, we shall not make much use of the notion of standard basis, but it will be important for us to have empirical estimates of the distributional distance. Therefore, we shall fix a specific choice of the sets $B_k$ for the case of discrete alphabets and for the case $A=\R$ (which is easily generalisable to $A=\R^d$, $d\in\N$); we shall also make the definition of the distributional distance more specific reflecting these choices.

\begin{definition}[Distributional distance for finitely-valued processes]
 Let the alphabet $A$ be finite. Define 
\begin{equation}\label{eq:ddisd}
d(\rho_1,\rho_2):=\sum_{k=1}^\infty w_k \sum_{B\in A^k} |\rho_1(B)-\rho_2(B)|.
\end{equation}
\end{definition}
While equivalent to the general one, this more-specified  formulation is  better suited for constructing practical algorithms: we are taking the differences in probabilities of each word of length $k$, and then take a weighted sum over all $k\in\N$. 

For real-valued processes, we shall fix the usual set of cylinders to put in the distributional distance.
Consider the  sets $B^{m,l}, m,l \in \N$ which are obtained via the partitioning of $A^m$ into  cubes  
of dimension $m$ and volume $2^{-ml}$, starting at the origin, and enumerated clockwise in each direction. 
\begin{definition}[Distributional distance for real-valued processes] \label{defn:dd}
Let $A=\R$. Define 
\begin{equation}\label{eq:ddisr}
d(\rho_1,\rho_2)=\sum_{m,l=1}^\infty w_m w_l \sum_{B\in B^{m,l}} |\rho_1(B)-\rho_2(B)|.
\end{equation}
\end{definition}

The general formulation~\eqref{eq:ddis} is more compact and thus more convenient for the theoretical analysis; we shall therefore use it in the proofs, while still assuming the concrete choice of the parameters $w_k$ and $B_k$ whenever necessary. The more specific formulations~\eqref{eq:ddisd} and~\eqref{eq:ddisr} are more convenient for constructing algorithms and empirical estimates.

Note, however,  that the   definition~\eqref{eq:ddisr} is not exactly equivalent to the general definition~\eqref{eq:ddis}. Indeed, each of the sets $B^{m,l}$ is infinite, and all the individual sets inside of $B^{m,l}$ are assigned the same weight $w_mw_l$. This is not a problem, since the total $\rho_1$- as well as $\rho_2$- probability of all the sets in $B^{m,l}$ is 1. Indeed, it is a simple exercise to check that the proofs in the subsequent chapters go through for either of the definitions. We therefore take the liberty to use the definition~\eqref{eq:ddis} in the proofs, but refer to the more-specified definition~\eqref{eq:ddisr} when speaking about the algorithms. The unconvinced  reader may note that the sets $B^{m,l}$ can be made finite but growing with $l$, i.e., defined so as to cover growing parts of the space $A^m$ with finer partitions, leaving all the rest of the space as a single element of the space $B^{m,l}$. This way, the partitions become finite and the triple sum in~\eqref{eq:ddisr} can be converted back to the single sum in~\eqref{eq:ddis} with a different choice of the weights $w_k$.

It is easy to see that $d$ is a metric (with any choice of the parameters).
When talking about closed and open subsets of $\S$ we assume the topology of~$d$.
With this  topology, the space $\mathcal P$ of process distributions is separable. The set $\S$ of stationary distributions is its closed subset. In addition, for the case of finite-valued alphabets, the sets $\mathcal P$ and $\S$ are complete and compact. (The general result \cite[Lemmas~8.2.1, 8.2.2]{Gray:88} says that $\mathcal P$ is complete and compact in case the generating set $(B_k)_{k\in\N}$ is standard; this is the case in our definition~\eqref{eq:ddisd} but not in~\eqref{eq:ddisr}.)
Proofs of these facts can be found in \cite{Gray:88}.

\chapter{Basic inference}\label{ch:basic}
In this chapter we consider some basic problems of statistical inference that underly the rest of the problems addressed in this volume. Namely, we shall see that the distributional distance can be estimated empirically, and consider some immediate implications of this fact. On the other hand, it is shown that there is no asymptotically consistent  solution to the problem of discrimination (homogeneity testing) for stationary ergodic processes. 

The main results of the chapter can be summarized as follows.
\begin{svgraybox}
\begin{itemize}
 \item The distributional distance between stationary ergodic processes can be estimated consistently.
 \item There is no consistent discrimination procedure for stationary ergodic processes: no matter how long the sequences are, it is not possible to say whether they were generated by the same or different distributions.
 \item Based on the estimates of the distributional distance, one can solve the three-sample problem: say which two of the given three samples were generated by the same distribution. 
\end{itemize}
\end{svgraybox}

\section{Estimating   the distance between processes and reconstructing a process}
The main building block of the approach presented in this book is the rather simple fact that the distributional distance can be estimated empirically, simply replacing unknown probabilities with frequencies. The resulting estimate is asymptotically consistent for arbitrary stationary ergodic processes. 

\begin{definition}[empirical distributional distance] For samples $\x,\y\in A^*$, define empirical distributional distance $\hat d(\x,\y)$ as
\begin{equation}\label{eq:edd1}
\hat d(\x,\y):=\sum_{i=1}^\infty w_i |\nu(\x,B_i)-\nu(\y,B_i)|.
\end{equation}
Similarly, we  can define the empirical distance when only one of the process measures is unknown:
\begin{equation}\label{eq:edd2}
\hat d(\x,\rho):=\sum_{i=1}^\infty w_i |\nu(\x,B_i)-\rho(B_i)|,
\end{equation}
where $\rho\in{\mathcal E}$  and $\x\in A^*$.
\end{definition}

The following  lemma establishes consistency of these estimates. 
\begin{lemma}\label{th:dd} Let two samples $\x=(X_1,\dots,X_k)$ and
$\y=(Y_1,\dots,Y_m)$ be generated by
stationary ergodic processes $\rho_\x$ and $\rho_\y$ respectively. Then
\begin{itemize}
\item[(i)] $\lim_{k,m\rightarrow\infty}\hat d(\x,\y)=d(\rho_\x,\rho_\y)\ \as$
\item[(ii)] $\lim_{k\rightarrow\infty}\hat d(\x,\rho_\y)=d(\rho_\x,\rho_\y)\ \as$
\end{itemize}
\end{lemma}
\begin{proof}
For any $\epsilon>0$ we can find such an index $J$ that
$\sum_{i=J}^\infty w_i<\epsilon/2$.
Moreover, by ergodic theorem, for each $j$ we have $\nu((X_1,\dots,X_k),B_j)\rightarrow \rho_\x(B_j)$
a.s., so that, with probability~1,
$$
 |\nu((X_1,\dots,X_k),B_j) - \rho(B_j)|<\epsilon/(4Jw_j)
$$ from some step $k$ on; define $K_j:=k$.
 Let
$K:=\max_{j<J}K_j$ ($K$ depends
on the realization $X_1,X_2,\dots$).
Define analogously $M$ for the sequence $(Y_1,\dots,Y_m,\dots)$. Thus, for $k>K$ and $m>M$ we have
\begin{multline*}
   |\hat d(\x,\y) - d(\rho_\x,\rho_\y)| =
\\
\left|\sum_{i=1}^\infty
w_i\big(|\nu(\x,B_i)-\nu(\y,B_i)| - |\rho_\x(B_i)-\rho_\y(B_i)| \big)
\right|\\
   \le \sum_{i=1}^\infty w_i\big(|\nu(\x,B_i)-\rho_\x(B_i)| +
|\nu(\y,B_i)-\rho_\y(B_i)| \big) \\
    \le \sum_{i=1}^J w_i\big(|\nu(\x,B_i)-\rho_\x(B_i)| +
|\nu(\y,B_i)-\rho_\y(B_i)| \big) +\epsilon/2\\
   \le \sum_{i=1}^Jw_i(\epsilon/(4Jw_i) + \epsilon/(4Jw_i))
+\epsilon/2 =\epsilon,
\end{multline*}
which proves the first statement. The second statement can be proven analogously.
\end{proof}

Note that the second statement of the lemma implies that  a stationary ergodic process (or an ergodic component of a stationary process) can be asymptotically reconstructed from growing segments of a sequence it generates. 

While we shall not make use of this fact, it is also instructive to note that memory-$k$ approximations of a stationary process $\rho$ converge to $\rho$ in distributional distance. This fact is rather easy to see from the definitions. %

\section{Calculating $\hat d$}
The expressions \eqref{eq:edd1}, \eqref{eq:edd2} may seem impossible to calculate, since they involve infinite sums. However, as we shall see in this section, they are  easy to calculate exactly and, furthermore, can  be approximated using only quasilinear computational resources.

First of all, note that, for a finite sample, for finite alphabets there are only finitely many non-zero summands in~\eqref{eq:edd1} and~\eqref{eq:edd2}. For  real-valued alphabets, there are infinitely many non-zero summands, but most of these can be collapsed, as they have the same value. 

We proceed  with the more-specified versions of the empirical distributional distance, which are empirical estimates of~\eqref{eq:ddisd} and~\eqref{eq:ddisr}. Given two samples $\x=X_{1..n_1}$ and $\y:=Y_{1..n_2}$, let $n:=\max\{n_1,n_2\}$ be the size of the longer sample and define
\begin{equation}\label{eq:eddisd}
\hat d(\x,\y):=\sum_{k=1}^{k_n} w_k \sum_{B\in A^k} |\nu(\x,B)-\nu(\y,B)|,
\end{equation}
and for real-valued processes
\begin{equation}\label{eq:eddisr}
\hat d(\x,\y):=\sum_{m=1}^{m_n} \sum_{l=1}^{l_n} w_m w_l \sum_{B\in B^{m,l}} |\nu(\x,B)-\nu(\y,B)|.
\end{equation}
where $k_n,m_n,l_n$ are integer-valued parameters that grow to infinity with $n$.

First of all, note that any values of $k_n,m_n,l_n$ that monotonically increase to infinity still give consistent estimates of the distributional distance (e.g., one can check that the argument of the proof of Lemma~\ref{th:dd} is unaffected).
On the other hand,  if we set $k_n\equiv\infty$ in~\eqref{eq:eddisd}, then the inner sum in~\eqref{eq:eddisd} still has at most $n$ non-zero terms for $k\le n$ and is 0 for $k>n$. This makes the precise calculation of~\eqref{eq:eddisd} at most quadratic.

 Moreover, there is no reason to calculate the summands corresponding to $k_n\approx n$ since they are clearly not good estimates of the corresponding probabilities. In fact, it is reasonable to set $k_n$ of order $\log n$, since longer subsamples are expected to be met at most once (see, for example, \cite{Kontoyiannis:94}).

Similarly, for~\eqref{eq:eddisr}, let us begin by  showing that calculating $\hat d$ is fully tractable with $m_n,l_n\equiv\infty$.
Observe that for fixed $m$ and $l$, the sum 
\begin{equation}\label{eq:psum}
T^{m,l}:=\sum_{B\in B^{m,l}} |\nu(X_{1..n_1},B)-\nu(Y_{1..n_2},B)|
\end{equation}
has not more than $n_1+n_2 -2m+2$  nonzero terms  (assuming $m\le n_1,  n_2$; the other case is obvious).
Indeed, there are $n_1-m+1$ tuples of size $m$ in the sequence $\x$ 
namely, 
$X_{1..m}, X_{2..m+1},\dots,X_{n_1-m+1..n_1}$ and likewise for the sequence $\y$.
Therefore, $T^{m,l}$ can be obtained by a finite number of calculations. 

Furthermore, let 
\begin{equation}\label{eq:smin}
s=\min_{\substack{X_i\ne Y_j\\i=1..n_1, j=1..n_2}}|X_i-Y_j|,
\end{equation} and 
observe that $T^{m,l}=0$ for all $m>n$ and for each $m$, for all $l>\log s^{-1}$ the term $T^{m,l}$ is constant.
That is, for each fixed $m$ we have
\begin{equation*}
 \sum_{l=1}^\infty w_mw_l T^{m,l}=w_mw_{\log s^{-1}} T^{m,\log s^{-1}}+\sum_{l=1}^{\log s^{-1}} w_mw_l T^{m,l}
\end{equation*}
so that we simply double the weight of the last nonzero term. 
(Note also that  $s$ is bounded above by the length of the binary precision in representing the random variables $X_i,Y_j$.)
Thus, even with $m_n,l_n\equiv\infty$ one can calculate $\hat d$ precisely.
Moreover, for a fixed $m \in 1..\log n$ and $l \in 1..\log s^{-1}$ for every sequence 
$\x$ the frequencies $\nu(\x,B),~B \in B^{m,l}$
may be calculated  using suffix trees or suffix arrays,  with $\O(n)$ worst case construction and search
complexity (see, e.g., \cite{Ukkonen:95}). %
 Searching all  $z:=n-m+1$ occurrences of 
subsequences of length $m$ results in $\O(m+z) = \O(n)$ complexity. 
This brings the overall computational complexity of \eqref{eq:eddisr} to $\O(n m_n \log s^{-1})$; 
this can potentially be improved using specialized structures, e.g.,~\cite{Grossi:05}. 

The parameters  $m_n$ play the same role as $k_n$ in the discrete case, and so can be set to be of order $\log n$ for the same reason. 
Finally, to choose  $l_n<\infty$ one can either fix some constant based on the bound on the precision in real computations, 
or choose it in such a way that each cell $B^{m,l_n}$ contains no more than $\log n$ points for all $m=1..\log n$ largest values of $l_n$.
Thus, we arrive at the following conclusion.
\begin{svgraybox}
Empirical  distributional distance~\eqref{eq:eddisd}, \eqref{eq:eddisr}  is efficiently computable, and can be approximated  using only quasilinear computational resources. 
\end{svgraybox}

\section{The three-sample problem}\label{s:three}
Let there be given three samples $\x,\y,\z\in A^*$. %
Each sample is generated by a stationary ergodic process $\rho_\x$,
$\rho_\y$ and $\rho_\z$ respectively.
Moreover, it is known that either $\rho_\z=\rho_\x$ or $\rho_\z=\rho_\y$,
but $\rho_\x\ne \rho_\y$.
We wish to construct a test that, based on the finite samples $\x, \y$
and $\z$ will tell
whether $\rho_\z=\rho_\x$ or $\rho_\z=\rho_\y$.

This problem is known under the names of three-sample problem and (process) classification. Its i.i.d.\ version, i.e., the case when each of the samples consists of i.i.d.\ random variables,  is one of the classical problems of mathematical statistics (e.g., \cite{Lehmann:86}). The case of dependent
time series was considered  in \cite{Gutman:89}, where a solution is presented under the finite-memory assumption. The material presented here is based on \cite{Ryabko:103s}.

Essentially, the problem is to answer the question ``which distribution is closer to which other distribution'' based on the three  samples given. The test we shall consider is doing this based on the estimates of the distributional distance.

Thus, let us consider a test that chooses the sample $\x$ or $\y$ according to whichever is closer to $\z$ in $\hat d$.
That is, we define the test $L(\x,\y,\z)$ as follows. If $\hat d(\x,\z)\le \hat d(\y,\z)$ then
the test says that the sample $\z$ is generated by the same
process as the sample $\x$, otherwise it says that the sample $\z$ is generated by the same
process as the sample $\y$.

\begin{definition}[Process classifier] Define the classifier $L:A^*\times A^*\times A^*\rightarrow\{\text{``x'',``y''}\}$ as follows
$$
L(\x,\y,\z):=\left\{ \begin{array}{rl} \text{ ``x'' }&\text{ if } \hat d(\x,\z)\le \hat d(\y,\z) \\ \text{``y''}  & \text{ otherwise, } \end{array}\right.
$$
for $\x,\y,\z\in A^*$.
 
\end{definition}

\begin{theorem} The test $L(\x,\y,\z)$ makes only a finite number of errors when
$|\x|, |\y|$ and $|\z|$ go to infinity, with probability 1:  if $\rho_\x=\rho_\z$ then
$$L(\x,\y,\z)=\text{ ``x'' }$$ from some $|\x|, |\y|, |\z|$ on with probability~1; otherwise
  $$L(\x,\y,\z)=\text{ ``y'' }$$ from some $|\x|, |\y|, |\z|$ on with probability~1.
\end{theorem}
\begin{proof}
 From the fact that $d$ is a metric and from Lemma~\ref{th:dd} we
conclude that $\hat d(\x,\z)\rightarrow0$ (with probability~1) if and only if
$\rho_\x=\rho_\z$. So, if $\rho_\x=\rho_\z$ then by assumption
$\rho_\y\ne\rho_\z$ and $\hat d(\x,\z)\rightarrow0$ a.s. while
$$\hat d(\y,\z)\rightarrow d(\rho_\y,\rho_\z)\ne0.$$
Thus in this case $\hat d(\y,\z)>\hat d(\x,\z)$ from some $|\x|, |\y|, |\z|$ on with probability~1, from which moment we have $L(\x,\y,\z)=\text{ ``x'' }$. The opposite case is analogous.
\end{proof}

\section{Impossibility of discrimination}\label{s:hom}
The following problem is variously known as {\em (process) discrimination}, {\em homogeneity testing} or {\em two-sample testing}. 
For the asymptotic version we consider here the name process discrimination is more suited, and so this is the name we adopt in this section, reserving the name homogeneity testing for other versions.

Two series of  observations $X_1,X_2,\dots,X_n,\dots$ and $Y_1,Y_2,\dots,Y_n,\dots$  are presented  sequentially.
On each time step $n$ we would like to say whether the distributions generating the samples $X_1,\dots,X_n$ and $Y_1,\dots,Y_n$ are the same or different.
 In this section we are after an impossibility result, so we restrict the consideration to the case of the binary-valued processes. 

Here we shall see that there is no asymptotically consistent discrimination procedure for the stationary ergodic processes with binary alphabet. The notion of consistency is perhaps the weakest one can think of: it is shown that for any discrimination procedure its  expected answer does not converge to the correct one at least for some processes. 
In fact, a stronger result is established, showing that there is no asymptotically consistent discrimination procedure for a smaller set of process, namely, that of $B$ processes.  
The class of B-processes  (formally defined below) is sufficiently wide to include, for example, $k$-order Markov processes and functions thereof, but, on the other hand,  it is a strict subset of the set of stationary ergodic processes. 

The material of this section is after \cite{Ryabko:10discr}. The additional definitions introduced  ($B$ processes, $\bar d$-distance) as well as the proof of the main theorem are not necessary for understanding the material of the subsequent chapters.

\subsection{Setup and definitions}
Let the alphabet be binary,   $A:=\{0,1\}$.
A {\em discrimination procedure} (or a homogeneity test) $D$ is a family of mappings $D_n: A^n\times A^n\rightarrow\{0,1\}$, $n\in\N$,  that 
maps a pair of samples $(X_1,\dots,X_n)$, $(Y_1,\dots,Y_n)$ into a binary (``yes'' or ``no'') answer: the samples are generated by 
different distributions, or they are generated by the same distribution. 

A discrimination procedure $D$ is {\em asymptotically consistent for a set $\C$ of process distributions} if for any two  distributions
$\rho_\x,\rho_\y\in\C$ independently generating the sequences $X_1,X_2,\dots$ and $Y_1,Y_2,\dots$ correspondingly the expected output converges to the correct answer: the following limit exists and the equality holds
\begin{equation}\label{eq:defdiscr}
\lim_{n\rightarrow\infty}\E  D_n((X_1,\dots,X_n), (Y_1,\dots,Y_n))=\left\{\begin{array}{ll}0 & \text{ if $\rho_\x=\rho_\y$,}\\ 1 &\text{ otherwise. }\end{array}\right.
\end{equation}
This is perhaps the weakest notion of correctness one can consider.

Clearly, asymptotically consistent discriminating procedures exist for many classes of processes, for example for the class of all i.i.d.\ processes (e.g. \cite{Lehmann:86})
and various parametric families. Indeed, for i.i.d.\ samples one usually requires stronger forms of consistency than the asymptotic notion considered here.

To be able to define the set of $B$-processes, we need to introduce another distance between process distributions, the $\bar d$ distance. 

For two finite-valued stationary processes $\rho_\x$ and $\rho_\y$  the {\em $\bar d$-distance} $\bar d(\rho_\x,\rho_\y)$ is said to be less than $\epsilon$ 
 if there exists a single stationary process $\nu_{xy}$ on pairs $(X_n,Y_n)$, $n\in\N$, such that $X_n$, $n\in\N$ are distributed according 
to $\rho_\x$ and $Y_n$ are distributed according to $\rho_\y$ while 
\begin{equation}\label{eq:dbar} \nu_{xy}(X_1\ne Y_1)\le \epsilon.\end{equation}
The infimum of the $\epsilon$'s for which a coupling can be found such that (\ref{eq:dbar}) is satisfied is taken to be the $\bar d$-distance
between $\rho_\x$ and $\rho_\y$. 

\begin{definition}
 A process is called a {\em  $B$-process} (or a Bernoulli process) if it is in the $\bar d$-closure of the set of all aperiodic stationary ergodic $k$-step Markov processes, where $k\in\N$.
\end{definition}
For more information on $\bar d$-distance and $B$-processes %
see~\cite{Ornstein:74}. 

\subsection{The main result}

\begin{theorem}\label{th:discr}
 There is no asymptotically consistent discrimination procedure for the set of all $B$-processes.
\end{theorem}

Before presenting the proof, it is worth putting this result in the context of other results on $B$-processes. 
As  mentioned in the introduction, Ornstein and Weiss~\cite{Ornstein:90} construct an estimator  $\bar s_n$ %
such that 
\begin{equation}\label{eq:bes}
 \lim_{n\to\infty} \bar s_n((X_1,\dots,X_n),(Y_1,\dots,Y_n))=\bar d(\rho_1,\rho_2)\ \ \rho_1\times\rho_2{\text{--a.s.}}
\end{equation}
if both processes $\rho_1$ and $\rho_2$ generating the samples $X_i$ and $Y_i$ respectively are $B$-processes.
In the same work it is shown that there is no estimator $\bar s_n$ for which~(\ref{eq:bes}) holds for every pair $\rho_1, \rho_2$
of stationary ergodic processes.

Comparing these result to those on distributional distance presented in the previous section (namely, Lemma~\ref{th:dd}),  we can say that the stronger the distance the harder it is to estimate:
the distributional distance can be consistently estimated for stationary ergodic processes, the $\bar d$ distance
can be consistently estimated for $B$-processes but not for stationary ergodic processes, while the strongest
possible distance--- the one that gives discrete topology, cannot be consistently estimated for $B$-processes, 
as  shown in this section.

It is also worth noting that the proof given below yields a slightly stronger results, namely, the impossibility of discrimination between finite-dimensional (including single-dimensional) marginals of the processes. Specifically, correctness of the discrimination procedure~\eqref{eq:defdiscr}
can be replaced with the following
\begin{equation}\label{eq:defdiscr2}
\lim_{n\rightarrow\infty}\E  D_n((X_1,\dots,X_n), (Y_1,\dots,Y_n))=\left\{\begin{array}{ll}0 & \text{ if $\rho_\x(X_1=0)=\rho_\y(Y_1=0)$,}\\ 1 &\text{ otherwise. }\end{array}\right.
\end{equation}
with the same proof carrying over.

The proof, presented below, is by contradiction. It is assumed that a consistent discrimination procedure exists, and a process is exhibited 
that will trick such a procedure to give divergent results.
The construction on which the proof is based uses  the  ideas of the ``random walk over the diagonal'' construction  used in \cite{BRyabko:88}  to demonstrate that  
consistent prediction for stationary ergodic processes is impossible (see also its exposition in \cite{Gyorfi:98}).
\begin{proof} 
We will assume that asymptotically consistent discrimination  procedure $D$ for the class of all $B$-processes exists, and will construct a $B$-process $\rho$ such that 
if both sequences $X_i$ and $Y_i$, $i\in\N$ are generated by $\rho$ then $\E D_n$ diverges; this contradiction will prove the theorem.

The scheme of the proof is as follows. 
On Step 1 we construct  a sequence of processes $\rho_{2k}$,  $\rho_{d2k+1}$, and $\rho_{u 2k+1}$, where $k=0,1,\dots$.
On Step 2 we construct a process $\rho$, which is shown to be  the limit of  the sequence $\rho_{2k}$, $k\in\N$, in $\bar d$-distance. 
On Step 3 we show  that two independent runs of the process $\rho$ have 
a property that (with high probability) they first behave like two runs  of a single process $\rho_0$, then like two runs
of two different processes $\rho_{u1}$ and $\rho_{d1}$, then like two runs of a single process $\rho_2$, and so on, thereby showing
that the test $D$ diverges and obtaining the desired contradiction.

Assume that there exists an asymptotically consistent discriminating procedure~$D$. 
Fix some $\epsilon\in(0,1/2)$ and $\delta \in[1/2,1)$, to be defined on Step~3.

{\em Step 1.} We will construct the sequence of process $\rho_{2k}$,  $\rho_{u 2k+1}$, and $\rho_{d 2k+1}$, where $k=0,1,\dots$.

{\em Step 1.0.}
Construct the process $\rho_0$ as follows. A Markov chain $m_0$ is defined on the set $\N$ of states.
From each state $i\in\N$ the chain passes to the state $0$ with probability $\delta$ and to the state ${i+1}$ with
probability $1-\delta$. 
 With transition probabilities so defined, the chain possesses a unique stationary distribution $M_0$ on the set $\N$, 
 which can be calculated explicitly using e.g. \cite[Theorem VIII.4.1]{Shiryaev:96}, and is as follows:  $M_0(0)=\delta$,  $M_0(k)=\delta(1-\delta)^k$, 
  for all  $k\in\N$.
Take this distribution as the initial distribution over the states. 

The function $f_0$ maps the states to the output alphabet $\{0,1\}$ as follows: $f_0(i)=1$ for every $i\in\N$.
 Let $s_t$ be the state of the chain at time $t$. The process $\rho_0$
is defined as $\rho_0= f_0(s_t)$ for  $t\in\N$. %
As a result
of this definition, the process $\rho_0$ simply outputs $1$ with probability $1$ on every time step (however, by using different functions $f$ 
we will have less trivial processes in the sequel).  Clearly, the constructed process is stationary ergodic and a B-process.
So, we have defined the chain $m_0$ (and  the process $\rho_0$) up to a parameter $\delta$.

{\em Step 1.1.} We begin with the process $\rho_0$ and the chain $m_0$  of the previous step.
 Since the test D is asymptotically consistent  we will have
$$
 \E_{\rho_0\times\rho_0}  D_{t_0}((X_1,\dots,X_{t_0}), (Y_1,\dots,Y_{t_0})) <\epsilon,
$$ from some  $t_0$ on, where both samples
$X_i$ and $Y_i$ are generated by $\rho_0$ (that is, both samples consist of 1s only). 
Let $k_0$ be such an index that 
the chain $m_0$ starting from the state $0$ with probability $1$ does not reach the state $k_0-1$ by time $t_0$ (we can take $k_0=t_0+2$). 

Construct two processes $\rho_{u1}$ and $\rho_{d1}$ as follows. 
They are also based on the Markov chain $m_0$, but the functions $f$ are different. 
The function $f_{u1}:\N\rightarrow\{0,1\}$ is defined as follows: $f_{u1}(i)=f_0(i)=1$ for $i\le k_0$ and  $f_{u1}(i)=0$ for $i>k_0$.
The function $f_{d1}$ is identically $1$ ($f_{d1}(i)=1$, $i\in\N$).
The processes $\rho_{u1}$ and $\rho_{d1}$
are defined as $\rho_{u1}= f_{u1}(s_t)$  and $\rho_{d1}= f_{d1}(s_t)$ for $t\in\N$. Thus the process $\rho_{d1}$ will again produce only 1s, but the process $\rho_{u1}$ will occasionally produce~0s.

{\em Step 1.2.}
Being run on two samples generated by the processes $\rho_{u1}$ and $\rho_{d1}$ which both start from the state 0, 
 the test $D_n$ on the first $t_0$ steps produces many 0s,
since on these first $k_0$ states all the functions $f$, $f_{u1}$ and $f_{d1}$ coincide.
 However, since the processes are different and the test is asymptotically consistent (by assumption), the test starts producing 1s, until by  a certain 
 time step $t_1$ almost all answers are 1s. 
Next we will construct the process $\rho_2$ by ``gluing'' together  $\rho_{u1}$ and $\rho_{d1}$ and continuing them in such a way that, being 
run on two samples produced by $\rho_2$ the test first produces 0s (as if the samples were drawn from $\rho_0$), then, with probability close to 1/2 it will
produce many 1s (as if the samples were from $\rho_{u1}$ and $\rho_{d1}$) and then again 0s. 

The process $\rho_2$ is the pivotal point of the construction, so we give it in some detail.  On step 1.2a we present 
the construction of the process, and on step 1.2b we show that this process is a $B$-process by demonstrating that 
it  is equivalent to a (deterministic) function of a Markov chain.

{\em Step 1.2a.}
Let $t_1>t_0$ be such a time index that 
$$
 \E_{\rho_{u1}\times\rho_{d1}}  D_k((X_1,\dots,X_{t_1}), (Y_1,\dots,Y_{t_1})) >1-\epsilon,
$$ where the samples
$X_i$ and $Y_i$ are generated by $\rho_{u1}$ and $\rho_{d1}$ correspondingly (the samples are generated independently; that is, the process are based on two 
independent copies of the Markov chain $m_0$).
 Let $k_1>k_0$ be such an index that 
the chain $m$ starting from the state 0 with probability $1$ does not reach the state $k_1-1$ by time $t_1$.

Construct the process $\rho_2$ as follows (see fig.~\ref{fig:rho2}).
\begin{figure*}[h]\caption{The processes $m_2$ and $\rho_2$. {\small The states are depicted as circles, the arrows symbolize transition probabilities: from every state the process returns to 0 with probability $\delta$ or goes to the next state with probability $1-\delta$. From the switch $S_2$ the process
passes to the state indicated by the switch (with probability 1); here it is the state  $u_{k_0+1}$. When the process passes through the reset { $\bf R_2$} the switch $S_2$ is 
set to either $up$ or $down$ with equal probabilities. (Here $S_2$ is in the position $up$.) The function $f_2$ is 1 on all states except $u_{k0+1},\dots,u_{k1}$ where it is 0; $f_2$ applied to the states output by $m_2$ defines $\rho_2$.}}\label{fig:rho2}
\begin{picture}(200,120)(-0,-50)
\put(20,10){\circle{5}}
\put(18,3){\text{\tiny{0}}}
\put(22,10){\vector(1,0){20}}
\put(22,11){\text{\tiny{$1-\delta$}}}

\put(46,10){\circle{5}}
\put(45,3){\text{\tiny{1}}}
\put(49,10){\vector(1,0){20}}
\put(49,11){\text{\tiny{$1-\delta$}}}

\put(73,10){\text{{...}}}
\put(85,10){\vector(1,0){20}}
\put(85,11){\text{\tiny{$1-\delta$}}}
\put(110,10){\circle{5}}
\put(107,3){\text{\tiny{$k_0$}}}

\put(115,10){\vector(1,0){20}}
\put(112,11){\text{\tiny{$1-\delta$}}}
\put(135,4){\text{\tiny{$S_2$}}}
\put(140,10){\circle*{5}}
\put(143,13){\vector(1,1){2}}
\put(150,15){\text{\tiny{$up\ (1/2)$}  \ \ \ \ $R_2$\ \  resets $S_2$ stochastically }}
\put(150,1){\text{\tiny{$down\ (1/2)$} probability}}

\put(145,15){\vector(1,1){20}}
\put(151,28){\text{\tiny{1}}}

\put(168,40){\circle{5}}
\put(165,35){\text{\tiny{$u_{k_0+1}$}}}
\put(170,40){\vector(1,0){20}}
\put(170,41){\text{\tiny{$1-\delta$}}}
\put(194,40){\text{{...}}}
\put(206,41){\vector(1,0){20}}
\put(206,41){\text{\tiny{$1-\delta$}}}
\put(231,40){\circle{5}}
\put(228,33){\text{\tiny{$u_{k_1}$}}}

\put(146, 2){\vector(1,-1){20}}
\put(149,-13){\text{\tiny{1}}}

\put(168,-18){\circle{5}}
\put(165,-25){\text{\tiny{\ \ $d_{k_0+1}$}}}
\put(170,-18){\vector(1,0){20}}
\put(170,-17){\text{\tiny{$1-\delta$}}}
\put(194,-18){\text{{...}}}
\put(206,-18){\vector(1,0){20}}
\put(206,-17){\text{\tiny{$1-\delta$}}}
\put(231,-18){\circle{5}}
\put(228,-25){\text{\tiny{\ \ \ $d_{k_1}$}}}

\put(237,37){\vector(1,-1){23}}
\put(249,25){\text{\tiny{$1-\delta$}}}
\put(247,-12){\text{\tiny{$1-\delta$}}}
\put(237,-18){\vector(1,1){23}}

\put(261,8){\text{\small{$\mathbf{R_2}$}}}
\put(150,8){\text{\tiny{$>$ - - - - - - - - - - - - - - - - - - - - - - - - - - - - - - - - - - - -}}}

\put(275,10){\vector(1,0){20}}
\put(279,11){\text{\tiny{$1$}}}
\put(299,10){\circle{5}}
\put(296,3){\text{\tiny{$k_1+1$}}}
\put(303,10){\vector(1,0){20}}
\put(303,11){\text{\tiny{$1-\delta$}}}
\put(326,10){\text{{...}}}

\put(35,60){\text{\small{$\leftarrow\leftarrow\leftarrow\leftarrow\leftarrow\leftarrow\leftarrow\leftarrow\leftarrow\leftarrow\leftarrow\leftarrow\leftarrow\leftarrow\leftarrow\leftarrow\leftarrow\leftarrow\leftarrow\leftarrow\leftarrow\leftarrow\leftarrow\leftarrow\leftarrow\leftarrow\leftarrow\leftarrow\leftarrow\leftarrow\leftarrow\leftarrow\leftarrow\leftarrow\leftarrow$}}}
\put(46,15){\vector(0,1){46}}
\put(49,30){\text{\tiny{$\delta$}}}
\put(110,15){\vector(0,1){46}}
\put(113,30){\text{\tiny{$\delta$}}}

\put(167,43){\vector(0,1){16}}
\put(170,50){\text{\tiny{$\delta$}}}
\put(231,43){\vector(0,1){16}}
\put(234,50){\text{\tiny{$\delta$}}}

\put(167,-20){\vector(0,-1){27}}
\put(162,-35){\text{\tiny{$\delta$}}}
\put(231,-20){\vector(0,-1){27}}
\put(226,-35){\text{\tiny{$\delta$}}}

\put(298,15){\vector(0,1){46}}
\put(301,30){\text{\tiny{$\delta$}}}

\put(35,60){\vector(-1,-3){16}}
\put(35,-45){\vector(-1,3){16}}

\put(35,-50){\text{\small{$\leftarrow\leftarrow\leftarrow\leftarrow\leftarrow\leftarrow\leftarrow\leftarrow\leftarrow\leftarrow\leftarrow\leftarrow\leftarrow\leftarrow\leftarrow\leftarrow\leftarrow\leftarrow\leftarrow\leftarrow\leftarrow\leftarrow\leftarrow\leftarrow\leftarrow\leftarrow\leftarrow\leftarrow\leftarrow\leftarrow\leftarrow\leftarrow\leftarrow\leftarrow\leftarrow$}}}
\end{picture}
\end{figure*}
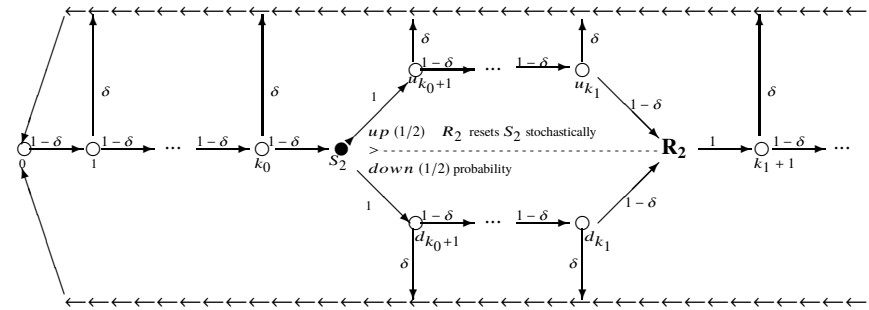
 It is based on a chain $m_2$ on which Markov assumption is violated. 
The transition probabilities on states $0,\dots,k_0$ are the same as for the Markov chain $m$ (from each
state return to 0 with probability $\delta$ or go to the next state with probability $1-\delta$). 

There are two ``special'' states: the ``switch'' $S_2$ and 
the ``reset'' $R_2$.  From the state $k_0$ the chain passes with probability $1-\delta$  to the ``switch'' state
$S_2$. 
The switch $S_2$ can itself have two values: $up$ and $down$. If $S_2$ has the value $up$ then from $S_2$ the chain passes to the state $u_{k_0+1}$ with probability 1, while
if $S_2=down$ the chain goes to $d_{k_0+1}$, with probability 1.  If the chain reaches the state $R_2$ then 
the value of $S_2$ is set to $up$ with probability 1/2 and with probability 1/2 it is set to $down$.
 In other words, the first transition from $S_2$ is random 
(either to $u_{k_0+1}$ or to  $d_{k_0+1}$ with equal probabilities) and then this decision is remembered until the ``reset'' state $R_2$ is visited, whereupon the
switch again assumes the values $up$ and $down$ with equal probabilities.

The rest of the transitions are as follows. From each state $u_i$, $k_0\le i\le k_1$ the chain passes to the state $0$ with probability $\delta$ and to the next state $u_{i+1}$ with probability $1-\delta$.
From the state $u_{k_1}$ the process goes with probability $\delta$ to 0 and with probability $1-\delta$ to the ``reset'' state $R_2$. The same with states $d_i$: 
for $k_0<i\le k_1$ the process returns to 0 with probability $\delta$ or goes to the next state $d_{i+1}$ with probability $1-\delta$, where the next state
for $d_{k_1}$ is the ``reset'' state $R_2$. From $R_2$ the process goes with probability 1 to the state $k_1+1$ where from the chain continues ad infinitum:
to the state 0 with probability $\delta$ or to the next state $k_1+2$ etc. with probability $1-\delta$.

The initial distribution on the states is defined as follows.
 The probabilities of the states $0..k_0, k_1+1,k_1+2,\dots$ are the same
as in the Markov chain $m_0$, that is, $\delta(1-\delta)^j$, for $j=0..k_0, k_1+1,k_1+2,\dots$. 
For the states $u_j$ and $d_j$, $k_0<j\le k_1$ define  their initial probabilities to be 1/2 of the probability of the corresponding
state in the chain $m_0$, that is $m_2(u_j)=m_2(d_j)=m_0(j)/2=\delta(1-\delta)^j/2$. Furthermore, if the chain starts in a state $u_j$,
$k_0<j\le k_1$, then the value of the switch $S_2$ is $up$, and if it starts in the state $d_j$ then the value of the switch $S_2$ is $down$,
whereas if the chain starts in any other state then the probability distribution on the values of the switch $S_2$ is  1/2 for either $up$ or $down$.

The function $f_2$ is defined as follows: $f_2(i)=1$ for $0\le i\le k_0$ and $i>k_1$ (before the switch and after the reset); $f_2(u_i)=0$ for all $i$, $k_0<i\le k_1$ and $f_2(d_i)=1$ for all $i$, $k_0<i\le k_1$.
The function $f_2$ is undefined on $S_2$ and $R_2$, therefore there is no output on these states (we also assume that passing through $S_2$ and $R_2$ does not increment time). As before, the process $\rho_2$ is defined as $\rho_2=f_2(s_t)$ where
$s_t$ is the state of $m_2$ at time $t$, omitting the states $S_2$ and $R_2$.
The resulting process s illustrated on fig.~\ref{fig:rho2}.

{\em Step 1.2b.}
To show that the process $\rho_2$ is stationary ergodic and a $B$-process, we will 
show that it is equivalent to a function of a stationary ergodic Markov chain, whereas all such process
are known to be $B$ (e.g. \cite{Shields:96}). 
The construction is as follows  (see fig.~\ref{fig:m2}). This chain has states %
 $k_1+1,\dots$ and also 
$u_0,\dots,u_{k_0},u_{k_0+1},\dots,u_{k_1}$ and
$d_0,\dots,d_{k_0},d_{k_0+1},\dots,d_{k_1}$.
\begin{figure*}[!t]\caption{The process $m_2'$. \small{The function $f_2$ is 1 everywhere except the states $u_{k_0+1},\dots,u_{k_1}$, where it is 0.}}\label{fig:m2}
\begin{picture}(200,120)(-0,-50)

\put(20,40){\circle{5}}
\put(18,34){\text{\tiny{$u_0$}}}
\put(22,40){\vector(1,0){20}}
\put(22,41){\text{\tiny{$1-\delta$}}}

\put(46,40){\circle{5}}
\put(45,34){\text{\tiny{$u_1$}}}
\put(49,40){\vector(1,0){20}}
\put(49,41){\text{\tiny{$1-\delta$}}}

\put(73,40){\text{{...}}}
\put(85,40){\vector(1,0){20}}
\put(85,41){\text{\tiny{$1-\delta$}}}
\put(110,40){\circle{5}}
\put(107,34){\text{\tiny{$u_{k_0-1}$}}}

\put(115,40){\vector(1,0){20}}
\put(112,41){\text{\tiny{$1-\delta$}}}
\put(140,40){\circle{5}}
\put(137,34){\text{\tiny{$u_{k_0}$}}}
\put(143,40){\vector(1,0){20}}
\put(142,41){\text{\tiny{$1-\delta$}}}

\put(20,-20){\circle{5}}
\put(18,-26){\text{\tiny{$d_0$}}}
\put(22,-20){\vector(1,0){20}}
\put(22,-19){\text{\tiny{$1-\delta$}}}

\put(46,-20){\circle{5}}
\put(45,-26){\text{\tiny{$d_1$}}}
\put(49,-20){\vector(1,0){20}}
\put(49,-19){\text{\tiny{$1-\delta$}}}

\put(73,-20){\text{{...}}}
\put(85,-20){\vector(1,0){20}}
\put(85,-19){\text{\tiny{$1-\delta$}}}
\put(110,-20){\circle{5}}
\put(107,-26){\text{\tiny{$d_{k_0-1}$}}}

\put(115,-20){\vector(1,0){20}}
\put(112,-19){\text{\tiny{$1-\delta$}}}
\put(140,-20){\circle{5}}
\put(137,-26){\text{\tiny{$d_{k_0}$}}}
\put(143,-20){\vector(1,0){20}}
\put(142,-19){\text{\tiny{$1-\delta$}}}

\put(168,40){\circle{5}}
\put(165,35){\text{\tiny{$u_{k_0+1}$}}}
\put(170,40){\vector(1,0){20}}
\put(170,41){\text{\tiny{$1-\delta$}}}
\put(194,40){\text{{...}}}
\put(206,41){\vector(1,0){20}}
\put(206,41){\text{\tiny{$1-\delta$}}}
\put(231,40){\circle{5}}
\put(228,33){\text{\tiny{$u_{k_1}$}}}

\put(168,-18){\circle{5}}
\put(165,-25){\text{\tiny{$d_{k_0+1}$}}}
\put(170,-18){\vector(1,0){20}}
\put(170,-17){\text{\tiny{$1-\delta$}}}
\put(194,-18){\text{{...}}}
\put(206,-18){\vector(1,0){20}}
\put(206,-17){\text{\tiny{$1-\delta$}}}
\put(231,-18){\circle{5}}
\put(228,-25){\text{\tiny{$d_{k_1}$}}}

\put(237,37){\vector(1,-1){23}}
\put(249,25){\text{\tiny{$1-\delta$}}}
\put(247,-12){\text{\tiny{$1-\delta$}}}
\put(237,-18){\vector(1,1){23}}

\put(265,10){\circle{5}}
\put(262,3){\text{\tiny{$k_1+1$}}}

\put(270,10){\vector(1,0){20}}
\put(270,11){\text{\tiny{$1-\delta$}}}
\put(294,10){\circle{5}}
\put(291,3){\text{\tiny{$k_1+2$}}}
\put(298,10){\vector(1,0){20}}
\put(298,11){\text{\tiny{$1-\delta$}}}
\put(321,10){\text{{...}}}

\put(35,59){\text{\small{$\leftarrow\leftarrow\leftarrow\leftarrow\leftarrow\leftarrow\leftarrow\leftarrow\leftarrow\leftarrow\leftarrow\leftarrow\leftarrow\leftarrow\leftarrow\leftarrow\leftarrow\leftarrow\leftarrow\leftarrow\leftarrow\leftarrow\leftarrow\leftarrow\leftarrow\leftarrow\leftarrow\leftarrow\leftarrow\leftarrow\leftarrow\leftarrow\leftarrow\leftarrow\leftarrow$}}}

\put(46,45){\vector(0,1){16}}
\put(48,49){\text{\tiny{$\delta$}}}
\put(110,45){\vector(0,1){16}}
\put(112,49){\text{\tiny{$\delta$}}}

\put(167,43){\vector(0,1){16}}
\put(170,50){\text{\tiny{$\delta$}}}
\put(231,43){\vector(0,1){16}}
\put(234,50){\text{\tiny{$\delta$}}}

\put(167,-20){\vector(0,-1){27}}
\put(170,-35){\text{\tiny{$\delta$}}}
\put(231,-20){\vector(0,-1){27}}
\put(234,-35){\text{\tiny{$\delta$}}}

\put(109,-26){\vector(0,-1){21}}
\put(110,-35){\text{\tiny{$\delta$}}}
\put(139,-26){\vector(0,-1){21}}
\put(141,-35){\text{\tiny{$\delta$}}}
\put(45,-26){\vector(0,-1){21}}
\put(47,-35){\text{\tiny{$\delta$}}}

\put(38,-49){\text{\small{$\leftarrow\leftarrow\leftarrow\leftarrow\leftarrow\leftarrow\leftarrow\leftarrow\leftarrow\leftarrow\leftarrow\leftarrow\leftarrow\leftarrow\leftarrow\leftarrow\leftarrow\leftarrow\leftarrow\leftarrow\leftarrow\leftarrow\leftarrow\leftarrow\leftarrow\leftarrow\leftarrow\leftarrow\leftarrow\leftarrow\leftarrow\leftarrow\leftarrow\leftarrow\leftarrow$}}}

\put(294,13){\vector(0,1){45}}
\put(295,34){\text{\tiny{$\delta/2$}}}
\put(265,13){\vector(0,1){45}}
\put(266,34){\text{\tiny{$\delta/2$}}}

\put(294,3){\vector(0,-1){48}}
\put(295,-34){\text{\tiny{$\delta/2$}}}
\put(265,3){\vector(0,-1){48}}
\put(266,-34){\text{\tiny{$\delta/2$}}}

\put(35,60){\vector(-1,-1){16}}
\put(38,-45){\vector(-1,1){19}}
\end{picture}
\end{figure*}
 From the states $u_i$, $i=0,\dots,k_1$ the 
chain passes with probability $1-\delta$ to the next state $u_{i+1}$, where the next state for $u_{k_1}$ is $k+1$ and with probability $\delta$ returns to the state $u_0$ (and not to the state 0). Transitions
for the state $d_0,\dots,d_{k_1-1}$ are defined analogously. Thus the states $u_{k_i}$ correspond to the state $up$ of the switch $S_2$ and the states $d_{k_i}$~--- to the state
$down$ of the switch. Transitions for the states $k+1,k+2,\dots$ are defined as follows: with probability $\delta/2$ to the state $u_0$, with probability $\delta/2$ to 
the state $d_0$,  and with probability $1-\delta$ to the next state.
Thus, transitions to 0 from  the states with indices greater than $k_1$ corresponds to the reset $R_2$. Clearly, the chain $m_2'$ as defined possesses a unique stationary distribution $M_2$ over the set of states and $M_2(i)>0$ for every state $i$. Moreover, this distribution is the same as the initial distribution on the states of the chain $m_0$,
except  for the states $u_i$ and $d_i$, for which we have $m_2'(u_i)=m_2'(d_i)=m_0(i)/2=\delta(1-\delta)^i/2$, for $0\le i\le k_0$. 
We take  this distribution as its initial distribution on the states of $m_2'$. The resulting process $m_2'$ is stationary ergodic, and a $B$-process, since
it is   a function of a Markov chain \cite{Shields:96}. 
It is easy to see that if we define the function $f_2$ on the states of $m_2'$ as 1 on all states except $u_{k_0+1},\dots,u_{k_1}$, then the resulting
process is exactly the process $\rho_2$. Therefore, $\rho_2$ is stationary ergodic and a $B$-process.

{\em Step 1.$k$.}
As before, we can continue the construction of 
 the processes $\rho_{u3}$ and $\rho_{d3}$, that start with a segment of $\rho_2$. 
Let $t_2>t_1$ be a time index such that 
$$
 \E_{\rho_2\times\rho_2}  D_{t_2} <\epsilon,
$$ where both samples are generated by $\rho_2$. Let $k_2>k_1$ be  such an index that when starting from the state 0  the process $m_2$ with probability 1 does not 
reach $k_2-1$ by time $t_2$ (equivalently: the process $m_2'$ does not reach $k_2-1$ when starting from either  $u_0$ or $d_0$). 
The processes $\rho_{u3}$ and $\rho_{d3}$ are based on the same process $m_2$ as  $\rho_2$. The functions $f_{u3}$ and $f_{d3}$ coincide with $f_2$  on all states up to the state $k_2$ (including the states $u_i$ and $d_i$, $k_0<i\le k_1$).  After $k_2$ the function $f_{u3}$ outputs 0s while $f_{d3}$ outputs 1s: $f_{u3}(i)=0$, $f_{d3}(i)=1$ for 	$i>k_2$.

Furthermore, we find a time  $t_3>t_2$ by which we have 
$
 \E_{\rho_{u3}\times\rho_{d3}}  D_{t_3} >1-\epsilon,
$ where the samples are generated by  $\rho_{u3}$ and $\rho_{d3}$, which is possible since $D$ is consistent.
Next, find an  index $k_3>k_2$ such that the process $m_2$ does not reach $k_3-1$ with probability $1$ if the processes $\rho_{u3}$ and $\rho_{d3}$ 
are used to produce two independent sequences and both start from the state 0.
  We then construct the process $\rho_4$ based on a (non-Markovian) process $m_4$ by ``gluing''
together $\rho_{u3}$ and $\rho_{d3}$ after the step $k_3$ with a switch $S_4$ and a reset $R_4$ exactly as was done when constructing the process $\rho_2$. 
The process $m_4$ is illustrated on fig.~\ref{fig:m4}a). 
The process $m_4$ can be shown to be equivalent to a Markov chain $m_4'$, 
 which  is constructed analogously to the chain $m_2'$  
(see fig.~\ref{fig:m4}b).
Thus, the process $\rho_4$ is can be shown to be a $B$-process.

\begin{figure*}[h]\caption{a) The processes $m_4$. %
b) The Markov chain $m_4'$}\label{fig:m4}
\begin{picture}(200,120)(45,-50)

\newsavebox{\blo}%
\savebox{\blo}(0,0)[l]{
\put(20,10){\line(1,0){15}}
\put(35,10){\circle*{4}}
\put(37,12){\vector(1,1){2}}

\put(38,13){\line(1,1){10}}
\put(38,7){\line(1,-1){10}}

\put(48,23){\line(1,0){15}}
\put(48,-3){\line(1,0){15}}

\put(63,23){\line(1,-1){11}}
\put(63,-3){\line(1,1){11}}
\put(38,8){\text{\tiny{$S_2$}}}
\put(74,8){\text{\tiny{$R_2$}}}
\put(98,8){\text{\tiny{$S_4$}}}

\put(81,10){\line(1,0){15}}
\put(47,25){\text{\tiny{$f_4=0$}}}
\put(18,12){\text{\tiny{$f_4=1$}}}
\put(47,-8){\text{\tiny{$f_4=1$}}}

}

\newsavebox{\blc}%
\savebox{\blc}(0,0)[l]{
\put(28,23){\line(1,0){15}}
\put(28,-3){\line(1,0){15}}

\put(48,23){\line(1,0){15}}
\put(48,-3){\line(1,0){15}}

\put(63,23){\line(1,-1){11}}
\put(63,-3){\line(1,1){11}}

\put(81,10){\line(1,0){15}}
\put(47,25){\text{\tiny{$f_4=0$}}}
\put(18,25){\text{\tiny{$f_4=1$}}}
\put(18,-8){\text{\tiny{$f_4=1$}}}
\put(47,-8){\text{\tiny{$f_4=1$}}}
\put(78,12){\text{\tiny{$f_4=1$}}}

}

\newsavebox{\bloo}%
\savebox{\bloo}(0,0)[l]{
\put(20,10){\line(1,0){15}}
\put(35,10){\circle*{4}}
\put(37,8){\vector(1,-1){2}}

\put(38,13){\line(1,1){10}}
\put(38,7){\line(1,-1){10}}

\put(48,23){\line(1,0){15}}
\put(48,-3){\line(1,0){15}}

\put(63,23){\line(1,-1){11}}
\put(63,-3){\line(1,1){11}}
\put(74,8){\text{\tiny{$R_4$}}}

\put(81,10){\line(1,0){15}}
\put(47,25){\text{\tiny{$f_4=0$}}}
\put(18,12){\text{\tiny{$f_4=1$}}}
\put(47,-8){\text{\tiny{$f_4=1$}}}
\put(78,12){\text{\tiny{$f_4=1$}}}

}

\put(40,10){\usebox{\blo}}

\put(100,10){\usebox{\bloo}}

\put(196,9){\text{\tiny{$\dots$}}}

\put(234,46){\usebox{\blc}}
\put(234,-36){\usebox{\blc}}%

\put(341,15){\line(1,0){10}}
\put(341,-5){\line(1,0){10}}
\put(351,-5){\line(1,1){10}}
\put(351,15){\line(1,-1){10}}
\put(361,5){\line(1,0){10}}
\put(371,4){\text{\tiny{$\dots$}}}

\put(330,46){\line(1,-3){10}}
\put(330,-36){\line(1,3){10}}

\put(341,18){\text{\tiny{$f_4=0$}}}
\put(341,-10){\text{\tiny{$f_4=1$}}}

\put(40,40){\text{a)}}

\put(230,40){\text{b)}}

\end{picture}
\end{figure*}
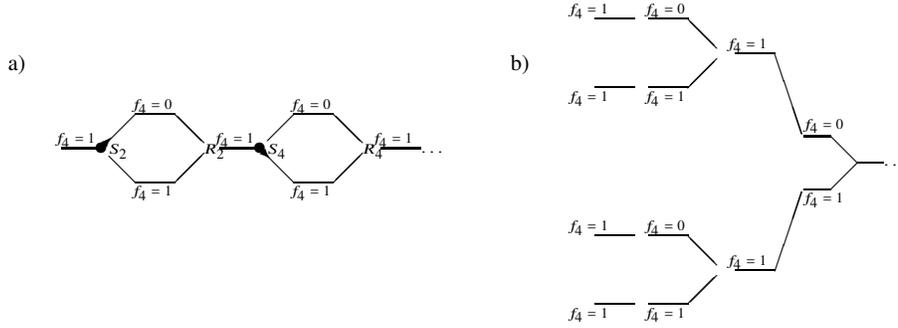

Proceeding this way we can construct the processes $\rho_{2j}$, $\rho_{u2j+1}$ and $\rho_{d2j+1}$, $j\in\N$ choosing the time steps $t_j>t_{j-1}$ so that
the expected output of the test approaches 0 by the time $t_j$ being run on  two samples produced by $\rho_j$ for even $j$, and approaches 1 by the time $t_j$ being run on samples produced by $\rho_{uj}$ and $\rho_{dj}$ for odd $j$:
\begin{equation}\label{eq:teven}
 \E_{\rho_{2j}\times\rho_{2j}}  D_{t_{2j}} <\epsilon
\end{equation}
and 
\begin{equation}\label{eq:todd}
 \E_{\rho_{u2j+1}\times\rho_{d2j+1}}  D_{t_{2j+1}} > (1-\epsilon).
\end{equation}
For each $j$ the number $k_j>k_{j-1}$  is selected in a such a way that the state  $k_j-1$ is not reached (with probability 1) by the  time $t_j$ when starting from the state 0.
Each of the processes $\rho_{2j}$, $\rho_{u2j+1}$ and $\rho_{dj2+1}$, $j\in\N$ can be shown to be stationary ergodic and a $B$-process by demonstrating
equivalence to a Markov chain, analogously to the Step 1.2. The initial state distribution of each of the processes $\rho_t, t\in\N$ is  $M_{t}(k)=\delta(1-\delta)^k$ and 
$M_{t}(u_k)=M_{t}(d_k)= \delta(1-\delta)^k/2$ for those $k\in\N$ for which the corresponding states are defined.

{\em Step 2.} 
Having defined $k_j$, $j\in\N$ we can define the process $\rho$. The construction is given on Step~2a, while on Step~2b we show
that $\rho$ is stationary ergodic and a $B$-process, by showing that it is the limit of the sequence  $\rho_{2j}$, $j\in\N$. 

{\em Step 2a.}
The process $\rho$ can be constructed as follows (see fig.~\ref{fig:rho}). 
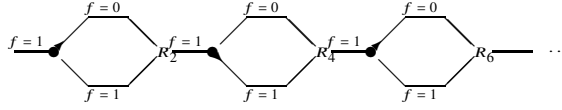
\begin{figure}[h]\caption{The processes $m_\rho$ and $\rho$. The states are on horizontal lines. The function $f$ being applied to the states of $m_\rho$ 
defines the process $\rho$. Its value is  $0$ on the states on the upper lines (states $u_{k_{2j}+1},\dots,u_{k_{2j+1}}$, where $k\in\N$) and 1 on the rest of the states.}\label{fig:rho}
\begin{picture}(200,47)(-70,-7)

\newsavebox{\bloc}%
\savebox{\bloc}(0,0)[l]{
\put(20,10){\line(1,0){15}}
\put(35,10){\circle*{4}}
\put(37,12){\vector(1,1){2}}

\put(38,13){\line(1,1){10}}
\put(38,7){\line(1,-1){10}}

\put(48,23){\line(1,0){15}}
\put(48,-3){\line(1,0){15}}

\put(63,23){\line(1,-1){11}}
\put(63,-3){\line(1,1){11}}
\put(74,8){\text{\tiny{$R_2$}}}

\put(81,10){\line(1,0){15}}
\put(47,25){\text{\tiny{$f=0$}}}
\put(18,12){\text{\tiny{$f=1$}}}
\put(47,-8){\text{\tiny{$f=1$}}}

}
\newsavebox{\blocc}%
\savebox{\blocc}(0,0)[l]{
\put(20,10){\line(1,0){15}}
\put(35,10){\circle*{4}}
\put(37,8){\vector(1,-1){2}}

\put(38,13){\line(1,1){10}}
\put(38,7){\line(1,-1){10}}

\put(48,23){\line(1,0){15}}
\put(48,-3){\line(1,0){15}}

\put(63,23){\line(1,-1){11}}
\put(63,-3){\line(1,1){11}}
\put(74,8){\text{\tiny{$R_4$}}}

\put(81,10){\line(1,0){15}}
\put(47,25){\text{\tiny{$f=0$}}}
\put(18,12){\text{\tiny{$f=1$}}}
\put(47,-8){\text{\tiny{$f=1$}}}

}

\newsavebox{\bloccc}%
\savebox{\bloccc}(0,0)[l]{
\put(20,10){\line(1,0){15}}
\put(35,10){\circle*{4}}
\put(37,12){\vector(1,1){2}}

\put(38,13){\line(1,1){10}}
\put(38,7){\line(1,-1){10}}

\put(48,23){\line(1,0){15}}
\put(48,-3){\line(1,0){15}}

\put(63,23){\line(1,-1){11}}
\put(63,-3){\line(1,1){11}}
\put(74,8){\text{\tiny{$R_6$}}}

\put(81,10){\line(1,0){15}}
\put(47,25){\text{\tiny{$f=0$}}}
\put(18,12){\text{\tiny{$f=1$}}}
\put(47,-8){\text{\tiny{$f=1$}}}

}

\put(-10,10){\usebox{\bloc}}
\put(50,10){\usebox{\blocc}}
\put(110,10){\usebox{\bloccc}}
\put(212,10){\text{\tiny{$\dots$}}}

\end{picture}
\end{figure}
 The construction is based on the (non-Markovian) process $m_\rho$ that has states $0,\dots,k_0$, $k_{2j+1}+1,\dots,k_{2(j+1)}$, $u_{k_{2j}+1},\dots,u_{k_{2j+1}}$ and $d_{k_{2j}+1},\dots,d_{k_{2j+1}}$
for $j\in\N$, along with switch states $S_{2j}$ and reset states $R_{2j}$. Each switch $S_{2j}$ diverts the process to the state $u_{k_{2j}+1}$ if the switch has value $up$ and
to $d_{k_{2j}+1}$ if it has the value $down$. The reset $R_{2j}$ sets $S_{2j}$ to  $up$ with probability 1/2 and to $down$ also with probability 1/2.
 From each state that is neither a reset nor a switch, the process goes to the next state with probability $1-\delta$ and returns
to the state 0 with probability $\delta$ (cf. Step 1$k$). %

The initial distribution $M_\rho$ on the states of $m_\rho$ is defined as follows. For every state $i$ such that $0 \le i\le k_0$ and  $k_{2j+1}<i\le k_{2_j+2}$, $j=0,1,\dots$,
define the initial probability of the state $i$ as $M_{\rho}(i)=\delta(1-\delta)^i$ (the same as in the chain $m_0$), and for the sets $u_j$ and $d_j$ 
(for those $j$ for which these sets are defined)
let $M_{\rho}(u_j)=M_{\rho}(d_j):=\delta(1-\delta)^i/2$ (that is, 1/2 of the probability of the corresponding state of $m_0$). 

The function $f$ is defined as 1 everywhere except 
for the states $u_j$ (for all $j\in\N$ for which $u_j$ is defined)  on which $f$ takes the value 0.  The process $\rho$ is defined at time $t$ as $f(s_t)$,
where $s_t$ is the state of $m_\rho$ at time $t$. 

{\em Step 2b.} To show that $\rho$ is a $B$-process, let us first show that it is stationary. 
Recall the definition~\ref{eq:ddis} of the distributional distance between (arbitrary) process distributions.
  The set of all stochastic processes, equipped with this distance, is complete, and the set of all
stationary processes is its closed subset \cite{Gray:88}. 
Thus, to show that the process $\rho$ is stationary it suffices to show that $\lim_{j\to\infty}d(\rho_{2j},\rho)=0$, since
the processes $\rho_{2j}$, $j\in\N$, are stationary. To do this, it is enough to demonstrate
that 
\begin{equation}\label{eq:lim}
\lim_{j\to\infty}  |\rho((X_{1},\dots, X_{|B|})=B)-\rho_{2j}((X_{1},\dots, X_{|B|})=B)|=0
\end{equation}
for each $B\in A^*$.
  Since the processes $m_\rho$ and $m_{2j}$ coincide on all 
states up to $k_{2j+1}$, we have 
\begin{multline*}
|\rho(X_{n}=a)-\rho_{2j}(X_{n}=a)| = |\rho(X_{1}=a)-\rho_{2j}(X_{1}=a)|\\ \le \sum_{k>k_{2j+1}} M_\rho(k) + \sum_{k>k_{2j+1}} M_{2j}(k) 
\end{multline*} for every $n\in\N$ and $a\in A$. 
Moreover, 
 for any  tuple $B\in A^*$ we obtain
\begin{multline*}
  |\rho((X_{1},\dots, X_{|B|})=B)-\rho_{2j}((X_{1},\dots, X_{|B|})=B)| \\ \le  |B|\left(\sum_{k>k_{2j+1}} M_\rho(k) + \sum_{k>k_{2j+1}} M_{2j}(k)\right)\to0
\end{multline*}
where the convergence follows from $k_{2j}\to\infty$. %
We conclude that~(\ref{eq:lim}) holds true, so that $d(\rho,\rho_{2j})\to 0$ and  $\rho$ is stationary.

To show that $\rho$  is a $B$-process, we will demonstrate that it
is the limit of the sequence $\rho_{2k}$, $k\in\N$ in the $\bar d$ distance (which was only defined for stationary processes). Since the set of all 
$B$-process is a closed subset of all stationary processes, it will follow that $\rho$ itself
is a $B$-process.  (Observe that this way we  get ergodicity of $\rho$ ``for free'', since the set of all ergodic processes is closed in $\bar d$ distance,
and all the processes $\rho_{2j}$ are ergodic.)
In order to show that $\bar d(\rho,\rho_{2k})\to0$ we have
 to find for each $j$ a  processes $\nu_{2j}$ on pairs $(X_1,Y_1),(X_2,Y_2),\dots$, such 
that $X_i$ are distributed according to $\rho$ and $Y_i$ are distributed according to $\rho_{2j}$, 
and such that $\lim_{j\to\infty}\nu_{2j}(X_1\ne Y_1)=0$. Construct such a coupling as follows.
Consider the chains $m_\rho$ and $m_{2j}$, which start in the same state (with initial distribution being $M_\rho$) and always take state transitions together,
where if the process $m_\rho$ is in the  state $u_{t}$ or $d_t$, $t\ge k_{2j+1}$ (that is, one of the states which the chain $m_{2j}$ does not have) 
then the chain $m_{2j}$ is in  the state $t$. The first coordinate of the process $\nu_{2j}$ is obtained by applying the function $f$ to the
process $m_\rho$ and the second by applying $f_{2j}$ to the chain $m_{2j}$. Clearly, the distribution 
of the first coordinate is $\rho$ and the distribution of the second is $\rho_{2j}$. Since the chains start in the same 
state and always take state transitions together, and since the chains $m_\rho$ and $m_{2j}$ coincide up to the 
state $k_{2j+1}$  we have $\nu_{2j}(X_1\ne Y_1)\le \sum_{k>k_{2j+1}} M_\rho(k)\to0$. Thus, $\bar d(\rho, \rho_{2j})\to0$, so that  $\rho$ is a $B$-process.

{\em Step 3.}
Finally, it remains to show that the expected output of the test $D$ diverges if the test is run on two independent samples produced by $\rho$.

Recall that for all the chains $m_{2j}$, $m_{u2j+1}$ and $m_{d2j+1}$  as well as for the chain  $m_\rho$, the initial
probability of the state 0 is  $\delta$. By construction, if the process $m_\rho$ starts at the state 0 then up to the time step $k_{2j}$ it behaves exactly 
as $\rho_{2j}$ that has started at the state 0.  In symbols, we have 
\begin{equation}\label{eq:even1}
 E_{\rho\times\rho} ( D_{t_{2j}}| s_0^x=0, s_0^y=0)=E_{\rho_{2j}\times\rho_{2j}} ( D_{t_{2j}}| s_0^x=0, s_0^y=0)
\end{equation}
for $j\in\N$, where $s_0^x$ and $s_0^y$ denote the initial states of the processes generating the samples $X$ and $Y$ correspondingly.

We will use the following simple decomposition 
\begin{equation}\label{eq:dec}
 \E( D_{t_j}) = \delta^2\E( D_{t_j}|s^x_0=0, s^y_0=0) %
+ (1-\delta^2)\E( D_{t_j}|s^x_0\ne0 \text{ or } s^y_0\ne0),
\end{equation}
From this, (\ref{eq:even1}) and~(\ref{eq:teven}) we have 
\begin{multline}\label{eq:even2}
 \E_{\rho\times\rho} (  D_{t_{2j}}) \le \delta^2 \E_{\rho\times\rho} ( D_{t_{2j}} | s^x_0=0, s^y_0=0) +(1-\delta^2) \\ 
= \delta^2 \E_{\rho_{2j}\times\rho_{2j}} ( D_{t_{2j}} | s^x_0=0, s^y_0=0) +(1-\delta^2) \\
\le \E_{\rho_{2j}\times\rho_{2j}} +(1-\delta^2)  < \epsilon +(1-\delta^2).
\end{multline}

 For odd indices, if the process $\rho$ starts at the state 0 then (from the definition of $t_{2j+1}$) by the  time $t_{2j+1}$ it does not 
reach the reset $R_{2j}$;  therefore, in this case  the value of the switch $S_{2j}$ does not change up to the time $t_{2j+1}$.
 Since the definition of $m_\rho$ is symmetric with respect to the values $up$ and $down$ of each switch, the probability that two samples $x_1,\dots,x_{t_{2j+1}}$ and $y_1,\dots,y_{t_{2j+1}}$ generated  independently
by (two runs of) the process $\rho$ produced different values of the switch $S_{2j}$ when passing through it for  the first time is 1/2. In other words, with probability 1/2 two samples generated 
by  $\rho$ starting at the state 0 will look by the time $t_{2j+1}$ as two samples generated by $\rho_{u2j+1}$ and $\rho_{d2j+1}$ that has started at state 0. Thus
\begin{equation}\label{eq:odd1}
 E_{\rho\times\rho} ( D_{t_{2j+1}}| s_0^x=0, s_0^y=0)\\ \ge \frac{1}{2}E_{\rho_{u2j+1}\times\rho_{d2j+1}} ( D_{t_{2j+1}}| s_0^x=0, s_0^y=0)
\end{equation}
for $j\in\N$. Using this, (\ref{eq:dec}),  and~(\ref{eq:todd}) we obtain
\begin{multline}\label{eq:odd2}
 \E_{\rho\times\rho} (  D_{t_{2j+1}}) \ge  \delta^2 \E_{\rho\times\rho} ( D_{t_{2j+1}} | s^x_0=0, s^y_0=0)\\
\ge \frac{1}{2} \delta^2 \E_{\rho_{2j+1}\times\rho_{2j+1}} ( D_{t_{2j+1}} | s^x_0=0, s^y_0=0)\\
\ge  \frac{1}{2} \left( \E_{\rho_{2j+1}\times\rho_{2j+1}} (  D_{t_{2j+1}}) - (1-\delta^2)\right)> \frac{1}{2} (\delta^2-\epsilon).
\end{multline}

Taking $\delta$ large and $\epsilon$ small (e.g. $\delta=0.9$ and $\epsilon=0.1$), we can make the bound~(\ref{eq:even2}) close to 0 and the bound~(\ref{eq:odd2}) close to 1/2, and the  expected output of the test will cross these values infinitely often. %
Therefore, we have shown that the expected output of the test $D$ diverges on two independent runs of the process $\rho$, contradicting the consistency of $D$.
This contradiction concludes the proof.	
\end{proof}

\chapter{Clustering and change-point problems}\label{ch:clchp}
In the previous chapter we have considered some basic questions of statistical inference. It was  established that, when speaking about stationary ergodic processes, one can answer questions like ``which distribution is closer to which'' but not ``are these distributions the same,'' based on samples. In this chapter we shall see how these questions come into play when considering more complex problems, namely, clustering and change-point problems.

 Clustering is grouping together samples generated by the same distributions, while change-point problems are concerned with delimiting parts of a sample that are generated by a single process distribution.  At first glance, it seems that this kind of questions should be impossible to solve, since we cannot even answer the simple ``same-different'' question about distributions. However, we shall see that often, and mainly in the case when the total number of different distributions is known, these questions can be reduced to answering the ``which one is closer'' question, and thus admit a solution.

All the algorithms that are  mentioned in this chapter  do not present any significant computational  challenges, perhaps except for calculating the distributional distance (see Section~\ref{s:dd} above about that). Therefore, we omit algorithmic and implementational details; the interested reader can find these in the corresponding papers that also present experimental evaluations of the algorithms: \cite{Khaleghi:15clust} for clustering and \cite{Khaleghi:12mchp,Khaleghi:14,Khaleghi:15chp} for change-point problems. The material in this chapter is mainly after \cite{Ryabko:10clust,Khaleghi:15clust} for clustering and  \cite{Ryabko:103s} for change-point problems, with some results of \cite{Khaleghi:12mchp,Khaleghi:14,Khaleghi:15chp,Ryabko:17clin} given without proofs.

\section{Time-series clustering}
Given a finite set of objects, the  problem  of ``clustering'' similar objects together, in the absence
of any examples of ``good'' clusterings,  is notoriously hard to formalize.
Most of the work on clustering is concerned with particular parametric data-generating models, or with analysing  particular algorithms,
a given  similarity measure, and (very often) a given number of clusters.
It is clear that, as in almost  learning problems, in clustering finding the right similarity measure is an integral 
part of the problem. However, even if one assumes  the similarity 
measure known, it is hard to define what a good clustering is \cite{Kleinberg:02, Zadeh:09}. 
What is more, even if one assumes the similarity measure to be simply the Euclidean distance (on the plane), 
and the number of clusters  $k$ known, then clustering may still appear intractable for computational reasons \cite{Mahajan:09}.

The problem acquires a different angle when one wishes to cluster  processes. That is, 
each data point is itself a time-series sample. 
This version of the problem has numerous applications, such as clustering biological data, financial
observations, or behavioural patterns, and as such it has gained a tremendous attention in the literature.

A crucial  observation to make in the case of clustering processes, is that   one can 
 benefit from the notion of ergodicity to define what appears to be a very natural notion of consistency.
Ergodicity means that the distribution of a sample can be determined in asymptotic, or approximated arbitrary well if the sample size is long enough.
This makes the the following goal achievable. 
\begin{svgraybox}
  Given  $N$ samples $\x_1=(x^1_1,\dots,x^1_{n_1}),\dots, \x_N=(x^N_1,\dots,x^N_{n_N})$, each drawn by one out of $\kappa$ unknown process distributions, group together those and only those samples that were generated by the same distribution.
\end{svgraybox}

The samples are  $\x_j$ are  not assumed to be drawn 
independently; rather, it is assumed that the joint distribution of the samples is stationary ergodic. The target clustering is as follows:
those and only those samples are put into the same cluster  that were generated by the same distribution. 
A clustering algorithm is called asymptotically consistent if it  outputs only the correct answer with probability~1 from some $n$ on, where $n$ is the length of the shortest sample, $n:=\min\{n_1,\dots,n_N\}.$ 
Note the particular regime of asymptotic: not with respect
to the number of samples $N$, but with respect to the length of the samples $n_1,\dots,n_N$.

Clearly, the problem of clustering in this   formulation %
is a direct generalisation of the three-sample problem of Section~\ref{s:three}. Indeed, the latter problem can be seen as clustering $N=3$ samples into $\kappa=2$ clusters, where $\kappa$ is given. At the same time, the discrimination problem of Section~\ref{s:hom} can be seen as clustering $N=2$ samples into either $\kappa=1$ or $\kappa=2$ clusters, with $\kappa$ unknown. 

Anticipating, from this  we can already see when it is possible and when it is not possible to have a consistent algorithm for clustering stationary ergodic time series.
\begin{svgraybox}
 There exists a consistent algorithm for clustering stationary ergodic time series if and only if the number of clusters $\kappa$ is known. 
\end{svgraybox}

We proceed below with a more formal problem formulation and the exposition of the algorithm.
\subsection{Problem formulation}
The clustering problem can be defined as follows.  $N$ samples $\x_1,\dots,\x_N$ are given, where 
each sample $\x_i$ is of length $n_i$: $\x_i=X^i_{1..n_i}$.
The samples are generated by a distribution $P$ over $(A^N)^\infty$, that is, a distribution that generates an infinite sequence of $N$-tuples.
\begin{equation}\label{eq:mat1}
\left [\begin{array}{ccccc} X_1^1 & \dots  & X_{n_1}^1 & \dots \\ 
 & \vdots &  & \\
X_1^N & \dots  & X_{n_N}^N & \dots 
  \end{array} \right ] %
\end{equation} 
The marginal distribution of each sequence $X^i_{1..n_i,..}$ is  one out of $\kappa$ different (and  unknown) stationary ergodic distributions $\rho_1,\dots,\rho_\kappa\in\mathcal E$. Note that we allow the samples $\x_1,\dots,\x_N$ to be dependent; the only requirement is on the marginal distributions (they should be stationary ergodic). 
Thus, there is a partitioning $\G=\{\G_1,\dots,\G_\kappa\}$ of  the set $\{1..N\}$ into $\kappa$ {\em disjoint} subsets $\G_j, j=1..\kappa$ 
$$
\{1..N\}=\cup_{j=1}^\kappa \G_j,
$$
such that $\x_j$, $1\le j\le N$ is generated by $\rho_j$ if and only if $j\in \G_j$.
The partitioning $\G$ is called the {\em target (or ground-truth) clustering}  and the sets $\G_i, 1\le i\le \kappa$, are called the
{\em target clusters}. Given samples $\x_1,\dots,\x_N$ and a target clustering $\G$,  let  $\G(\x)$
denote the cluster that contains~$\x$.

A {\em clustering function} $F$ takes a finite number of samples  %
$\x_1,\dots,\x_N$ and a parameter $k$ (the target number of 
clusters) and outputs a partition $F(\x_1,\dots,\x_N,(k))=\{T_1,\dots,T_k\}$ of the set $\{1..N\}$.
\begin{definition}[asymptotic consistency] Let a finite number $N$ of samples be given, and let the target 
clustering partition be $\G$. Define  $n=\min\{n_1,\dots,n_N\}$.
 A clustering function  $F$ is strongly asymptotically consistent if  
$$
  F(\x_1,\dots,\x_N,\kappa)=\G
$$ from some $n$ on with probability~1.
 A clustering function is weakly asymptotically consistent if  
$$
  P(F(\x_1,\dots,\x_N,\kappa)=\G)\to1.
$$%
\end{definition}

Note that the consistency is asymptotic with respect to {\em the minimal length of the sample}, and not with respect to the {\em number of samples}.

\subsection{A clustering algorithm and its consistency}\label{s:clusta}
Here we present an algorithm that is shown  to be asymptotically consistent in the general framework introduced.
What makes this simple algorithm interesting is that it requires only $\kappa N$ distance calculations (where $\kappa$ is the number of clusters),  that is, much less than is needed to calculate the distance between each two sequences. 

In short, Algorithm~\ref{a0} initialises the clusters using farthest-point initialisation, 
and then assigns each remaining point to  the nearest cluster.   
More precisely, the sample $\x_1$ is assigned as the first cluster centre. 
Then a sample is found that is farthest away from $\x_1$ in the 
empirical distributional distance $\hat d$ 
and is assigned as the second cluster centre.  For each $k=2..\kappa$ 
the $k^{\text{th}}$ cluster centre is sought as the sequence 
with the largest minimum distance from the already assigned cluster centres 
for $1..k-1$. 
 By the last iteration we have $\kappa$ cluster centres.   (This  initialisation procedure was proposed in \cite{Katsavounidis:94} in the context of  $k$-means clustering.)
Next, the remaining samples are each assigned 
to the closest cluster.
\begin{algorithm}[!h]
\caption{Offline clustering  }
\label{a0}
\begin{algorithmic}[1]
\State 
\bf {INPUT}: sequences $S:=\{\x_1, \cdots, \x_N\}$, Number $\kappa$ of clusters 
\State{\bf \em Initialize $\kappa$-farthest points as cluster-centres:}
\State{$c_1 \gets 1$}
\State{$C_1\gets\{c_1\}$}
\For{$k=2..\kappa$}
\State $c_k\gets \argmaxdisp {i=1..N} \displaystyle \min_{j=1..k-1} \hat{d}(\x_i,\x_{c_{j}})$, {\em where ties are broken arbitrarily}
\State $C_k \gets \{c_k\}$ 
\EndFor
\State\bf{\em Assign the remaining points to closest centres:}
\For {$i=1 .. N$}
\State $k \gets \argmin_{j\in \bigcup_{k=1}^\kappa C_{k}}\hat{d}(\x_i,\x_{j})$
\State $C_k \gets C_k \cup \{i\}$
\EndFor
\State \bf {OUTPUT}: clusters $C_1, C_2, \cdots, C_{\kappa}$
\end{algorithmic}
\end{algorithm}
\begin{theorem}\label{th:cons}
 Algorithm~\ref{a0} is strongly asymptotically consistent %
provided that the correct number $\kappa$ of clusters is known, and the marginal distribution of each sequence $\x_i, i=1..N$ is stationary ergodic. 
\end{theorem}
To main idea of the proof is as follows.  Lemma~\ref{th:dd} implies that, 
if the samples in $S$ are long enough, the samples that are generated by the same process distribution 
are closer to each other than to the rest of the samples. 
Therefore, the samples chosen as cluster centres are each generated by a different process distribution.
The theorem then follows from the fact that the algorithm assigns the rest of the samples to the closest clusters. 
\begin{proof}
Let $n$ denote the shortest sample length in $S$:
$$
n_{\min}:=\min_{i \in 1..N} n_i.
$$
Denote by
$\delta$ the minimum nonzero distance between the process distributions:
$$
\delta:=\min_{k \neq k' \in 1..\kappa}\hat{d}(\rho_k,\rho_{k'}).
$$
Fix $\epsilon \in (0,\delta/4)$. 
Since there are a finite number $N$ of samples, 
by Lemma~\ref{th:dd} for all large enough $n_{\min}$  we have
\begin{align}\label{th:cons:consistency_d}
\sup_{\substack{k \in 1..\kappa\\i \in \G_k \cap \{1..N\}}}\hat{d}(\x_i,\rho_k) \leq \epsilon.
\end{align}
where $\G_k,~k=1..\kappa$ denote the ground-truth partitions. 
By \eqref{th:cons:consistency_d} and applying the triangle inequality   
we obtain
\begin{align}\label{th:cons:ineq0}
\sup_{\substack{k \in 1..\kappa \\i,j \in \G_k\cap \{1..N\}}}\hat{d}(\x_i,\x_j) \leq 2\epsilon. 
\end{align}
Thus, for all large enough $n_{\min}$ we have
\begin{align}\label{th:cons:ineq1}
\inf_{\substack{i \in \G_k \cap \{1..N\}\\j \in \G_{k'}\cap \{1..N\}\\k \neq k' \in 1..\kappa}}\hat{d}(\x_i,\x_j) 
&\geq \inf_{\substack{i \in \G_k \cap \{1..N\}\\j \in \G_{k'}\cap \{1..N\}\\k \neq k' \in 1..\kappa}} d(\rho_k,\rho_{k'})-\hat{d}(\x_i,\rho_k)-\hat{d}(\x_j,\rho_{k'})\notag \\
&\geq \delta- 2\epsilon
\end{align}
where the first inequality follows from the triangle inequality, 
and the second inequality follows from \eqref{th:cons:consistency_d} and the definition of $\delta$. 
In words,  \eqref{th:cons:ineq0} and \eqref{th:cons:ineq1}  
mean that the samples in $\S$ 
that are generated by the same process distribution are closer to each other than to the rest of the samples. 
Finally,  for all  $n_{\min}$ large enough to have \eqref{th:cons:ineq0} and \eqref{th:cons:ineq1} we obtain 
$$\displaystyle \max_{i=1..N}\min_{k=1..\kappa-1} \hat{d}(\x_i,\x_{c_k})\geq \delta-2\epsilon>\delta/2$$
where, as specified by Algorithm~\ref{a0}, $c_1:=1$ and $c_k:= \argmaxdisp {i=1..N} \displaystyle \min_{j=1..k-1} \hat{d}(\x_i,\x_{c_{j}}),~k=2..\kappa$. 
Hence, the indices $c_1,\dots,c_\kappa$ will be chosen to index 
sequences generated by different process distributions. To derive the consistency statement, it remains
to note that, by \eqref{th:cons:ineq0} and \eqref{th:cons:ineq1}, each remaining sequence will be assigned to the cluster centre corresponding
to the sequence generated by the same distribution.
\end{proof}

\subsection{Extensions: unknown $k$, online clustering and clustering with respect to independence}
In this section we briefly consider several extensions and modifications of the process clustering problem. The problems are only outlined, and the details are left out; the interested reader is referred to the corresponding papers that treat each of these problems in detail.
\subsubsection{Unknown number of clusters}\label{s:clunk}
As mentioned in the beginning of this section, if the number of clusters $\kappa$ is unknown, then the problem provably has no solution. Thus, if we really want to have a consistent algorithm that does not require $\kappa$, then something has to give in. Sacrificing the generality is one way of doing it. Clearly, if we assume that the speed of convergence of frequencies has a known upper-bound, as is the case when time-series are i.i.d.\ or mixing (with a bound on the mixing coefficient) then everything becomes possible. The resulting time-series clustering problem is still interesting, but clearly falls out of the scope of this volume. A simple example of an algorithm that is consistent in this setting can be found in \cite{Ryabko:10clust,Khaleghi:15clust}. It is worth noting that it remains open to establish tight upper- and lower-bounds on the error probability  of clustering algorithms even for the case of i.i.d.\ time series.
\subsubsection{Online clustering}\label{s:clon}
An interesting and practical modification of the clustering problem consists in taking it ``online.'' On each time step, new samples are revealed, which can be either a continuation of some of the time-series available on the previous steps, or form a new time series. The asymptotic setting commands that the length of each time series should grow to infinity, as should the number of time series, though they may do so in an arbitrary manner. As before, the only requirement we would like to make is that the marginal distribution of each of the processes is stationary and ergodic. There are only $\kappa$ different marginal distributions, the number $\kappa$ of these distributions is known but this is all the information we get. 

Let us describe the problem a little more formally.
Consider the two-way infinite matrix ${\bf X}$ of $A$-valued random variables
\begin{equation}\label{eq:mat}
{\bf X} := \left [\begin{array}{ccccc} X_1^1 & X_2^1 & X_3^1 & \dots 
\\ X_1^2 & X_2^2 & \dots & \dots \\ \vdots & \vdots & \ddots & \vdots\\
  \end{array} \right ] %
\end{equation} 
generated by some  probability distribution $P$ on $((A^\infty)^\infty,\mathcal B_2)$, where $\mathcal B_2$ is the corresponding Borel sigma-algebra. 
The matrix $\X$ can be seen as an infinite sequence of infinite sequences; since $(A,\mathcal B_1)$ is a standard probability space, so is $(A^\infty,\mathcal B_\infty)$ and thus $((A^\infty)^\infty,\mathcal B_2)$ is well-defined (e.g., \cite{Gray:88}).

Assume that the marginal distribution of $P$ on each row of ${\bf X}$ is one of $\kappa$  unknown stationary 
ergodic process distributions $\rho_1, \rho_2, \dots, \rho_\kappa$. 
Thus, the matrix $\bf{X}$ corresponds to infinitely many  one-way infinite sequences, each of which is generated
by a stationary ergodic distribution. 
Aside from this assumption, we do not make any further assumptions
on the distribution 
$P$ that generates~${\bf X}$.
This means that the rows of ${\bf X}$ (corresponding to different time-series samples)  are allowed to be dependent, 
and the dependence can be arbitrary; 
one can even think of the dependence between samples as {\em adversarial}.
For notational convenience we assume that the distributions 
$\rho_k, k= 1..\kappa$ are ordered based on the order of appearance of their first rows (samples) in ${\bf X}$.

As in the offline setting, the ground-truth partitioning of $\bf{X}$ is defined by grouping the rows that have the same marginal distribution.
Let $$\G=\{\G_1,\dots,\G_\kappa\}$$  be a partitioning
of $\N$ into $\kappa$  disjoint  
subsets $\G_k,~k=1..\kappa$,
such that the marginal distribution of $\x_i$, $i \in \N$ is  $\rho_k$ for some $k \in 1..\kappa$ if and only if $i\in \G_k$.
The partitioning $\G$ is called  the ground-truth clustering. 

Introduce also the notation $\G|_N$ for the restriction of $\G$ to the first $N$ sequences: $$\G|_N:=\{\G_k\cap\{1..N\}:k=1..\kappa\}.$$

At every time step $t \in \N$, a part $S(t)$ of ${\bf X}$ 
is observed corresponding to the first $N(t) \in \N$ rows of ${\bf X}$, 
each of length $n_i(t),i\in 1..N(t) $, 
i.e. $$S(t)= \{ \x_1^t, \cdots \x_{N(t)}^t \}~\text{where}~\x_i^t := X_{1..n_i(t)}^i.$$

We assume that the number of samples, as well as 
the individual sample-lengths grow with time.  
That is, the length $n_i(t)$ of each sequence  $\x_i$ is  nondecreasing  and grows to infinity (as a function of  time $t$).
The number of sequences $N(t)$ also grows to infinity. 
Aside from these assumptions, the functions $N(t)$ and $n_i(t)$ are completely arbitrary.

An algorithm is called asymptotically consistent in the online setting, if, for every $N$ w.p.1 from some point on the clustering $C$ output by the algorithm coincides with the ground-truth on the first $N$ samples, i.e.\ $C|_{N}=\G|_N$.

It turns out that this setting admits a consistent clustering algorithm.
\begin{theorem}
 There exists an algorithm that is asymptotically consistent in the online setting, provided that the marginal distribution of  each sequence $\x_i, i\in\N$ is stationary ergodic. 
\end{theorem}
The proof of this theorem, along with the corresponding  algorithm, can be found in~\cite{Khaleghi:15clust}.

It is worth noting that the main challenge in constructing such an algorithm is the fact that, on every time step $t$, we do not know whether all of the $\kappa$ different distributions are already present, or the $N(t)$ are generated by fewer than $\kappa$ different distributions. The solution is based on a weighted average of clusterings, each constructed based on the first $N$ rows, with carefully selected weights.
\subsubsection{Clustering with respect to independence}
The clustering problem considered in the previous sections may be seen as clustering {\em with respect to distribution}:  putting together those and only those samples that are generated by the same distribution. Another way to look at clustering time series is grouping them with respect to (in)dependence.  Thus, the problem is as follows.
\begin{svgraybox}
  Given a set $S=(\x_1,\dots,\x_N)$ of samples,  it is required to find the finest partitioning $\{U_1,\dots,U_k\}$ of  $S$ into clusters  such that the clusters $U_1,\dots,U_k$ are mutually independent. 
\end{svgraybox}

The formal model is the same as in clustering with respect to distribution: the probability distribution is that on the space of infinite sequence of $N$-tuples~\eqref{eq:mat1}. However, in this setting we  require the joint distribution to be stationary ergodic, whereas before we only had to put this constraint on the martinal distribution of the samples.

What makes this problem very different from the previous one, and, in fact, from  the rest of the problems considered in the clustering literature, is that,  since mutual independence is the target, pairwise similarity measurements are of no use. Therefore,  traditional clustering algorithms are inapplicable, since they are based on calculating some distance between pairs of objects (in the case of the previous sections,  time-series samples)  $\x_i,\x_j$.  

Thus, to solve this problem we have to go back to the first principles and first consider what should we do if the joint distribution of all the samples is known. After that, it is instructive to consider i.i.d.\ samples, before turning to stationary ergodic distributions. 
While the detailed considerations of this problem takes us outside the scope of this volume, here it is worth mentioning in which cases a solution to this problem exists, and some ideas behind it. 

For stationary ergodic distributions a consistent algorithm can be constructed provided the correct number of clusters is known. The algorithm is based on calculating empirical estimates of the following measure of independence between groups of samples.  In the expression below,  $h()$ stands for Shannon entropy, and $[\cdot]^l$ is a quantization of the random variable in question to the cells of a partition similar to $B^{m,l}$ but finite.
\begin{definition}[sum-information] For stationary processes $x_1,\dots,x_k$ define the sum-information
 \begin{multline}\label{eq:smir}
  \I(\x_1,\dots,\x_N):=\sum_{m=1}^\infty {1\over m} w_m \sum_{l=1}^\infty  {1\over l} w_l
  \\ 
  \left(\sum_{i=1}^N h([X^i_{1..m}]^l)\right) - h ([X^1_{1..m}]^l,\dots,[X^N_{1..m}]^l).
 \end{multline}
\end{definition}
This quantity has certain similarities to the distributional distance: it is also a weighted sum of certain discrepancies between marginal distributions of growing dimension. However, instead of simple differences in probabilities, we are using entropy, and whereas before we were considering only pairs of random variables, here we have generalized this to groups of arbitrary sizes.
Note also that this is not an estimator but a theoretical quantity; to estimate it empirically, one replaces the probabilities $[X^i_{1..m}]^l$ with the corresponding frequencies. 

The details of the algorithms and proofs  can be found in \cite{Ryabko:17clin}.  It is worth noting that the online version of this problem (akin to the one considered in Section~\ref{s:clon})  so far remains unexplored.
\section{Change-point problems}\label{s:chp}
Change-point problems are concerned with sequences  in which the distribution of the data changes over time in an abrupt manner. The latter means that the sequence can be divided into segments, such that each segment is generated by a single time-series distribution, and between the segments the distributions are different. 

 It is  another
classical problem, with vast literature
on both parametric (see e.g. \cite{basseville:93})
and non-parametric (see e.g. \cite{brodsky:93}) methods  for solving
it. 
As usually in statistics, most literature deals with the case of i.i.d.\ data within each segment, with generalisations to dependent data reaching up to and including distributions with mixing  \cite{brodsky:93, Giraitis:95}. The important exception is the work \cite{Carlstein:93}, which considers stationary ergodic sequences. The latter work makes a further assumption that the single-dimensional marginals (of $X_i$) before and after the change point are different. As was shown in \cite{Ryabko:103s}, this assumption is not necessary; here, as in the preceding sections, we follow this latter approach.

Change-point problems can be roughly divided into {\em estimation} problems and {\em detection} problems. To better explain this, consider the case of a single change. 
A sample
$Z_1,\dots,Z_n$ is given, where, for a certain $\theta\in(0,1)$,  $Z_1,\dots,Z_{\lfloor n\theta\rfloor}$ are generated
according to some distribution $\rho_X$ and $Z_{{\lfloor n\theta\rfloor}+1},\dots,Z_n$
are generated according to some distribution $\rho_Y$. Change-point {\em estimation} is about finding the parameter $\theta$ (or, equivalently, the change point ${\lfloor n\theta\rfloor}$), knowing that it exists, that is, knowing that $\rho_Y\ne\rho_X$. On the other hand, detection problems are concerned with determining whether there is a change point in the first place, that is, finding out whether $\rho_X=\rho_Y$. Various formulations exist, mainly focusing on detecting the change quickly after it appears. 

Given the results of the preceding sections, it should be clear at this point that if all we know is that all the distributions in question are stationary and ergodic, then it is, in general, not possible to tell whether there is a change point in the sequence or not. Thus, we will be only concerned with change-point estimation problems. 

Another point that needs to be clarified is the {\em asymptotic regime} that we are using. We are working with a single sample of a fixed size, $n$,
$$Z_1,\dots,Z_{\lfloor n\theta\rfloor},Z_{{\lfloor n\theta\rfloor}+1},\dots,Z_n,$$
yet the statements will be about what happens when $n$ goes to infinity. In fact, we are talking about two samples whose lengths grow to infinity. If we imagine them being stuck together and each increasing in length to the right, then this would somehow make the change point obvious each time the length of sample to the left of it increments. This is why we are not considering an ``online'' setting where the samples would grow. Rather, we are considering only an ``offline'' version, where the sample is fixed. In this setting, saying that, for example, the estimate $\hat\theta$ approaches $\theta$ as $n$ grows to infinity simply means that for large enough $n$, $\hat \theta$ is arbitrarily close to $\theta$, and does not mean that the algorithm is dealing with samples of increasing sizes. 

An important constraint, which  is present in one way or another in all the change-point models, is on how far a change point can be from the boundaries of the samples. Indeed, if, say $Z_1$ is generated by one distribution but already $Z_2$ by another, so the change point occurs at time step 2, then hardly any algorithm can make any meaningful inference. A common way to tackle this is to require  the size of each segment (generated by a single distribution) to be linear in the length of the whole combined sample, $n$.  This is made explicit in the formulations we adopt, where we refer to change points as $\theta n$, and the goal is to estimate $\theta$. Moreover, given the fact that there are no speeds of convergence available for (stationary) ergodic   distributions, this requirement is essential, since the initial $o(n)$ part of any sample can be effectively arbitrary, whatever function one assumes in that $o(n)$.
Thus, we can state the following.
\begin{svgraybox}
 Consistent change-point estimation algorithms for stationary ergodic processes are only possible under the constraint that the length of each segment generated by a single distribution is linear in the total sample size $n$.
\end{svgraybox}

In this chapter we treat in detail  the case of a single change point. Extensions to multiple change points are given without proofs, referring the interested reader to the corresponding papers. However, it is worth noting that, in spite of the impossibility of discriminating between processes and thus detecting a change point, the case of an unknown number of change points is not entirely hopeless, and in fact in some cases admits a solution that does not require putting further restrictions on the distributions generating the data.

\subsection{Single change point}
The sample $Z=(Z_1,\dots,Z_{n})$ is the  concatenation of two
parts $X=(X_1,\dots,X_{\lfloor n\theta\rfloor})$ and $Y=( Y_1,\dots,Y_m)$, where $m=n-{\lfloor n\theta\rfloor}$,
so that $Z_i=X_i$ for $1\le i\le {\lfloor n\theta\rfloor}$ and $Z_{{\lfloor n\theta\rfloor}+j}=Y_{j}$ for $1\le
j\le m$. The samples $X$ and $Y$ are generated  by
two different stationary ergodic processes with alphabet $A$.
The distributions of the processes  are unknown. The value ${\lfloor n\theta\rfloor}$ is
called the \emph{change point}. Moreover, in this first  setting, we assume that $\theta$ is bounded away from 0 and from 1 with known upper and lower bounds:
 $\alpha n  < k
< \beta n$ for some known $0<\alpha\le\beta<1$ (for sufficiently large $n$). In the next setting we shall discuss how to get rid of this assumption.

It is required to estimate  $\theta$ (or, equivalently, the change point ${\lfloor n\theta\rfloor}$)
based on the sample~$Z$.

 For each $t$, $1\le t\le n$, denote  $U^t$  the sample
$(Z_1,\dots,Z_t)$ consisting of the first $t$ elements of the sample
$Z$, and denote
  $V^t$ the remainder $(Z_{t+1},\dots,Z_{n})$.
\begin{definition}[Change point estimator]
Define the change-point estimate $\hat \theta:  A^*\rightarrow(0,1)$ as follows:
$$
\hat \theta(X_1,\dots,X_n):={1\over n}\argmax_{t\in[\alpha n,n- \beta n]} \hat d(U^t,V^t).
$$
\end{definition}
The following theorem establishes asymptotic consistency of this estimator.

\begin{theorem}\label{th:chp}
 For the estimate $\hat \theta$ of the change point $\theta n$ we have
$$
\lim_{n\to\infty}|\hat \theta-\theta|=0\ \as
$$ where  $n$ is the size of the sample. 
\end{theorem}
\begin{proof} Denote $k:={\lfloor n\theta\rfloor}$. To prove the statement, we will show that, for every
$\gamma$, $0<\gamma<1$, with probability 1 the inequality $\hat d(U^t,V^t)<\hat d(X,Y)$ holds
for each $t$ such that  $\alpha k\le t<\gamma k$, possibly except for a finite number of times (in $n$). Thus we will
show that linear $\gamma$-underestimates occur only a finite number of
times, and for
overestimate it is analogous. Fix some $\gamma$, $0<\gamma<1$ and
$\epsilon>0$.  Let $J$ be big enough to have $\sum_{i=J}^\infty
w_i<\epsilon/2$ and
also big enough to have an index $j < J$  for which
$\rho_X(B_j)\ne\rho_Y(B_j)$.
Take $M_\epsilon\in\N$  large enough to have
$|\nu(Y,B_i)-\rho_Y(B_i)|\le\epsilon/2J$ for all $m>M_\epsilon$ and
for each $i$, $1\le i\le J$, and also to have $|B_i|/m<\epsilon/J$ for each $i$, $1\le i\le J$.
This is possible since empirical frequencies converge to the limiting probabilities
a.s.; note that  $M_\epsilon$ depends on  $Y_1,Y_2,\dots$ (cf. the proof of Lemma~\ref{th:dd}).
Find a  $K_\epsilon$ (that depends on $X$) such that for all
$k>K_\epsilon$  and for all $i$, $1\le i\le J$ we 
  have 
\begin{equation}\label{eq:fr}
|\nu(U^t,B_i)-\rho_X(B_i)|\le\epsilon/2J \text{ for each $t \in[\alpha n,\dots,k]$}
\end{equation}
 (this is possible simply because $\alpha n\to \infty$). 
Furthermore, we can select $K_\epsilon$ large enough to have
$$|\nu((X_s,X_{s+1},\dots,X_k),B_i)-\rho_X(B_i)|\le\epsilon/2J$$
 for each
$s\le\gamma k$: this follows from~(\ref{eq:fr}) and the identity 
$$\nu((X_s,X_{s+1},\dots,X_k)= \frac{k}{k-s}\nu((X_1,\dots,X_k)-\frac{s-1}{k-s}\nu(X_1,\dots,X_{s-1})+o(1).$$

So, for each $s\in[\alpha n,\gamma k]$ %
we  have
\begin{multline*}   \left|\nu(V^s,B_j) - \frac{(1-\gamma)k\rho_X(B_j)+m\rho_Y(B_j)}{(1-\gamma)k+m} \right|\\ \le
\Bigg|\frac{(1-\gamma)k\nu((X_s,\dots,X_k),B_j)+m\nu(Y,B_j)}{(1-\gamma)k+m} - \\
 \frac{(1-\gamma)k\rho_X(B_j)+m\rho_Y(B_j)}{(1-\gamma)k+m} \Bigg|
+ \frac{|B_j|}{m+\gamma k}\ \
 \le
    3\epsilon/J,
\end{multline*}
for $k>K_\epsilon$ and $m>M_\epsilon$ (from the definitions of $K_\epsilon$ and $M_\epsilon$).
Hence
\begin{multline*}
\left|\nu(X,B_j)-\nu(Y,B_j)\right|- \left|\nu(U^s,B_j) -\nu(V^s,B_j)\right|\\
 \ \ \ge \left|\nu(X,B_j)-\nu(Y,B_j)\right| \hfill
\\- \left|\nu(U^s,B_j) -
\frac{(1-\gamma)k\rho_X(B_j)+m\rho_Y(B_j)}{(1-\gamma)k+m} \right|-
3\epsilon/J\\
 \ \ \ge \left|\rho_X(B_j)-\rho_Y(B_j)\right| \hfill
\\ - \left|\rho_X(B_j) -
\frac{(1-\gamma)k\rho_X(B_j)+m\rho_Y(B_j)}{(1-\gamma)k+m} \right|-
4\epsilon/J \\ = \delta_j - 4\epsilon/J,
\end{multline*}
for some $\delta_j$ that depends only on $k/m$ and $\gamma$.
Summing over all $B_i$, $i\in\N$,  we get
$$
\hat d(X,Y)-\hat d(U^s,V^s)\ge w_j\delta_j-5\epsilon,
$$
for all $n$ such that $k>K_\epsilon$ and $m>M_\epsilon$, which is
positive for  small enough~$\epsilon$.
\end{proof}
\subsection{Multiple change points, known number of change points}
The following generalization is  considered in this section. First,  the number of change points is allowed to be arbitrary, though it still has to be known. Second, we get rid of the assumption that there is a known lower bound on the distance between a change point and the sequence boundaries (its start and its end), as well as between change points. 

The details of the algorithm and the proof of its consistency are omitted and can be found in~\cite{Khaleghi:15chp}.

The problem is as
 follows.
A sample
\begin{equation}\label{eq:xchp}
\x:=
Z_{1},\dots,Z_{\lfloor n\theta_1\rfloor}, 
Z_{\lfloor n\theta_1\rfloor+1},\dots,Z_{\lfloor n\theta_2\rfloor}, \dots,
Z_{\lfloor n\theta_{\kappa}\rfloor+1},\dots,Z_n
 \end{equation}
is given, which is  formed as the concatenation of  $\kappa+1$ non-overlapping segments, 
where $\kappa\in\N$ and  $0<\theta_1 <\dots<\theta_\kappa<1$.  
Each segment is generated by some unknown process distribution. 
The  distributions that generate every pair of consecutive segments 
are different.  
The parameters $\theta_k,~k=1..\kappa$ specifying the 
change points $\lfloor n\theta_k \rfloor$ are  unknown and have to be estimated. 
The distributions that generate the segments are unknown, but are assumed to be stationary ergodic. 
A formal probabilistic model for this process is via considering the matrix of random variables \eqref{eq:mat1}, where the marginal distribution of each row is stationary ergodic. The sample $\x$ is then formed by concatenating parts of these rows. 

Denote for convenience $\theta_0:=1$ and $\theta_{\kappa+1}:=n$ and define the minimal distance between change points as
\begin{equation}\label{eq:lmin}
\mingap:=\min_{k=1..\kappa+1} \theta_k-\theta_{k-1}.
\end{equation}
Let us first assume that there is a known lower bound $\lambda>0$ on this parameter: $\lambda<\mingap$.
Then, knowing this lower bound and the number of change points $\kappa$, one can construct a consistent algorithm as follows.

Break the whole sample $\x$ into short consecutive segments each of which cannot contain more than one change point (the actual algorithm, proposed in \cite{Khaleghi:15chp}, uses segments of length $n\mingap/3$). Find a candidate change point in each of the segments, using the single-change-point algorithm of the previous section. Then, select $\kappa$ of these candidate change-points that maximize the following scoring function. The scoring function  $\Delta_{\x}(a,b)$ takes an arbitrary segment $(a,b)$ in the sample and measures how close to each other (in the distributional distance)  its first and second halves are
\begin{equation}\label{defn:Delta}
\Delta_{\x}(a,b):= \hat{d} \left(Z_{a.. \lfloor \frac{a+b}{2}\rfloor},Z_{\lceil \frac{a+b}{2}\rceil..b} \right).
\end{equation} 
The reason the algorithm works is as follows. The single-change point estimates are consistent in the case there is exactly one change point in the segment they are applied to; this can be demonstrated in the same way the single change-point estimator was proven consistent in the previous section.
Next, the segments that do not contain any change point will see their score~\eqref{defn:Delta} converge to 0, while those that do contain a change point, to a non-zero constant. Since we know how many change points there are, it suffices to select the $\kappa$ highest-scoring ones.

The next step is to get rid of the requirement of a known bound $\lambda$ on $\mingap$. This is done by constructing a series of $\kappa$-tuples of change-point estimators, each for a different value of a candidate $\mingap$, which are then combined with carefully selected weights. 
This gives the following theorem. 
\begin{theorem}
 There exists an algorithm for finding $\kappa$ change points that is asymptotically consistent provided each segment is generated by a stationary ergodic distribution and $\kappa$ is known.
\end{theorem}

\begin{svgraybox}
 For stationary ergodic time series, asymptotically consistent estimation of  multiple change points whose number is known is possible  without any extra assumptions, besides that the length of each segment is linear in the sample size $n$.
\end{svgraybox}

\subsection{Unknown number of change points}
The result on impossibility of process discrimination (Section~\ref{s:hom}) implies that it is provably impossible to distinguish between the cases of 0 and 1 change point for stationary ergodic samples. Yet, it appears impractical to assume that the exact number of change points is given to an algorithm. Thus, a search for other, more constrained, formulations is warranted. Two such formulations are briefly considered here: providing an exhaustive list of change points, and the case of a known number of different distributions but an unknown number of change points. The details of the algorithms and proofs are left out, and can be found in the corresponding papers \cite{Khaleghi:12mchp,Khaleghi:14}. In both of these formulations we assume a known lower bound on the distance $\mingap$ between the change points~\eqref{eq:lmin}.

It is worth making a distinction with the related problem of clustering. In that problem, if the number of clusters $k$ is unknown, then all we can do is to resort to more restrictive assumptions on the process distributions (see Section~\ref{s:clunk}). On the other hand, for the change-point problem, it is still possible to get around the fact that the number of change points $k$ is unknown, while  only assuming that the process distributions are stationary ergodic. Specifically, one formulation that allows us to do it is assuming that the total number of distributions  is known (Section~\ref{s:knownr} below). Indeed, the number of distributions defines the number of clusters in the clustering problem, but, in change-point problems, still  allows the number of change points to be arbitrary.
\subsubsection{Listing change points}
Not knowing the number of change points, one could try to provide a {\em ranked list} of change points, that should include all the 
``true'' change points, and possibly also other, spurious,  points. The longest such a list could be is $n$, the size of the sample; or $\lceil 1/\lambda\rceil$ if we assume that the minimum distance between the change points is lower-bounded by $\lambda$.  Such a ranked list could be useful if we knew that the first $\kappa$ listed change points are the true change points, even if the rest of the listed points are extraneous.  It turns out that this is indeed achievable. The algorithm is very similar to the one of the preceding section, with~\eqref{defn:Delta}  used as a ranking  function and the single-change-point algorithm used to find candidate change points in each of the segments.
The result one can obtain is thus the following \cite{Khaleghi:12mchp}.
\begin{theorem}
 There exists an algorithm that, given a sample $\x$~\eqref{eq:xchp} generated by stationary ergodic distributions, provides a list of change-point candidates which has the property that, with probability 1 as $n$ goes to infinity, from some $n$ on its first $\kappa$ elements are within $o(n)$ of the change points $\theta_i$, $i=1..\kappa$.
\end{theorem}

\subsubsection{Known number of distributions, unknown number of change points}\label{s:knownr}
A sample with $\kappa$ change points can be, in general, generated by $\kappa+1$ different distributions. However, it can be generated by fewer distributions too, for example, by two distributions only irrespective of the value of $\kappa$. This formulation with the total number of distributions $r$ smaller than $\kappa+1$ may make sense in various applications. For example, imagine a text written by two authors each of which wrote many different parts of the text. Here the number of distributions is 2 and is known a priori, but the number of change points may be large and unknown.  

It turns out that if the number of change points is unknown, it is still possible to locate them, if the total number of distributions generating the segments is known. Here as well we assume a known lower-bound $\lambda$ on the minimal distance between change points. 
The algorithm starts by producing an exhaustive list of $1/\lambda$ change points with the algorithm of the previous section.  
It then clusters all the  resulting segments of the sample into $r$ clusters, where $r$ is the number of different distributions that is assumed given.
The clustering algorithm can be chosen to be that of Section~\ref{s:clusta}, with $r$ as the target number of clusters.
This result in the following statement.

\begin{theorem}
 There exists an algorithm that, given a sample $\x$~\eqref{eq:xchp} generated by $r$ different stationary ergodic distributions, the number $r$ and a lower-bound $\lambda$ on $\mingap$, 
provides an estimate $\hat \kappa\in[1..n]$ and a  list of change points estimates $\theta_i$, $i=1..\hat\kappa$ that are asymptotically consistent: $$\lim_{n\to\infty}\hat \kappa=\kappa$$ and $$|\hat\theta_i-\theta_i|=o(n)$$ for $i=1..\kappa.$
\end{theorem}

The details of the algorithm and proofs can be found in \cite{Khaleghi:14}.

In conclusion, we can formulate the following statement.
\begin{svgraybox}
For stationary ergodic distributions generating the data between change points whose number is unknown, it is possible to find the correct number  and provide consistent estimates of the change points if (and only if) the total number of different distributions  is known.
\end{svgraybox}

\chapter{Hypothesis testing}\label{ch:ht}
Given a sample $X_1,\dots,X_n$,  %
 we wish to decide whether it was generated by a distribution  belonging to a family $H_0$, versus it was generated by a  distribution 
belonging to a family $H_1$. 
As before, the only assumption we are willing to make about the the distribution generating the sample is that it is stationary ergodic.

In this chapter where we assume that $X_i$ are from a {\em finite  alphabet} $A$. Moreover, unlike in the previous chapters, in this one we shall delve a little deeper into the theory of stationary processes, and use some of its facts other than the simple convergence of frequencies. In particular, it will be of essence that the space $\S$ of stationary processes is compact with the topology of the distributional distance: a fact that holds for finite-alphabet processes with distance~\eqref{eq:ddisd} but not for real-valued processes with the distance~\eqref{eq:ddisr}.

The material of this chapter mainly follows \cite{Ryabko:121c,Ryabko:141u}.

\section{Introduction}
A test is a function that takes a sample and gives a  binary (possibly incorrect) answer:  the sample
was generated by a distribution from $H_0$ or from $H_1$.
 An answer $i\in\{0,1\}$ is correct if the sample is generated by a distribution that belongs to $H_i$, and otherwise the test is said to  make an error. It often makes sense to distinguish between two types of error, depending on which of the hypotheses holds true. Thus, we say that the test makes a {\em Type I} error if $H_0$ is true but the test says $H_1$ is true, and we say that the test makes {\em Type II error} if the opposite takes place: the test says $H_0$ while $H_1$ is true. 
Note that in case neither $H_0$ nor $H_1$ holds true  the output of the test may be arbitrary and we are not speaking about any kind of error; generally, one cannot say anything about the behaviour of the test in such a  case.

Here we are concerned with the general question of characterizing those pairs of $H_0$ and $H_1$ for which consistent tests exist.

 Several notions of consistency are considered.
 For two of these notions of consistency we find some necessary and some sufficient conditions for the existence of a consistent test, expressed
 in topological terms. The topology is that of distributional distance in the form~\eqref{eq:ddisd}.  For one notion of consistency, namely, for asymmetric  consistency,  the necessary and sufficient conditions coincide when $H_1$ is the 
 complement of $H_0$, thereby providing a complete characterization. This suggests that the topology of the distributional distance is indeed the right one to study these problems.  %

Each of the notions of consistency considered has been studied extensively (sometimes in slightly different formulations) for i.i.d.\ data. It is thus instructive to provide characterisations of those hypotheses for which consistent tests exist for this more restrictive model and see how it relates to the general case of stationary ergodic time series, which we do in this chapter whenever possible.

In the rest of this section we consider various examples of the problem of hypothesis testing that motivate studying it in the general form; we also introduce various notions of consistency used. In the next section, a simple example of hypothesis testing is considered in some detail, exposing various concepts used, including the notions of consistency, the topological criteria for consistency in simpler spaces and the role of the ergodic  decomposition.

\subsection{Motivation and examples}
Before introducing the definitions of consistency, let us give some examples motivating the general problem
in question. Most of these examples are classical problems studied in mathematical statistics and related fields,
mostly for i.i.d. data, with much literature devoted to each of them.
The classical Neyman-Pearson formulation  of the hypothesis testing problem is testing a simple hypothesis $H_0=\{\rho_0\}$ versus
a simple hypothesis $H_1=\{\rho_1\}$, where $\rho_0$ and $\rho_1$ are two  distributions that are  completely known.
A more complex   but more realistic  problem is when only 
one of the hypothesis is simple, $H_0=\{\rho_0\}$ but the alternative is general, for example, in our framework  $H_1$ could be  the set 
of all stationary ergodic processes that are different from $\rho_0$. 
This is the so-called goodness-of-fit  or identity-testing problem.
Here $\rho_0$ would typically be some specific distribution of interest, such as the Bernoulli i.i.d.\ distribution with equal probabilities of outcomes.

Generalizing the latter example is the class of hypothesis testing problems that can be described as  {\em model verification} problems. Suppose
we have some relatively simple (possibly parametric) set of assumptions, and we wish to test whether 
the process generating the given sample satisfies this assumptions. As an example, $H_0$ can be the
set of all $k$-order Markov processes (fixed $k\in\N$) and $H_1$ is the set of all stationary ergodic 
processes that do not belong to  $H_0$; one may also wish to consider more restrictive alternatives, for example 
$H_1$ is the set of all $k'$-order Markov processes for some $k'>k$. Of course, instead of Markov processes
one can consider other models, e.g. hidden Markov processes. A similar problem 
is that of testing that the process has entropy less than some given $\epsilon$ versus its entropy exceeds $\epsilon$, 
or versus its entropy is greater than $\epsilon+\delta$ for some positive $\delta$. 

Yet another type of hypothesis testing problems concerns {\em property testing}. Suppose we are given 
two samples, generated independently of each other by stationary ergodic distributions, and we wish to test 
the hypothesis that they are independent versus they are not independent. Or, that they are generated by the same
process versus they are generated by different processes. 

In all the considered cases, when the hypothesis testing problem turns out to be too difficult (i.e. there 
is no consistent test for the chosen notion of consistency) for the case of stationary ergodic processes, one may wish to restrict either $H_0$, $H_1$
or both $H_0$ and $H_1$ to some smaller class of processes. Thus, one may wish to test the hypothesis of independence when, 
for example, both processes are known to have finite memory, but the alternative is allowed to be general: the complement of the set $H_0$ to the set  $\mathcal E$ of stationary ergodic processes (on pairs). 

\section{Types of consistency}\label{s:cons}

There are different types of consistency of tests, corresponding to how strong a guarantee one wishes
to have on the probability of error. Three notions of consistency are considered here: uniform, asymmetric (or $\alpha$-level),   and asymptotic consistency. They represent different trade-offs between the strength of the guarantees one can obtain and the generality of hypotheses pairs for which consistent tests exist. 

\subsection{Uniform consistency}
We start with what appears the strongest notion, uniform consistency. 
It requires  both probabilities of error to be uniformly bounded.
 More precisely,  {\em uniform consistency} requires that for each $\alpha$ there exist a sample size $n$ such that probability of error 
 is upper-bounded by $\alpha$ for samples longer than $n$.
\begin{definition}[uniform consistency]
A test $\phi$ is called uniformly consistent if for every $\alpha$ there is an $n_\alpha\in\N$ such that
for every $n\ge n_\alpha$ the probability of error on a sample of size $n$ is less than $\alpha$: $\rho(X\in A^n: \phi(X)=i) <\alpha$ for every $\rho\in H_{1-i}$ and every $i\in\{0,1\}$.
\end{definition}
This notion of consistency has been extensively studied in the algorithms community for i.i.d.\ data under a slightly different formulation:
the probability of each error is required to be bounded by a fixed number, typically 1/3, and the problem is to 
find minimal sample sizes necessary to achieve this error. The interpretation is that if one can get 1/3 probability 
of error then one can make it arbitrary small by taking more (independent) samples; see, for example, \cite{Goldreich:98,Batu:01,Batu:04,Guha:06}. %
The definition above is adapted for dependent data.

For i.i.d.\ samples it is easy to establish 
a criterion for the existence of a consistent test: there exists a uniformly consistent test if and only if $H_i, i=\{0,1\}$ are contained in closed non-overlapping sets.
Here the topology is just that of the  Euclidean distance  on the space of parameters defining the distributions over $A$.
Indeed, to see that the condition is necessary, it is enough to notice that the sets of distributions $\rho$ satisfying $\rho(B)\le\alpha$
are closed  for any fixed $B\in A^*$ and $\alpha\in[0,1]$, in particular for $B:=\{z_{1..n}:\phi(z_{1..n})=0\}$. 
On the other hand, to construct a test it is enough to take a neighbourhood over (say) $H_0$ of radius that slowly decreases with $n$:
for large enough $n$ the neighbourhood will not intersect $H_1$ (since both sets are closed), and one can use concentration of measure
results for i.i.d.\ distributions to show that if the radius decreases slow enough then the test is consistent. From this description 
it is clear that the some generalizations to processes with mixing are possible.  See also \cite{Csiszar:04} for related results (for i.i.d.\ data).

\subsection{Asymmetric consistency}
The next notion of consistency is the classical one used in mathematical statistics (e.g.,\cite{Lehmann:86,Kendall:61}): 
the probability of Type I error is fixed at the given level $\alpha$, and the probability of Type II error 
goes to 0.
This definition is well-suited for pair of hypotheses  that are by nature asymmetric, such as singleton $H_0$, or 
hypotheses where $H_1$ is the complement to $H_0$, for example, ``the distribution belongs to a given parametric model'' versus ''it is 
stationary ergodic but not in the model,''  or the examples considered in this work: ``distributions generating a pair of samples are independent'' versus they 
are not, or ``distributions are the same'' versus they are not. 
The definition is as follows.

\begin{definition}[Asymmetric  consistency]
 Call $\alpha$-level  test $\psi^\alpha, \alpha\in(0,1)$ asymmetrically consistent as a test of $H_0$ against $H_1$ if:
\begin{itemize}
\item[(i)] The probability of Type I error is always bounded by $\alpha$: $\rho\{X\in A^n: \psi^\alpha(X)=1\}\le\alpha$ for every $\rho\in H_0$,
every $n\in\N$ and every $\alpha\in(0,1)$, and
\item[(ii)]  Type II error is made not more than a finite number of times with probability 1: 
 $\rho(\lim_{n\rightarrow\infty} \psi^\alpha(X_{1..n})=1)=1$ %
for every  %
$\rho\in H_1$ and every $\alpha\in(0,1)$. 
\end{itemize}
\end{definition}

Similar to the  case of uniform consistency, here it is easy to see what is the criterion for the i.i.d.\ samples.
There exists an asymmetrically  consistent  test if and only if $\cl{H_0}$ does not intersect $H_1$; see the next section for a more detailed explanation, and also  \cite{Csiszar:04} for this and related results.

\subsection{Asymptotic consistency}
Finally, what appears to be the weakest notion of consistency is perhaps the simplest to formulate:
the error (of each type) has to be made finitely many times w.p.1. 
\begin{definition}[asymptotic consistency] 
A test $\phi$ is called uniformly consistent if for every $\rho\in H_i$, $i=0,1$  we have $\lim_{n\to\infty}\phi(z_{1..n})=i$ $\rho$-a.s.
\end{definition}
This weakest notion of consistency gives strongest negative results, which is why we used it in the Section~\ref{s:hom} to show that there is no consistent test for homogeneity (process discrimination).

For real-valued i.i.d.\ samples, where the hypotheses are formulated about the means of the distributions,  this notion has been studied in \cite{Dembo:94}.
For the case of distributions with finite $p$ moments with $p>1$ the  following criterion is obtained: there exists an asymptotically 
consistent test if and only if $H_0$ and $H_1$ are contained in disjoint $F_\sigma$ sets. (A set is $F_\sigma$ if it is a countable 
union of closed sets.) It can be seen that the same criterion holds in our case (finite-valued distributions) if the samples are i.i.d.

This notion of consistency has been given considerable attention in the time-series literature, perhaps because it is rather 
weak and thus appears more suited for time-series analysis. In particular, some  specific hypotheses have been studied in \cite{Ornstein:90,Morvai:05}. 
For the general case of stationary ergodic distributions, \cite{Nobel:06} obtains a generalization of the results of \cite{Dembo:94}, providing some sufficient conditions for the existence of a consistent test for real-valued processes, in terms of the topology of weak convergence.

\subsection{Other notions of consistency}
Many other notions of consistency exist in statistics and related fields.
 For example,   a variation on the notion of asymmetric consistency common in the literature  is requiring   the probability of Type~I error to be bounded by $\alpha$ only in asymptotic.
 Most of other notions of consistency  are focussed on speeds of convergence and thus are of little interest in our context.  
For example, one can require the probability of error (of each type) to decrease exponentially fast; see~\cite{Csiszar:04} for some characterisations.

\section{One example that explains hypotheses testing}\label{s:one}
Let us consider a rather simple example that illustrates various concepts used and difficulties encountered.
The example will be that of homogeneity testing (or process discrimination) for binary-valued ($A=\{0,1\}$) processes; we will consider i.i.d.\ processes and  Markov chains, in addition to stationary ergodic distributions. For the i.i.d.\ case, it is easy to find a e topological characterisation of those hypotheses for which consistent tests exist, so we do this for illustrative purposes.  The example hypothesis considered here, homogeneity testing,  is the problem we have addressed in  Section~\ref{s:hom} for asymptotic consistency in the general case. Here the main focus is on a stronger notion of consistency, namely asymmetric consistency, and on simpler processes. The goal is to illustrate the topological conditions that characterize the existence of consistent tests. The Markov case already shows why ergodic decomposition plays such an important role in finding the criteria for the existence of tests.

\subsection{Bernoulli i.i.d.\ processes}
Before considering dependent time series, let us see what would be the criterion for the existence of an asymmetrically consistent test for i.i.d.\ data, and apply it to our example of homogeneity testing. 

Thus,  we are speaking about Bernoulli distributions. Each such distribution $\rho$  can be identified with the parameter $\rho(X_1=0)\in[0,1]$, and each hypothesis $H_i$ with a subset of the parameter space $[0,1]$.
Recall that  a test $\phi_\alpha(\x)$, which receives an additional parameter $\alpha\in(0,1)$,  is said to be asymmetrically consistent, if, for every sample size $n$ the and every probability of Type I error (that is, error under $H_0$) is upper-bounded by $\alpha$, while the probability of Type II error (error under $H_1$) goes to 0.  
It is easy to see that there exists an asymmetrically  consistent  test if and only if $\cl{H_0}$ does not intersect $H_1$. 
Here the topology is just that of the  Euclidean distance  on the parameter space. 
Indeed, to see that the condition is necessary, it is enough to notice that the sets of distributions $\rho$ satisfying $\rho(B)\le\alpha$
are closed  for any fixed $B\in A^*$ and $\alpha\in[0,1]$, and in particular for $B:=\{z_{1..n}:\phi(z_{1..n})=1\}$. Thus, if, for the given sample size $n$, the probability that the test says  $H_1$  is upper-bounded by $\alpha$ for every $\rho\in H_0$ (Type I error) then the same holds for every $\rho\in\cl{H_0}$.
We have shown that it is necessary for $H_0$ to be closed in order for an asymmetrically consistent test against the complement to exist. To show sufficiency, we need to construct a test for an arbitrary closed $H_0$. To do so, consider a closed set $H_0$  and the closed set $C\subset[0,1]$ of parameters that defines it. Take a sequence of neighbourhoods $C_n$ over  $C$ of such radii that, for every $n$ and $\rho\in H_0$, the probability of samples of size $n$ that the frequency of 0 falls into $C_n$ equals $\alpha$. Note that the radius of these neighbourhoods  decreases with $n$ (because of the law of large numbers), which means that for every distribution $\rho\in H_1$ there is a large enough $n$ such that (the parameter that defines) $\rho$ is outside $C_n$. This implies that the Type II error goes to 0. 

The hypothesis of homogeneity is formulated for $\x_1,\x_2$ and states that their distributions $\rho_1,\rho_2$ are equal. Thus, we are speaking about distributions on pairs of samples (which, for the sake of simplicity, we consider independent). For Bernoulli distributions, this is a two-parameter space $[0,1]^2$. The hypothesis $H_0$ is the diagonal $\{(x,x): x\in[0,1]\}$, which is of course closed, and so a consistent test exists. 
Similarly, for uniform consistency the criterion is that $\cl{H_0}\cap \cl{H_1}=\emptyset$. Thus, there is no uniformly consistent test for homogeneity, and, more generally, there is no uniformly consistent test for any $H_0$ against its complement.  If    we want to have a uniformly consistent test for homogeneity, we need to change the alternative hypothesis $H_1$. For example, change $H_1$ to ``the distributions differ by at least $\epsilon$.''        This ensures the existence of a uniformly consistent test at the cost of  creating  an $\epsilon$-buffer zone between $H_0$ and $H_1$, in which, in general, we cannot say anything about the behaviour of a test.                                                                                                                                                                                                                                                                                                                                                                      

\subsection{Markov chains}                                                                                                                                                                                                                                                                                                                                                                               Moving on to the case of two-state Markov chains,  we have now two $[0,1]$-valued  parameters: the probabilities to change the state. As before, the state space is binary: $A=\{0,1\}$. Let us try to guess that the criterion for the existence of an asymmetrically consistent test is the same as in the i.i.d.\ case: there exists a consistent test iff $\cl{H_0}\cap H_1=\emptyset$, with the Euclidean topology of the parameter space,  and let us look what it gives for the hypothesis of homogeneity.
Consider a specific set of Markov chains, call it $m_\epsilon$. These are defined so  that the probability to change a state (from 0 to 1 as well as from 1 to 0) for the chain $m_\epsilon$ is $\epsilon$, and the initial distribution is given by $m_\epsilon(X_1=0)=1/2$. When $\epsilon$ goes to $0$, the limit  of $m_\epsilon$ (in the space $[0,1]^2$ of parameters) is $m_0$. The latter is a stationary distribution which is a mixture of two Dirac distributions $\delta_0$ and $\delta_1$: one concentrated on the sequence of 0s and the other on the sequence of 1s. This is the {\em ergodic decomposition} of $m_0$:  $m_0=1/2(\delta_0+\delta_1)$. Note that $m_\epsilon$ for $\epsilon>0$ are stationary and ergodic, but $m_0$ is stationary but not ergodic. And here lies the source of the trouble. For the hypothesis of homogeneity, consider the pair of distributions $(m_\epsilon,m_\epsilon)$. When $\epsilon\to0$, the limit is $(m_0,m_0)$, which is the mixture 
$$1/4((\delta_0,\delta_0)+(\delta_1,\delta_1)+(\delta_0,\delta_1)+(\delta_1,\delta_0)).$$
Call this mixture $W_0$.
Note that, under the distribution $(m_0,m_0)$, with probability 1/2 we observe two  different sequences, one is all 0s and the other all 1s. In other words, under the ergodic decomposition $W_0$ of $(m_0,m_0)$, with probability 1/2 we observe two different distributions, either $(\delta_0,\delta_1)$ or $(\delta_1,\delta_0)$, so that $W_0(H_1)=1/2$. Nonetheless, the distribution $(m_0,m_0)$ itself is of course in $H_0$. 

Let us now demonstrate that there is no asymmetrically consistent test for $H_0$ against its complement to the set of Markov chain distributions.
As in the i.i.d.\ case,  the sets of distributions $\rho$ satisfying $\rho(B)\le\alpha$
are closed  for any fixed $B\in A^*$ and $\alpha\in[0,1]$, and in particular for $B:=\{z_{1..n}:\phi(z_{1..n})=0\}$. Thus, for any test $\phi$ and any given sample size $n$, if the sets $\{(X_{1..n},Y_{1..n})\in (A^n)^2:\phi_\alpha(X_{1..n}=1)\}$ on which the test says $H_1$ (makes Type I error) have probability at most $\alpha$ with respect to every $(m_\epsilon,m_\epsilon)$ for $\epsilon>0$, then they also have probability at most $\alpha$ under the distribution $(m_0,m_0)$. The latter distribution, however, is concentrated on four pairs of $n$-tuples $(000..0,111..1), (111..1,111..1)(111..1,000..0)(000..0,000..0)$. This means that for $\alpha<1/4$ the test must say that the distributions are the same when presented with at least one of the 
pairs of samples $(000..0,111..1)$ or $(111..1,000..0)$. Since this happens for every $n$, we conclude that any such test is inconsistent: its Type II error does not go to 0: it is at least 1/4 for infinitely many $n$ under at least one of  the distributions $(\delta_0,\delta_1)$ or $(\delta_1,\delta_0)$.

Thus, we have shown that there is no asymmetrically consistent test for homogeneity for (stationary ergodic) Markov chains. The reason for this is that, while the set $H_0$ is closed, it is not closed under ergodic decompositions. Specifically, there exists a distribution $\rho\in  H_0$ (namely, $\rho=(m_0,m_0)$), whose ergodic decomposition $W_0$ is such that $W_0(H_1)=1/2$. Ergodic decompositions of the limit points of $H_0$  is what we need to take care of in the general case of stationary ergodic distributions.

As the last word about homogeneity testing for Markov chains, let us note that, unlike for stationary ergodic distributions, there exists an {\em asymptotically} consistent test for this hypothesis for this set of processes. Indeed, ergodic Markov chains mix exponentially fast (e.g., \cite{hernandez:03}), which is enough to construct a test, considering sets around $H_0$ that shrink sufficiently slowly. An example of such an algorithm for the more general problem of clustering  distributions with mixing can be found in \cite{Khaleghi:15clust}.

\subsection{Stationary ergodic processes}
Finally, let us pass to the general case of stationary ergodic distributions. The topology of the distributional distance that we work with is a direct generalisation of the Euclidean topology of the parameter spaces on the Bernoulli and Markov distributions that we considered. In fact, the topology induced by the distributional distance on these parameter spaces is exactly the same.

As we have seen in the Markov case, the main problem is with the limit points of $H_0$ and their ergodic decompositions. More generally, while the set $\S$ of stationary processes is closed in the topology of the distributional distance, the set $\mathcal E$ of stationary ergodic distributions is not (its closure is $\S$). 
This parallels the situation with Markov chains: the closure of the set of stationary ergodic Markov chains is the set of all stationary Markov chains.

For the case of asymmetric consistency for stationary ergodic processes, the pinnacle result presented in this chapter is the following criterion: there exists an asymmetrically consistent test of $H_0\subset\mathcal E$ against its complement $H_1=\mathcal E\backslash H_0$ if and only if $H_0$ has probability 1 with respect to the ergodic decomposition of every process in the closure of ${H_0}$. This is a corollary of the more general result presented in this chapter for the case when $H_0$ is not necessarily the complement of $H_1$; however the condition only becomes ``if and only if'' in the case of the complement.  This result can be directly applied to the hypothesis of homogeneity testing to show that there is no asymmetrically consistent test against its complement: indeed, the proof that $H_0$ is not closed under taking ergodic decompositions is by the Markov example of the previous subsection. %

\section{Topological characterizations}\label{s:main}
 In this section we formulate our criteria for the existence of consistent tests, and 
 give constructions of the tests which are consistent if and only if consistent tests exist.

 These constructions are not exactly algorithms, since one can hardly talk about algorithms
 whose input is an arbitrary set of distributions. However, the tests specify what should
 be estimated and how the decision should be made. Therefore, we provide  procedures that work if anything works at all;
 turning them into  efficient algorithms for specific problems is an interesting direction for further research.

The tests presented below are based on { empirical estimates of  the distributional distance}.
We shall first generalize this to measure the distance between a sample and a set of distributions (a hypothesis), rather than a single distribution or another samples.

 For a sample $X_{1..n}\in A^n$ and a hypothesis $H\subset\mathcal E$ define 
$$
\hat d(X_{1..n},H)=\inf_{\rho\in H} \hat d(X_{1..n},\rho).
$$

For $H\subset\S$, denote $\cl{H}$ the closure of $H$ with respect to the topology of $d$.

\subsection{Uniform testing}
For $H_0,H_1\subset\S$, the {\em uniform test} $\phi_{H_0,H_1}$ is constructed   as follows. For each $n\in\N$ let
\begin{equation}
 \phi_{H_0,H_1}(X_{1..n})\\:=\left\{\begin{array}{ll} 0& \text{ if }  \hat d(X_{1..n},\cl{H_0}\cap\mathcal E)<\hat d(X_{1..n},\cl{H_1}\cap\mathcal E), \\ 1 & \text{ otherwise.} \end{array}\right.
\end{equation}
Since the set $\S$ is a complete separable  metric space, it is easy to see that 
the function $\phi_{H_0,H_1}(X_{1..n})$ is measurable provided $\cl{H_0}$ is measurable.

\begin{theorem}[uniform testing]\label{th:uni}
 Let $H_0, H_1$ be measurable subsets of $\mathcal E$.   If  %
$W_\rho(H_i)=1$ for every $\rho\in \cl{H_i}$
then 
the test $\phi_{H_0,H_1}$ is uniformly consistent.
Conversely, if there exists a uniformly consistent test for $H_0$ against $H_1$ then 
$W_\rho(H_{1-i})=0$ for any $\rho\in \cl{H_i}$.
\end{theorem}

The proof is deferred to section~\ref{s:proofsht}.

The following corollary, which is easy to see already for i.i.d.\ distributions (see Section~\ref{s:one}), for the general case is an immediate consequence of the second statement of the theorem above. 

\begin{corollary}\label{th:nouni}
 There is no uniformly consistent test for any  hypothesis $H_0$ against its complement $\mathcal E\setminus H_0$ unless one of these hypotheses is empty.
\end{corollary}

\subsection{Asymmetric testing}

Construct the {\em asymmetric test}  $\psi_{H_0,H_1}^\alpha, \alpha\in(0,1)$ as follows.
For each  $n\in\N$, $\delta>0$ and $H\subset\mathcal E$ define the neighbourhood  $b^n_\delta(H)$  of $n$-tuples around $H$ as 
  $$b^n_\delta(H):=\{X\in A^n: \hat d(X,H)\le\delta\}.$$
Moreover, let 
$$
\gamma_n(H,\theta):= \inf \{\delta: \inf_{\rho\in H} \rho(b^n_\delta(H))\ge\theta\}
$$ be the smallest radius of a neighbourhood around $H$ that has probability not less than $\theta$ with respect to any process in $H$, and let $C^n(H,\theta):=b^n_{\gamma_n(H,\theta)}(H)$ be the neighbourhood of this radius.
Define  
$$
\psi^\alpha_{H_0,H_1}(X_{1..n}):=\left\{\begin{array}{ll} 0& \text{ if }  X_{1..n}\in C^n(\cl{H_0}\cap\mathcal E,1-\alpha), \\ 1& \text{ otherwise.} \end{array}\right.
$$
Again, it is easy to see that 
the function $\phi_{H_0,H_1}(X_{1..n})$ is measurable, since the set $\S$ is separable. %

\begin{theorem}\label{th:asym}
 Let $H_0, H_1$ be measurable subsets of $\mathcal E$. 
If $W_\rho(H_0)=1$ for every $\rho\in \cl{H_0}$ then the test $\psi^\alpha_{H_0, H_1}$ is  asymmetrically consistent. 
Conversely, if there is an asymmetrically consistent test for $H_0$ against $H_1$ then $W_\rho(H_1)=0$ for any $\rho\in \cl{H_0}$.
\end{theorem}
For the case when $H_1$ is the complement of $H_0$ the necessary and sufficient conditions of Theorem~\ref{th:asym} coincide and
give the following criterion.

\begin{corollary}\label{th:asymc}
 Let $H_0\subset\mathcal E$ be measurable and let $H_1= \mathcal E\backslash H_0$. The following  statements are equivalent:
\begin{itemize}
 \item[(i)] There exists an asymmetrically consistent test for $H_0$ against $H_1$.
 \item[(ii)]\ The test $\psi^\alpha_{H_0, H_1}$ is  asymmetrically consistent. 
 \item[(iii)]\ \ The set $H_1$ has probability 0 with respect to the ergodic decomposition of every $\rho$ in the closure of $H_0$: 
              $W_\rho(H_1)=0$ for each $\rho\in\cl{H_0}$.
\end{itemize}
\end{corollary}
\begin{svgraybox}
 There exists an asymmetrically ($\alpha$-level) consistent test for a hypothesis $H_0\subset\mathcal E$ against its complement $\mathcal E\setminus H_0$ if and only if $H_0$ is closed and closed under taking ergodic decompositions, in the sense that $W_\rho(H_0)=1$ for every $\rho$ in the closure of ${H_0}$.
\end{svgraybox}

\section{Proofs}\label{s:proofsht}
In the proofs, we  often omit the subscript $H_0,H_1$ from $\psi^\alpha_{H_0,H_1}$ when it can cause no confusion.

The proofs  use the following lemmas.

\begin{lemma}[smooth probabilities of deviation]\label{th:count}
 Let $m>2k>2$, $\rho\in\S$, $H\subset\S$, and $\epsilon>0$. Then
\begin{equation}\label{eq:count1} 
\rho(\hat d(X_{1..m},H)\ge\epsilon)  
  \le 2\epsilon'^{-1} \rho(\hat d(X_{1..k},H)\ge \epsilon'),
\end{equation}
where $\epsilon':=\epsilon - \frac{2k}{m-k+1} - t_k$ with $t_k$ being the sum of  all the weights of tuples longer than $k$ in the definition of $d$: $t_k:=\sum_{i:|B_i|>k}w_i$. Further,
\begin{equation}\label{eq:count2} 
\rho(\hat d(X_{1..m}, H)\le\epsilon)
  \le 2\rho\left(\hat d(X_{1..k},H)\le \frac{m}{m-k+1}2\epsilon + \frac{4k}{m-k+1}\right).
\end{equation}
\end{lemma}
The meaning of this lemma is as follows.  For any word $X_{1..m}$, if it is far away from (or close to)
a given distribution $\mu$ (in the empirical distributional distance), then some of its shorter subwords $X_{i..i+k}$ are far from (close to) $\mu$ too.
 In other words, for a stationary distribution $\mu$,  it cannot happen that  a small 
 sample is likely to be close to $\mu$, but a larger sample is likely to be far.

\begin{proof}
Let $B$ be  a tuple such that $|B|<k$ and $X_{1..m}\in A^m$ be any sample of size $m>1$. 
The number of occurrences of $B$ in $X$ can be bounded by the number of occurrences of $B$ in subwords of $X$ of length $k$ as follows:
\begin{multline*}%
\#(X_{1..m},B) \le   \frac{1}{k-|B|+1}\sum_{i=1}^{m-k+1}\#(X_{i..i+k-1},B) +  2k \\ =\sum_{i=1}^{m-k+1}\nu(X_{i..i+k-1},B)+  2k.
\end{multline*}
Indeed, summing over $i=1..m-k$ the number of occurrences of $B$ in all $X_{i..i+k-1}$ we count each
occurrence of $B$ exactly $k-|B|+1$ times, except for those that occur in the first and last $k$ symbols.
Dividing by $m-|B|+1$, and using the definition (\ref{eq:freq}), %
 we obtain
\begin{equation}\label{eq:ff}
\nu(X_{1..m},B)  \le \frac{1}{m-|B|+1}\left(\sum_{i=1}^{m-k+1}\nu(X_{i..i+k-1},B)| + 2k\right).
\end{equation}
Summing over all $B$, for any $\mu$, we get
\begin{equation}\label{eq:fracs}
\hat d(X_{1..m},\mu)  \le \frac{1}{m-k+1} \sum_{i=1}^{m-k+1}\hat d(X_{i..i+n-1},\mu) + \frac{2k}{m-k+1}+t_k,
\end{equation}
where in the right-hand side $t_k$ corresponds to all the summands in the left-hand side for which $|B|>k$, where for the rest of the summands we used $|B|\le k$.
Since this holds for any $\mu$, %
we conclude that
\begin{multline}\label{eq:ne}
\hat d(X_{1..m}, H)  \le \frac{1}{m-k+1} \left(\sum_{i=1}^{m-k+1}\hat d(X_{i..i+k-1}, H)\right)  + \frac{2k}{m-k+1}+t_k.
\end{multline}
Note that the $\hat d(X_{i..i+k-1}, H)\in[0,1]$. Therefore, for the average in the r.h.s.\ of~\eqref{eq:ne} to be 
larger than $\epsilon'$, at least $(\epsilon'/2) (m-k+1)$ summands have to be larger than $\epsilon'/2$.

Using stationarity, we can conclude
$$
\rho\left(\hat d(X_{1..k},H)\ge\epsilon'\right)\ge (\epsilon'/2) \rho\left(\hat d(X_{1..m}, H)\ge\epsilon\right),
$$
proving~(\ref{eq:count1}). The second statement can be proven similarly; indeed, analogously to~(\ref{eq:ff}) we have
 \begin{multline*}
\nu(X_{1..m},B)  \ge \frac{1}{m-|B|+1}\sum_{i=1}^{m-k+1}\nu(X_{i..i+k-1},B)-  \frac{2k}{m-|B|+1} \\ \ge \frac{1}{m-k+1}\left(\frac{m-k+1}{m}\sum_{i=1}^{m-k+1}\nu(X_{i..i+k-1},B)\right) -  \frac{2k}{m},
\end{multline*} where we have used $|B|\ge 1$. Summing over different $B$,  we obtain (similar to~(\ref{eq:fracs})),
 \begin{equation}\label{eq:ne2}
\hat d(X_{1..m},\mu)  \ge \frac{1}{m-k+1} \sum_{i=1}^{m-k+1}\frac{m-k+1}{m}\hat d_k(X_{i..i+n-1},\mu) - \frac{2k}{m}
\end{equation}(since the frequencies are non-negative, there is no $t_n$ term here).
For the average in~\eqref{eq:ne2} to be smaller than $\epsilon$, at least half of the summands must be smaller than $2\epsilon$.
Using stationarity of $\rho$, this implies~(\ref{eq:count2}).
\end{proof}

\begin{lemma}\label{th:cont} Let $\rho_k\in\S$, $k\in\N$ be a sequence of processes that converges to a process $\rho_*$.
Then, for any $T\in A^*$ and $\epsilon>0$ if $\rho_k(T)>\epsilon$ for infinitely many indices $k$, then
$\rho_*(T)\ge\epsilon$
\end{lemma}
\begin{proof}
 The statement follows from the fact that  $\rho(T)$ is continuous as a function of~$\rho$. 
\end{proof}

\smallskip

\begin{proof}[of Theorem~\ref{th:asym}.]
To establish the first statement of Theorem~\ref{th:asym}, we have to show that the family of tests $\psi^\alpha$ is  consistent.
By construction, for any $\rho\in\cl{H_0}\cap\mathcal E$ we have $\rho(\psi^\alpha(X_{1..n})=1)\le\alpha$. 

To prove the consistency of $\psi$, it remains to show that 
$$
\xi(\lim_{n\to\infty}\psi^\alpha(X_{1..n})=1)
$$  for any $\xi\in H_1$ and $\alpha>0$.
To do this, fix any $\xi\in H_1$ and let 
$$
\Delta:=d(\xi,\cl{H_0}):=\inf_{\rho\in \cl{H_0}\cap\mathcal E} d(\xi,\rho).
$$ Since $\xi\notin\cl{H_0}$, we have $\Delta>0$. Suppose that there exists an $\alpha>0$, such that, for infinitely many $n$, some samples from the $\Delta/2$-neighbourhood of $n$-samples
around $\xi$ are sorted as $H_0$ by $\psi$, that is, $C^n(\cl{H_0}\cap\mathcal E, 1-\alpha)\cap b_{\Delta/2}^n(\xi)\ne\emptyset$. Then
for these $n$ we have $\gamma_n(\cl{H_0}\cap\mathcal E,1-\alpha)\ge\Delta/2$. 

This means that there exists an increasing sequence $n_m,m\in\N$, and a 
sequence $\rho_m\in \cl{H_0}$, $m\in\N$, such that %
  $$\rho_{m}(\hat d(X_{1..n_{m}},\cl{H_0}\cap\mathcal E)> \Delta/2)>\alpha.$$
Using Lemma~\ref{th:count}, (\ref{eq:count1}) (with $\rho=\rho_{m}$, $m=n_m$, $k=n_k$, and $H=\cl{H_0}$),  and taking  $k$  large enough to have $t_{n_k}<\Delta/4$,
 for every $m$  large enough to have $\frac{2n_k}{n_m-n_k+1}<\Delta/4$, we obtain 
\begin{multline}\label{eq:down-}
8\Delta^{-1}\rho_{m}\left(\hat d(X_{1..n_k},\cl{H_0})  \ge\Delta/4\right) \\
 \ge  \rho_{m}\left(\hat d(X_{1..n_m},\cl{H_0})\ge\Delta/2\right) > \alpha. 
\end{multline}
Thus, 
\begin{equation}\label{eq:new}
 \rho_{m}(b^{n_k}_{\Delta/4}(\cl{H_0}\cap\mathcal E))<1-\alpha\Delta/8.
\end{equation}

Since the set $\cl{H_0}$ is compact (as a closed subset of a compact set $\S$),  we may assume (passing to 
a subsequence, if necessary) that  $\rho_m$  converges
to a certain $\rho_*\in\cl{H_0}$.
Since~\eqref{eq:new} this holds for infinitely many $m$, 
using Lemma~\ref{th:cont} (with $T=b^{n_k}_{\Delta/4}(\cl{H_0}\cap\mathcal E)$) we conclude that
 $$\rho_*(b^{n_k}_{\Delta/4}(\cl{H_0}\cap\mathcal E))\le1-\Delta\alpha/8.$$
Since the latter inequality holds for infinitely many indices $k$ we also have 
$$
\rho_*(\limsup_{n\rightarrow\infty}\hat d(X_{1..n},\cl{H_0}\cap\mathcal E)>\Delta/4)>0.
$$
However,  we must have $\rho_*(\lim_{n\rightarrow\infty}\hat d(X_{1..n},\cl{H_0}\cap\mathcal E)=0)=1$
for every $\rho_*\in\cl{H_0}$: indeed, for $\rho_*\in\cl{H_0}\cap\mathcal E$ it follows from Lemma~\ref{th:dd}, and for $\rho_*\in\cl{H_0}\backslash\mathcal E$ 
from  Lemma~\ref{th:dd}, ergodic decomposition and the conditions of the theorem ($W_\rho(H_0)=1$ for $\rho\in\cl{H_0}$).

This contradiction shows that for every $\alpha$ there are not more than finitely many $n$ for which $C^n(\cl{H_0}\cap\mathcal E, 1-\alpha)\cap b_{\Delta/2}^n(\xi)\ne\emptyset$. 
To finish the proof  of the first statement,  it remains to note that, as follows from Lemma~\ref{th:dd}, 
\begin{equation*}
\xi\{X_1,X_2,\dots.:X_{1..n}\in b_{\Delta/2}^n(\xi)\text{ from some $n$ on}\}\ge \xi\left(\lim_{n\rightarrow\infty}\hat d(X_{1..n},\xi)=0\right) =1.
\end{equation*}

To establish the second statement of Theorem~\ref{th:asym}
 we assume that there exists a
 consistent test $\phi$ for $H_0$ against $H_1$, and we will show that
 $W_\rho(H_1)=0$ for every $\rho\in \cl{H_0}$. %
Take $\rho\in \cl{H_0}$ and suppose that  
\begin{equation}\label{eq:ash1}
W_\rho(H_1)=\delta>0. 
\end{equation}
We have 
$$
 \limsup_{n\to\infty} \int_{H_1} d W_\rho(\mu) \mu(\psi^{\delta/2}_n=0)  \le  \int_{H_1} d W_\rho(\mu)\limsup_{n\to\infty}  \mu(\psi^{\delta/2}_n=0)=0,
$$
where the inequality follows from Fatou's lemma (the functions under integral are all bounded by 1), and the equality from the consistency of $\psi$. 
Thus, from some $n$ on 
we will have $\int_{H_1} d W_\rho \mu(\psi^{\delta/2}_n=0) < 1/4$. Taking into account~\eqref{eq:ash1}, we conclude  $\rho(\psi^{\delta/2}_n=0)<1-3\delta/4$. 
For any set $T\in A^n$ the function $\mu(T)$ is continuous as a function of $T$. In particular, it holds for the
set $T:=\{X_{1..n}:\psi_n^{\delta/2}(X_{1..n})=0\}$. Therefore,
since $\rho\in\cl{H_0}$, for any $n$ large enough we can find a $\rho'\in H_0$ such that $\rho'(\psi^{\delta/2}_n=0)<1-3\delta/4$,
 which contradicts
the consistency of $\psi$. Thus, $W_\rho(H_{1})= 0$, and Theorem~\ref{th:asym} is proven.
\end{proof}
\smallskip
\begin{proof}[of Theorem~\ref{th:uni}.]
To prove the first statement of the theorem,
we will show that the test $\phi_{H_0,H_1}$ is a uniformly consistent test for $\cl{H_0}\cap\mathcal E$ against $\cl{H_1}\cap\mathcal E$ (and hence for $H_0$ against $H_1$),
under  the conditions of the theorem. %
Suppose that, on the contrary,  for some $\alpha>0$ for every $n'\in\N$ there is a process $\rho\in \cl{H_0}$ such that 
$\rho(\phi(X_{1..n})=1)>\alpha$ for some $n>n'$. 
Define 
$$
\Delta:=d(\cl{H_0},\cl{H_1}):=\inf_{\rho_0\in \cl{H_0}\cap\mathcal E, \rho_1\in \cl{H_1}\cap\mathcal E} d(\rho_0,\rho_1),
$$ which is positive since $\cl{H_0}$ and $\cl{H_1}$ are closed and disjoint.
We have
\begin{multline}\label{eq:union}
\alpha<\rho(\phi(X_{1..n})=1)\\ \le   \rho(\hat d(X_{1..n},H_0)\ge\Delta/2\ or \ \hat d(X_{1..n},H_1)<\Delta/2)\\ \le \rho(\hat d(X_{1..n},H_0)\ge\Delta/2) + \rho(\hat d(X_{1..n},H_1)<\Delta/2).
\end{multline} 
This implies that either 
$$\rho(\hat d(X_{1..n},\cl{H_0})\ge\Delta/2)>\alpha/2$$
 or
 $$\rho(\hat d(X_{1..n},\cl{H_1})<\Delta/2)>\alpha/2,$$
so that, by assumption, at least one of these inequalities holds for infinitely many $n\in\N$ for some sequence  $\rho_n\in H_0$.
Suppose that it is the first one, that is, there is an increasing sequence $n_i$, $i\in\N$ and a sequence $\rho_i\in\cl{H_0}$, $i\in\N$ 
such that 
\begin{equation}\label{eq:da}
\rho_i(\hat d(X_{1..n_i},\cl{H_0})\ge\Delta/2)>\alpha/2 \text{ for all }i\in\N. 
\end{equation}
The set $\S$ is compact, hence so is its closed subset $\cl{H_0}$. Therefore, the sequence $\rho_i$, $i\in\N$ must
contain a subsequence that converges to a certain process $\rho_*\in\cl{H_0}$. Passing to a subsequence if necessary, we may assume
that this convergent subsequence is the sequence $\rho_i$, $i\in\N$ itself.

Using Lemma~\ref{th:count}, (\ref{eq:count1}) (with $\rho=\rho_{n_m}$, $m=n_m$, $k=n_k$, and $H=\cl{H_0}$),  and taking  $k$  large enough to have $t_{n_k}<\Delta/4$,
 for every $m$  large enough to have $\frac{2n_k}{n_m-n_k+1}<\Delta/4$, we obtain 
\begin{multline}\label{eq:down}
8\Delta^{-1}\rho_{n_m}\left(\hat d(X_{1..n_k},\cl{H_0})  \ge\Delta/4\right) \\ \ge  \rho_{n_m}\left(\hat d(X_{1..n_m},\cl{H_0})\ge\Delta/2\right) > \alpha/2. 
\end{multline}
That is, we have shown that for any large enough index $n_k$ the inequality 
$\rho_{n_m}(\hat d(X_{1..n_k},\cl{H_0})\ge\Delta/4)> \Delta\alpha/16$ holds for 
infinitely many indices $n_m$. 
From this and Lemma~\ref{th:cont} with  $T=T_k:=\{X:\hat d(X_{1..n_k},\cl{H_0})\ge\Delta/4\}$  we conclude 
that $\rho_*(T_k)>\Delta\alpha/16$.
The latter holds for infinitely many $k$; that is,  $\rho_*(\hat d(X_{1..n_k},\cl{H_0})\ge\Delta/4)>\Delta\alpha/16$ infinitely often.
Therefore, 
$$
\rho_*(\limsup_{n\rightarrow\infty} d(X_{1..n},\cl{H_0})\ge\Delta/4)>0.
$$ However, we must have 
$$
\rho_*(\lim_{n\rightarrow\infty} d(X_{1..n},\cl{H_0})=0)=1
$$
for every $\rho_*\in\cl{H_0}$: indeed, for $\rho_*\in\cl{H_0}\cap\mathcal E$ it follows from Lemma~\ref{th:dd}, and for $\rho_*\in\cl{H_0}\backslash\mathcal E$ 
from  Lemma~\ref{th:dd}, ergodic decomposition and the conditions of the theorem. 

Thus, we have arrived at a contradiction that shows that $\rho_n(\hat d(X_{1..n},\cl{H_0})>\Delta/2)>\alpha/2$ cannot hold for 
infinitely many $n\in\N$ for any sequence of  $\rho_n\in\cl{H_0}$. Analogously, we can show that $\rho_n(\hat d(X_{1..n},\cl{H_1})<\Delta/2)>\alpha/2$
cannot hold for infinitely many $n\in\N$ for any sequence of $\rho_n\in\cl{H_0}$. Indeed,  using Lemma~\ref{th:count}, equation~(\ref{eq:count2}), we can show that
  $\rho_{n_m}(\hat d(X_{1..n_m},\cl{H_1})\le\Delta/2) > \alpha/2$ for a large enough $n_m$
implies $\rho_{n_m}(\hat d(X_{1..n_k},\cl{H_1})\le 3\Delta/4)> \alpha/4$ for a smaller $n_k$.
Therefore, if we assume that $\rho_n(\hat d(X_{1..n},\cl{H_1})<\Delta/2)>\alpha/4$ for infinitely many $n\in\N$ for some sequence of $\rho_n\in\cl{H_0}$, then 
we will also find a $\rho_*$ for which $\rho_*(\hat d(X_{1..n},\cl{H_1})\le 3\Delta/4)> \alpha/4$ for infinitely
many $n$, which, using Lemma~\ref{th:dd} and ergodic decomposition, can be shown to contradict the fact that $\rho_*(\lim_{n\rightarrow\infty}  d(X_{1..n},\cl{H_1})\ge\Delta)=1$.
 
Thus, returning to~(\ref{eq:union}), we have shown that from some $n$ on there is no $\rho\in \cl{H_0}$ for which $\rho(\phi=1)>\alpha$
holds true. The statement for $\rho\in\cl{H_1}$ can be proven analogously, thereby finishing the proof of the first statement.

To prove the second statement of the theorem, %
we assume that there exists a uniformly consistent test $\phi$ for $H_0$ against $H_1$, and we will show that
$W_\rho(H_{1-i})=0$ for every $\rho\in \cl{H_i}$.
Indeed, let $\rho\in \cl{H_0}$, that is, suppose that there is a sequence $\xi_i\in H_0, i\in\N$ such that $\xi_i\to\rho$. 
Assume  $W_\rho(H_1)=\delta>0$ and take $\alpha:=\delta/2$.
Since the test $\phi$ is uniformly consistent, there is an $N\in\N$ such  that  for every $n>N$ we have  
\begin{multline*}
 \rho(\phi(X_{1..n}=0))\le \int_{H_1} \phi(X_{1..n}=0)dW_\rho + \int_{\mathcal E\backslash H_1} \phi(X_{1..n}=0) dW_\rho \\ \le \delta\alpha + 1-\delta \le 1-\delta/2.
\end{multline*}
Recall that, for $T\in A^*$,  $\mu(T)$ is a continuous function in $\mu$. In particular, this holds for the set $T=\{X\in A^n: \phi(X)=0\}$, 
for any given $n\in\N$. Therefore, for every  $n>N$ and for every  $i$ large enough,  $\rho_i(\phi(X_{1..n})=0)<1-\delta/2$ implies also $\xi_i(\phi(X_{1..n})=0)<1-\delta/2$
which contradicts $\xi_i\in H_0$.
This contradiction shows $W_\rho(H_1)=0$ for every $\rho\in \cl{H_0}$. The case $\rho\in \cl{H_1}$ is analogous.
\end{proof}

\section{Examples}\label{s:ex}
Theorems~\ref{th:asym} and~\ref{th:uni}   can be used to check whether a consistent test exists 
for such  problems as identity, independence, estimating the order  of a (Hidden) Markov model, bounding
entropy, bounding distance, uniformity, monotonicity, etc. Some of these examples are considered in this section.

\subsection{Simple hypotheses, identity or goodness-of-fit testing} First of all, it is obvious that sets that consist of just
one or finitely many stationary ergodic processes are closed and closed under ergodic decompositions. Thus, they meet the conditions of Theorem~\ref{th:uni}, and so, 
for any pair of disjoint sets of this type, there exists a uniformly consistent test. (In particular, there 
is a uniformly consistent test for $H_0=\{\rho_0\}$ against $H_1=\{\rho_1\}$ iff $\rho_0\ne\rho_1$.)

A more interesting case is identity testing, also known as goodness-of-fit: %
this problem  consists in testing whether a distribution 
generating the sample obeys a certain given law, versus it does not. Thus, let $\rho\in\mathcal E$, $H_0=\{\rho\}$ and $H_1=\mathcal E\backslash H_0$.
In such a case  there is an asymmetrically consistent test for   $H_0$ against $H_1$:
indeed, the conditions of Theorem~\ref{th:asymc}  are easily verified.
It is worth noting that  (asymmetric) identity testing is a classical problem of mathematical statistics, with solutions (e.g. based on Pearson's $\chi^2$ statistic) for i.i.d.\  data (e.g. \cite{Lehmann:86}), and Markov chains \cite{Billingsley:61}. 
  For stationary ergodic processes, \cite{BRyabko:06b} gives an asymmetrically
 consistent test when $H_0$ has a finite and bounded memory, and \cite{Ryabko:103s} for the general case of stationary ergodic 
 real-valued processes.
 
As far as uniform testing is concerned, it is, first of all, clear that, just like in the i.i.d.\ case (cf. Section~\ref{s:one}),  for any $\rho_0$ 
 there is no uniformly consistent test for identity. Indeed, as we have seen (Corollary~\ref{th:nouni}),  for any non-empty $H_0$ there is no uniformly 
consistent test for $H_0$ against $\mathcal E\backslash H_0$ provided neither hypothesis is non-empty.
One might suggest at this point that, as in the i.i.d.\ case, a uniformly consistent test exists if we restrict $H_1$ to those processes
that are sufficiently far from $\rho_0$, for example, by introducing some $\epsilon$-padding around $H_0$. However, this is not the case. We can prove an even stronger negative result.
\begin{proposition}\label{th:no}
Let  $\rho,\nu\in\mathcal E$, $\rho\ne\nu$ and let  $\epsilon>0$.
There is no uniformly consistent test for $H_0=\{\rho\}$ against $H_1=\{\nu'\in\mathcal E: d(\nu',\nu)\le\epsilon\}$. 
\end{proposition}
The following conclusion can be made from this proposition.
\begin{svgraybox}
While distributional distance
is well-suited for characterizing those hypotheses for which consistent tests exist, it is not 
suited for  formulating the actual hypotheses. 
\end{svgraybox}
\hskip-4mm\noindent Apparently, a stronger distance is needed for the latter.

\begin{proof}[of Proposition~\ref{th:no}]
Consider the process $(X_1,Y_1),(X_2,Y_2),\dots$ on pairs \\ $(X_i,Y_i)\in A^2$, such that
the distribution of $X_1,X_2,\dots$ is $\nu$, the distribution of $Y_1,Y_2,\dots$ is $\rho$ and 
the two components $X_i$ and $Y_i$ are independent; in other words, the distribution of $(X_i,Y_i)_{i\in\N}$ is $\nu\times\rho$. 
Consider also a two-state stationary ergodic  Markov chain $\mu$, with two states $1$ and $2$, whose transition
probabilities are $\left(\begin{array}{cc}
                          1-p & p \\ q & 1-q
                         \end{array} \right)$, where $0<p<q<1$. 
The limiting (and initial) probability of the state $1$ is $p/(p+q)$ and that of the state $2$ is $q/(p+q)$. 
Finally, the process $Z_1,Z_2,\dots$ is constructed as follows:
$Z_i=X_i$ if $\mu$ is in the state $a$ and $Z_i=Y_i$ otherwise (here it is assumed that the chain $\mu$ generates
a sequence of outcomes independently of $(X_i,Y_i)$). Clearly, for every $p,q$ satisfying $0<p<q<1$ the process $Z_1,Z_2,\dots$ is stationary ergodic.
Let $p_m:=1/(m+1)$, $q_m:=\delta p_m/(1-\delta)$ for all $m\in\N$, where $\delta$ is a parameter to be defined shortly. Denote $\zeta_m$ the distribution of the process $(Z_i)_{i\in\N}$ with parameters $p_m,q_m$. With these parameters, $\mu(1)=\delta$ independently of $m$ (i.e, the Markov chain  underlying $\zeta_m$ spends $\delta$ time in the first state). 
Find $\delta>0$  sufficiently small   so 
as to have  for all $m$ sufficiently large  $d(\nu,\zeta_m)<\epsilon$, as is always possible since $\lim_{\delta\to0}\zeta_m=\nu$ uniformly in $m$. 
Thus, $\zeta_m\in H_1$ for all $m\in\N$. However, $\lim_{m\to\infty}\zeta_m=\zeta_\infty$
where $\zeta_\infty$ is the stationary distribution with $W_{\zeta_\infty}(\rho)=\delta$ and $W_{\zeta_\infty}(\nu)=1-\delta$.
Therefore, $\zeta_\infty\in\cl{H_1}$ and $W_{\zeta_\infty}(H_0)>0$, so that by Theorem~\ref{th:uni} there is no uniformly consistent
test for $H_0$ against $H_1$.
\end{proof}

\subsection{Markov and  Hidden Markov processes: bounding the order} 
Let us next consider finite-state Markov and hidden Markov processes.
 
For any $k$, there is an asymmetrically consistent  test of the hypothesis $\mathcal M_k$= ``the process is Markov of order not greater than $k$'' against $\mathcal E\backslash \mathcal M_k$.
For any $k$, there is an asymmetrically consistent test of $\mathcal{HM}_k$=``the process is given by a Hidden Markov process with not more than $k$ states''
against $H_1=\mathcal E\backslash \mathcal{HM}_k$.
Indeed, in both cases ($k$-order Markov, Hidden Markov with not more than $k$ states), the hypothesis $H_0$
is a parametric family, with a compact set of parameters, and a continuous function mapping parameters 
to processes (that is, to the space $\mathcal S$). Since the space $\S$ of stationary processes is compact, Weierstrass theorem then implies that the image of such a compact
parameter set is closed (and compact). Moreover, in both cases $H_0$ is closed under taking ergodic decompositions.
Thus, by Theorem~\ref{th:asym}, there exists an asymmetrically consistent test.

The problem of estimating the order of a (hidden) Markov process based on  sampling  had been addressed in a number of works. In the contest 
of hypothesis testing, asymmetrically consistent tests for $\mathcal M_k$ against $\mathcal M^t$ with $t>k$ were given in \cite{Anderson:57}, see also \cite{Billingsley:61}.
The existence of non-uniformly consistent tests (a notion weaker than that of asymmetric consistency) for $\mathcal M_k$ against $\mathcal E\backslash\mathcal M_k$, 
and of $\mathcal{HM}_k$ against $\mathcal E\backslash \mathcal{HM}_k$,  was established  in \cite{Kieffer:93}.  Asymmetrically consistent
tests for $\mathcal M_k$ against $\mathcal E\backslash\mathcal M_k$ were obtained in \cite{BRyabko:06a}, while 
for the formulation above that includes the case of asymmetric testing for $\mathcal {HM}_k$ against $\mathcal E\backslash\mathcal {HM}_k$  is from \cite{Ryabko:121c}. 

Considering the set $\mathcal M_*:=\cup_{k\in\N}\mathcal M_k$ of all finite-memory processes, it is easy to see that there is no asymmetrically consistent test for this set against its complement: indeed, $\cl M_*=\S$, so by Corollary~\ref{th:asym} there is no test. There is also no asymptotically consistent test for this hypothesis, even though it is possible to construct an estimator of the order of  a Markov chain that tends to infinity if the process is not Markov; see \cite{Morvai:05} and references. 

\subsection{Smooth parametric families} From the discussion in  the  previous example we can see that the following generalization is valid.
Let $H_0\subset\mathcal S$ be a set of processes that is continuously parametrized by a compact set of parameters. If $H_0$
is closed under taking ergodic decompositions, then there is an asymmetrically consistent test for $H_0$ against $\mathcal E\backslash H_0$.
In particular, this strengthens the mentioned result of \cite{Kieffer:93}, since a stronger notion of consistency is used, as well 
as a more general class of parametric families is considered. 

Clearly, a similar statement can be derived for uniform testing: given two disjoint sets  $H_0$ and $H_1$ each of which 
is continuously parametrized by a compact set of parameters and is closed under taking ergodic decompositions, 
there exists a uniformly consistent test of $H_0$ against $H_1$.

\subsection{Homogeneity testing or process discrimination}
This problem consists in testing, given two samples $X_{1..n}^1$ and $X_{1..n}^2$, whether the distributions generating these samples are the same or different. We have considered this problem in details in Section~\ref{s:hom} for the case of asymptotic consistency and stationary ergodic distinctions (and $B$-processes), and in  Section~\ref{s:one} for the case of asymmetric and uniform consistency and smaller sets of distributions. 
The results can be summarized in the following table.   Here we omit uniform testing in view of Corollary~\ref{th:nouni}.
\begin{table}[h]
\caption[]{Existence of a consistent test for the hypothesis of homogeneity against its complement, for different notions of consistency and classes of processes}\label{t:hom}
\begin{tabular}{l|l|l|l}
  & I.i.d. & Markov & Stationary ergodic  \\\hline
Asymmetric consistency& Test exists\ \ & No test \ \  & No test   \\
Asymptotic consistency& Test exists\ \  & Test exists \ \  & No test (Theorem~\ref{th:discr})   
\end{tabular}
\end{table}

\subsection{Independence}\label{s:ind}
Again, we are given two samples, $X_{1..n}^1$ and $X_{1..n}^2$. %
The hypothesis of independence is that the first process  is  independent from the second:
$\rho(X^1_{1..t}\in T_1, X^2_{1..t}\in T_2)=\rho(X^1_{1..t}\in T_1)\rho(X^2_{1..t}\in T_2)$ for any $(T_1,T_2)\in A^n$ and any $n\in\N$.

Let $\mathcal I$ be the set of all stationary ergodic processes (on pairs) satisfying this property.

\begin{proposition}\label{th:notest}
 There is no asymmetrically consistent  test for  independence (for jointly stationary ergodic samples). 
\end{proposition}
\begin{proof}
The example is based on the so-called translation process, which is constructed as follows. 
Fix some irrational  $\alpha \in (0,1)$ and select $r_0 \in [0,1]$ uniformly at random. 
For each $i=1..n..$ let $r_i=(r_{i-1}+\alpha) \mod 1$ (that is, the previous element is shifted by $\alpha$ to the right, considering the [0,1] interval looped). 
The samples $X_i$ are obtained from $r_i$ by thresholding at $1/2$, 
i.e. $X_i:=\mathbb{I}\{r_i>0.5\}$ (here $r_i$ can be considered hidden states). This process is stationary and ergodic; besides, it has 0 entropy rate~\cite{Shields:98}, and this is not the last of its peculiarities. 

Take now two independent copies of this process to obtain a pair $(\x_1,\x_2)=(X_1^1,X_1^2\dots,X_n^1,X_n^2,\dots)$. %
 The resulting process on pairs, which we denote $\rho$, is stationary, but it is not ergodic. To see the latter, observe that the difference between the corresponding hidden states remains constant. 
In fact, each initial state $(r_1,r_2)$ corresponds to an ergodic component of our process on pairs. By the same argument, these ergodic components are not independent. Thus, we have taken two independent copies of a stationary ergodic process, and obtained a stationary process which is not ergodic and whose ergodic components are pairs of processes that are not independent!

To apply Corollary~\ref{th:asymc}, it remains to show that the process $\rho$ we constructed can be obtained as a limit of stationary ergodic processes on pairs. To see this, consider, for each $\epsilon$,  a process $\rho_\epsilon$, whose construction is identical to $\rho$ except that instead of shifting the hidden states by $\alpha$ we shift them by $\alpha+u_i^\epsilon$ where $u_i^\epsilon$ are i.i.d.\ uniformly random on $[-\epsilon,\epsilon]$. It is easy to see that $\lim_{\epsilon\to0} \rho_\epsilon=\rho$ in distributional distance, and all $\rho_\epsilon$ are stationary ergodic.  Thus, if $H_0$ is the set of all stationary ergodic distributions on pairs, we have found a distribution $\rho\in\cl{H_0}$ such that $W_\rho(H_0)=0$. We can conclude that  there is no $\alpha$-level consistent test for $H_0$ against its complement. 
\end{proof}

In contrast to the situation with homogeneity testing described in Section~\ref{s:one},  testing  independence becomes possible if we restrict the processes to be Markov.

Indeed, using the notation of the previous sections, it is easy to see that Theorem~\ref{th:asym} implies that there exists an asymmetrically consistent test for $\mathcal I\cap \mathcal M_k$
against $\mathcal E\backslash \mathcal I$,  for any given $k\in\N$. %
Analogously, if we confine $H_0$ to Hidden Markov processes of a given order, then asymmetric testing is possible. That is,
there exists an an asymmetrically consistent test for $\mathcal I\cap \mathcal {HM}_k$
against $\mathcal E\backslash \mathcal I$, for any given $k\in\N$.

\section{Open problems}
In spite of rather general results on the existence of tests presented in this chapter, perhaps it would not be an exaggeration to say that the most important questions remain open. This section attempts to  precise and summarize these. 
\subsection{Relating the notions of consistency}
Before delving deeper into problems relating various notions of consistency and generalizing the corresponding results,  
note that  two of the  notions of consistency considered, asymmetric ($\alpha$-level) consistency and asymptotic consistency, 
require a certain convergence to hold with probability~1. Naturally, one could replace this convergence 
with convergence in probability. Let us call the resulting notion {\em weak} asymmetric or asymptotic consistency, and 
those introduced above let us call {\em strong}. While weak consistency indeed appears weaker at first sight, it is easy to see, as Nobel \cite{Nobel:06}
remarks, that weak  asymptotic consistency implies strong asymptotic consistency for the case of i.i.d.\
or strongly mixing processes. It is similarly easy to verify that the same is true for asymmetric consistency. Moreover, 
for asymmetric consistency, the criterion given in Corollary~\ref{th:asymc} holds equally well for strong and for weak 
consistency, so in the case $H_0=\mathcal E\backslash H_1$ weak and strong asymmetric consistency are equivalent 
for stationary ergodic distributions as well.  This  suggests that these notions may be equivalent in general.
\begin{conjecture}[weak=strong]
 For stationary ergodic distributions, if there exists a weakly asymmetrically consistent  (weakly asymptotically consistent) test,
then there exists  a strongly consistent asymmetrically (strongly asymptotically consistent) test.
\end{conjecture}

Passing to the relations between the notions of consistency, it might at first glance seam that asymmetric 
consistency is rather weak, since one of the errors does not go to zero. However, note that it is fixed at the given level $\alpha$
 independently of the sample size, and uniformly over $H_0$, making the resulting notion very strong.
In fact, from the discussion  on the i.i.d.\  processes in Section~\ref{s:one},  one can see that, for i.i.d.\ examples, 
uniform consistency is strictly stronger than  asymmetric consistency, and asymmetric consistency is strictly stronger than asymptotic consistency (in terms of the existence of tests).
One can conjecture that this is the case for stationary ergodic distributions as well. 
\begin{conjecture}[uniform$\Rightarrow$ asymmetric $\Rightarrow$ asymptotic consistency]
 Let $H_0,H_1\subset\mathcal E$. If there exists a uniformly consistent test for $H_0$ against $H_1$, then there
exists an asymmetrically consistent test for this pair of hypotheses. If there  exists an asymmetrically consistent  test  for $H_0$ against $H_1$, then there
exists an asymptotically consistent test for this pair of hypotheses. The opposite implications do not hold.
\end{conjecture}
Note that the implication ``uniform$\Rightarrow$  asymptotic consistency'' is rather obvious, and it is also obvious that the opposite does not hold. The question is, therefore, about the place of asymmetric consistency in the middle; more precisely, whether the strict inclusion generalises from the i.i.d.\ to the stationary ergodic case.
\begin{svgraybox}
 It remains open to see whether the  relation between the notions of consistency (uniform, asymmetric, asymptotic, weak/strong) that holds for i.i.d.\ processes carries over to the stationary ergodic case.
\end{svgraybox}

\subsection{Characterizing hypotheses for which consistent tests exist}

The { main open problem} that remains is to find necessary and sufficient conditions for the existence of each 
kind of the tests: uniform, asymmetric, and asymptotic. 

\begin{problem}
 Find necessary and sufficient conditions on hypotheses $H_0,H_1\subset\mathcal E$ for the existence of (uniformly, asymmetrically,  asymptotically) consistent tests.
\end{problem}
The only case for which the presented necessary and sufficient conditions coincide is the case of asymmetric consistency when $H_1=\mathcal E\setminus H_0$. It is not known whether the same conditions are necessary and sufficient for general pairs $H_0,H_1$ (i.e., when $H_1$ is not necessarily the complement of $H_0$).  However, the fact that for this case we have an ``if and only if'' criterion, suggests that the topology of the distributional distance is indeed the right one to consider for such characterisations.

Another important problem is to generalize the results of Chapter~\ref{ch:ht} to real-valued processes.
\begin{problem}
 Find generalisations of Theorems~\ref{th:asym},~\ref{th:uni} to real-valued processes.
\end{problem}
The main difference for the real-valued case is that, in the finite-alphabet case, the distributional distance in the form~\eqref{eq:ddisd} gives a compact space of distributions. This fact has been relied upon heavily in the proofs of the corresponding theorems. The distributional distance in the form~\eqref{eq:ddisr} does not result in a compact space of distributions. The general form~\eqref{eq:ddis} {\em can} give a compact space; indeed, as mentioned in Chapter~\ref{ch:pre}, this is the case if the sets $(B_i)_{i\in\N}$ form is a standard basis. However, as \cite{Gray:88} mentions, there is no easy constructing of such a basis for the real-valued case, even though such a basis exists. On the other hand, an explicit construction is required in order to speak about distance estimates.

\subsection{Independence testing}
Recall the problem of independence from Section~\ref{s:ind}:
  given two samples, $X_1,\dots,X_n$ and $Y_1,\dots,Y_m$, %
 it is required to test whether the   process  generating the first sample is  independent from the one generating the second. 

It is interesting to note that for the case of i.i.d.\ data, the problems of homogeneity testing and independence testing can be reduced to one another. The situation is different for dependent data, as we have seen already for the case of (discrete-state) Markov chains:  for these processes, there exists an asymmetric test for independence but not for homogeneity. Moreover, whereas for homogeneity (process discrimination) we have seen in Section~\ref{s:hom} that there is no asymptotically consistent test, for independence the question of the existence of such a test remains open.

 Thus, we can formulate what is known and what is not known about this problem in the following table, which can be compared  to the one about homogeneity testing (Table~\ref{t:hom}). 

\begin{table}[h]\label{t:ind}
\caption[]{Existence of a consistent test for the hypothesis of independence against its complement, for different notions of consistency and classes of processes. The differences with homogeneity  testing (Table~\ref{t:hom}) are marked in bold.}
\begin{tabular}{l|l|l|l}
  & I.i.d. & Markov & Stationary ergodic  \\\hline
Asymmetric consistency& Test exists\ \ & {\bf Test exists}\ \  & No test (Proposition~\ref{th:notest})   \\
Asymptotic consistency& Test exists\ \ & Test exists\ \ & {\bf Open question}   
\end{tabular}
\end{table}

\chapter{Generalizations}\label{ch:conc}
In this chapter we outline a number of generalizations of the results described in this volume. Some of these have already been made, while others present interesting directions for future research.
\section{Other distances}
The empirical  distributional distance on which the results of the previous chapters hinge can be seen as an ordinate way of counting frequencies of everything. One may wonder whether the same theoretical consistency results can be obtained while allowing one to benefit from using  some of the more sophisticated tools in the box. 

This is, indeed, possible, by considering different distances between processes, and then plugging in their estimates into the same algorithms. 
Here we try to see what distances can be used and which properties are required. While doing so we are mostly concerned with generalizing the results of Chapters~\ref{ch:basic} and~\ref{ch:clchp}, as the theory of hypothesis testing of Chapter~\ref{ch:ht} is somewhat more delicate.

Introduce the notation $\rho^k$ for the $k$-dimensional marginal distribution of a time-series distribution~$\rho$. %

\subsection{$\operatorname{sum}$ Distances}

Observe that the distributional distance $d$ in its more-specified formulations~\eqref{eq:ddisd} and~\eqref{eq:ddisr}  has the form 
\begin{equation}\label{eq:sum}
 d(\rho_1,\rho_2)=\sum_{k\in\N} w_k d_k(\rho_1^k,\rho_2^k),
\end{equation}
where $w_k$ are summable positive real weights and $d_k()$ is a certain distance between $k$-dimensional marginal distributions. %

It is easy to see that distances of this  form  can be consistently 
estimated, as long as $d_k$ can be consistently estimated for each~$k\in\N$; this is formalized in the following statement.
\begin{proposition}[estimating sum-based distances] 
Let  $\C$ be a set of process distributions. Let $d_k, k\in\N$ be a series of distances on the spaces of distributions over $A^k$ that are bounded uniformly in $k$, %
and such that 
there exists a series $\hat d_k(X_{1..n},Y_{1..n}),k\in\N$ of their consistent estimates: 
 $\lim_{n\to\infty} \hat d_k(X_{1..n},Y_{1..n})=d_k(\rho_1^k,\rho_2^k)$ a.s.,
whenever $\rho_1,\rho_2\in\C$ are chosen to generate the sequences. Then the distance $D$ given by~(\ref{eq:sum}) can be
consistently estimated using the estimate $\sum_{k\in\N} w_k\hat d_k(X_{1..n},Y_{1..n})$. 
\end{proposition}

Clearly, the distributional distance $d$ is an example of a distance in the form~(\ref{eq:sum}), and it satisfies the conditions 
of the proposition with $\mathcal C$ being the set of all stationary ergodic processes. 
Another example is the {\em telescope distance} %
considered in the next subsection.

\subsection{Telescope distance}
The telescope distance, introduced in  \cite{Ryabko:13red+}, is, in fact, a scheme for defining distances between processes.
In order to define the telescope distance,  we first start with a  metric on distributions on $A^k$. 
For two probability  distributions $P$ and  $Q$~on $(A^k,\B_k)$ for some $k\in\N$ and  a set $\H$ of measurable functions on $A^k$, one can define the  distance
\begin{equation}\label{eq:f}
d_\H(P,Q):=\sup_{h\in\H} |\E_P h- \E_Q h|.
\end{equation}
This metric in its general form has been studied since at least   \cite{Zolotarev:83} and 
 includes Kolmogorov-Smirnov  \cite{Kolmogorov:33} and Kantorovich-Rubinstein \cite{Kantorovich:57}  metrics as special cases.
It is measurable under mild conditions; in particular, separability of $\H$ is sufficient for this.
Moreover, it is easy to check that  $d_\H$ is a metric on the space of probability distributions over $A^k$ if 
 and only if $\H$ generates~$B_k$.

An example of the sets $\H$ are the sets of  hyperplanes in $\R^k$, $k\in\N$. 

Based on $d_\H$ we can construct a distance between time-series probability distributions.
For two time-series distributions $\rho_1,\rho_2$ and sets $\H_k$ of functions on $A^k$, $k\in\N$, we take the  $d_{\H_k}$ between $k$-dimensional marginal distributions
of $\rho_1$ and $\rho_2$ for each $k\in\N$, and sum them all up with decreasing weights.

\begin{definition}[telescope distance]\label{d:tele} 
For two processes $\rho_1$ and $\rho_2$ %
and a  sequence of sets of functions $\bH=(\H_1,\H_2,\dots)$ 
 define the {\em telescope distance} 
\begin{equation}%
 D_\bH(\rho_1,\rho_2)%
:=\sum_{k=1}^\infty w_k \sup_{h\in\H_k} |\E_{\rho_1} h (X_1,\dots,X_k)- \E_{\rho_2} h(Y_1,\dots,Y_k)|,
\end{equation}
where $w_k$, $k\in\N$ is  a sequence of positive summable real weights (e.g., the weights we were using before, $w_k:=1/k(k+1)$).
The {\em empirical telescope distance} is defined
as 
\begin{multline*}\label{eq:ets}
 \hat D_\bH(X_{1..n},Y_{1..m}):=\\
\sum_{k=1}^{\min\{m,n\}} w_k \sup_{h\in\H_k} \left|\frac{1}{n-k+1}\sum_{i=1}^{n-k+1}  h (X_{i..i+k-1})- \frac{1}{m-k+1}\sum_{i=1}^{m-k+1}  h(Y_{i..i+k-1})\right|.
\end{multline*}
\end{definition}

It is shown in \cite{Ryabko:13red+} that the empirical telescope distance so defined is a consistent estimate of the telescope distance, 
if the sets $\H_k$ are separable sets of indicator function of finite VC dimension. The separability condition comes  from \cite{Adams:12} where the corresponding uniform convergence result is established. 

The main appeal of the telescope distance is that it can be estimated using binary classification methods developed for i.i.d.\ data. Such methods  are abound in the machine learning literature. Thus, the telescope distance  allows one to channel these methods for use in problems involving  time series, such as clustering and the three-sample problem considered in Chapters~\ref{ch:basic},~\ref{ch:clchp}.  

The details of the algorithms, as well as the  proofs and experimental results, can be found in  \cite{Ryabko:13red+}.

\subsection{$\operatorname{sup}$ Distances}
A different way to construct a distance between time-series distributions based on their finite-dimensional marginals is to 
use the supremum instead of summation in~\eqref{eq:sum}: 
\begin{equation}\label{eq:sup2}
 d(\mu,\nu):=\sup_{k\in\N} d_k(\mu_k,\nu_k).
\end{equation}

Some commonly used metrics are defined in the form~(\ref{eq:sup2})  or have natural interpretations in this form, as the following two examples show.

\begin{definition}[total variation]
 For time-series distributions $\nu,\mu$ the total variation distance %
between them is defined as $D_{tv}(\mu,\nu):=\sup_{A\in\B}|\mu(A)-\nu(A)|$.
\end{definition}
It is easy to see that  $D_{tv}(\mu,\nu)=\sup_{k\in\N}\sup_{A\in\B_{k}}|\mu(A)-\nu(A)|$, so that the total variation distance has the form~(\ref{eq:sup2}).

For stationary ergodic distributions this distance is not very useful, since it just gives the discrete distance: $D_{tv}(\mu,\nu)=1$ if and only if $\mu\ne\nu$. This follows from the fact that any two different stationary ergodic distributions are singular with respect to one another.

Another example of a $\operatorname{sup}$-distance is the $\bar d$ distance, defined in Section~\ref{s:hom}. To see that it is indeed a $\operatorname{sup}$-distance, consider the following definition of it, which is equivalent to the previous one (see, e.g.\ \cite{Shields:96,Ornstein:90})
$$
\bar d(\rho_1,\rho_2):=\sup_{k\in\N}\frac{1}{ k}\inf_{p\in P}\sum_{i=1}^k \E_p\delta(x_i,y_i),
$$
where $P$ is the set of all distributions  over
 $A^k\times A^k$ generating a pair of samples $X_{1..k}, Y_{1..k}$ whose marginal 
distributions are $\rho_1^k$ and $\rho_2^k$ correspondingly. %

As explained in  Section~\ref{s:hom}, this distance turns out to bee too strong  for stationary ergodic processes but still useful for $B$-processes, since it is only possible to construct its consistent estimates for the latter set.

\subsection{Non-metric distances}
So far we have been considering distances  that constitute a metric on the space of all process distributions, or on the space of stationary process distributions. In particular, they have the property of exactness, that is $d(\rho_1,\rho_2)=0$ if and only if $\rho_1=\rho_2$.  This allowed us to solve such problems as clustering (with respect to distribution), where we cluster together those and only those samples that were generated by the same distribution. 

Sometimes a weaker goal may be appropriate. For example, one may wish to distinguish only between distributions that have different single-dimensional means and variances, or some other characteristics.
Depending on the characteristics  of the processes studied, it may be more or less straightforward to establish the consistency of their empirical estimates. However, if consistent empirical estimates are available, it should be reasonably straightforward to translate the algorithms and the results on  clustering and change-point problems to such distances. 

\subsection{AMS distributions}
A particular instance of non-metric distances described in the previous section are distances between the asymptotic-mean distributions of ergodic (non-stationary) or  AMS distributions.  For non-stationary distributions, in general, one cannot make any inference about the distribution of any initial segment given just one time series sample, which is the case in all the problems we have considered. However, we can make inference about the asymptotic means. We can thus consider the distance between the asymptotic-mean distributions. It is, in fact, the same distributional distance that we have worked with in this volume, only considered as the distance between asymptotic-mean distributions and not the process distributions themselves.  Of course, its  empirical estimates simply carry over. Note that, considered as a distance between process distributions, it is not a metric, since we can have $d(\rho_1,\rho_2)=0$ for $\rho_1,\rho_2$ that are different (but have the same asymptotic mean). With this distinction in mind, all the formulations of basic-inference, clustering and change-point problems translate to this this more general setting, with ``ergodic'' substituted for 
``stationary ergodic'' and ``AMS'' for ``stationary,'' and the proofs carry over intact.

\section{Piece-wise stationary processes}
When dealing with change-point problems (Section~\ref{s:chp}), we have defined a set of process distributions that can be seen as a generalization of stationary process distributions: piece-wise stationary processes. These are constructed by defining a sequence of integer-valued {\em change points}, such as between each two consecutive change points the distribution is stationary (or stationary ergodic). 

This kind of construction has been widely studied for more restrictive sets of processes, and  mainly for i.i.d.\ processes, resulting in piece-wise i.i.d.\ models; see, for example  \cite{Willems:96,Gyorgy:12} and references.

For the stationary ergodic case, we have seen that meaningful inference is possible for finitely many change points and linear-sized (in the total sample size $n$) segments between change points. While, constrained by the nature of the change-points problems we have considered, we have only dealt with fixed sample size and offline formulations, the distributions can be defined in a similar fashion on infinite sequences. 
A piece-wise stationary  distribution is thus identified with a sequence of stationary distributions and a sequence of change points. A number of inference problems can be formulated about these processes, including versions of the clustering and hypotheses-testing problems considered in this volume. Offline clustering and identity testing appear to be the first interesting problems to explore in this regard.
\section{Beyond time series}
\subsection{Processes over multiple dimensions}
 Time series, or discrete-time process distributions that are subject of this volume, can be seen as  discrete-coordinate stochastic processes extending to infinity in  one dimension. %
One can also consider discrete-coordinate {\em multi-dimensional} stochastic processes. The concept of stationarity and ergodicity can be defined similarly to the single-dimensional case. Thus, for a dimension $d\in\N$,
one can consider a process $(X_u)$    indexed by $u\in \N^d$, over the space
$((A^\infty)^d,\Omega)$ where $\Omega$ is the Borel sigma-algebra. Such processes are simply 
probability measures over  $((A^\infty)^d,\Omega)$.  Stationarity can be defined using shifts $T_i$ along each coordinate $i\in\{1..d\}$.
A process measure $\rho$  is called stationary if it is preserved under   shifts, that is $\rho(X_{[0,v)}\in B)=\rho(T_uX_{[0,v)}\in B)$ 
for all $u,v\in V$ and all Borel $B$.
Ergodic theorems can be established for such processes, see, for example, \cite{Krengel:85}. 
This is all one needs to use empirical estimates of the distributional distance, and thus formulate and solve basic-inference as well as clustering problems, similar to how it is done in Sections~\ref{s:three}, \ref{s:clusta}. The construct of the distributional distance appears to be  general enough even for some results on hypothesis testing of Chapter~\ref{ch:ht} to be generalizable to this setting.

Change-point problems morph into something much more complex, as change points become change boundaries. It  thus appears interesting to explore what kind of change-point-like problems admit solutions in this more general setting.
\subsection{Infinite random graphs}
Another way to generalize time series is to consider infinite random graphs. The necessary probability-theoretic foundations have been laid out in 
 \cite{Aldous:07,lyons2016probability}, while the work \cite{Benjamini:12} uses these to introduce the  notions and establish some basic facts of the ergodic theory on these spaces. 
It turns out that the distributional distance is a general enough construction to be ported directly to this more general case, and some of the results of this volume, including Theorem~\ref{th:asym}, can be generalized with little extra work. This is done in the work \cite{Ryabko:17gratest}, which also outlines a number of interesting research directions that emerge in this area.

  \bibliographystyle{plain}

\end{document}